\documentclass[11pt, svgnames]{amsart}
\usepackage[T1]{fontenc}

\usepackage[shortlabels]{enumitem}
\usepackage{multirow}

\usepackage{xcolor}
\usepackage{chngcntr}
\counterwithin{equation}{section}
\counterwithin{figure}{section}
\usepackage{enumitem}
\usepackage[export]{adjustbox}
\usepackage{comment}
\usepackage{import}
\usepackage{xifthen}
\usepackage{appendix}
\usepackage{etoolbox}

\appendixtitleon

\usepackage[margin=2cm]{geometry}

\usepackage{amsfonts}
\usepackage{amsmath}
\usepackage{amsthm}
\usepackage{mathtools}
\usepackage{amssymb}
\usepackage{microtype}
\usepackage{amsbsy}
\usepackage{nicefrac}
\usepackage{tensor}
\usepackage{bm}
\usepackage{bbm}
\usepackage{dsfont}
\usepackage{cancel}
\usepackage{empheq}
\usepackage{braket}
\usepackage{stmaryrd}

\DeclareMathOperator{\End}{End}
\DeclareMathOperator{\Hom}{Hom}
\DeclareMathOperator{\spn}{span}

\DeclareMathOperator{\ad}{ad}

\DeclareMathOperator{\im}{Im}
\DeclareMathOperator{\Vir}{Vir}
\DeclareMathOperator{\Inv}{Inv}

\DeclareMathOperator{\Res}{Res}
\DeclareSymbolFont{bbold}{U}{bbold}{m}{n}
\DeclareSymbolFontAlphabet{\mathbbold}{bbold}

\makeatletter
\renewcommand*\env@matrix[1][\arraystretch]{%
  \edef\arraystretch{#1}%
  \hskip -\arraycolsep
  \let\@ifnextchar\new@ifnextchar
  \array{*\c@MaxMatrixCols c}}
\makeatother

\makeatletter
\newcommand{\medoplus}{\mathbin{\mathpalette\make@med\oplus}}
\newcommand{\medotimes}{\mathbin{\mathpalette\make@med\otimes}}
\newcommand{\make@med}[2]{%
  \vcenter{\hbox{%
    \scalebox{1.5}{$\m@th#1#2$}%
  }}%
}
\makeatother

\newcommand{\dotr}{\mbox{$\boldsymbol{\cdot}$}}
\newcommand{\g}{\ensuremath{\mathfrak{g}}}

\newcommand{\fa}{\mathfrak{a}}

\newcommand{\M}{\mathfrak{M}}
\newcommand{\B}{\mathfrak{B}}

\newcommand{\ZZ}{\mathbb{Z}}
\newcommand{\CC}{\mathbb{C}}
\newcommand{\NN}{\mathbb{N}}
\newcommand{\PP}{\mathbb{P}}

\newcommand{\W}{\mathcal{W}}

\DeclareMathOperator{\id}{id}

\DeclareMathOperator{\Der}{Der}
\DeclareMathOperator{\Inn}{Inn}

\DeclareMathOperator{\Aut}{Aut}
\DeclareMathOperator{\InnAut}{InnAut}

\newcommand{\del}{\partial}
\newcommand{\diff}[1]{\frac{d}{d#1}}

\newcommand{\diffeval}[2]{\left.\frac{d}{d#1}\right|_{#1=#2}}

\makeatletter
    \newcommand{\extp}{\@ifnextchar^\@extp{\@extp^{\,}}}
    \def\@extp^#1{\mathop{\bigwedge\nolimits^{\!#1}}}
\makeatother

\DeclareMathOperator{\HH}{H}
\DeclareMathOperator{\HL}{HL}

\DeclareMathOperator{\Ab}{Ab}
\DeclareMathOperator{\Mix}{Mix}
\DeclareMathOperator{\ZMix}{ZMix}

\DeclareMathOperator{\BMix}{BMix}

\newcommand{\nonzero}{\setminus \{0\}}

\newcommand{\ang}[1]{\langle #1 \rangle}

\newcommand{\lsemidirect}[1]{{\ltimes}\!_{#1} \mkern2mu}

\newcommand\restr[2]{{
  \left.\kern-\nulldelimiterspace % automatically resize the bar with \right
  #1 % the function
  \vphantom{\big|}
  \right|_{#2}
  }}

\newcommand{\case}[2]{\vspace{2mm} \textbf{Case #1:} #2 \vspace{2mm}}

\newtheorem{proposition}{Proposition}[section]
\newtheorem{lemma}[proposition]{Lemma}
\newtheorem{theorem}[proposition]{Theorem}
\newtheorem{corollary}[proposition]{Corollary}

\theoremstyle{definition}
\newtheorem{defn}[proposition]{Definition}

\newtheorem{Notation}[proposition]{Notation}
\newtheorem{remark}[proposition]{Remark}

\newcommand{\etalchar}[1]{$^{#1}$}

\title[Semi-direct sums of the Witt algebra with its intermediate series modules]{Central extensions, derivations, and automorphisms of semi-direct sums of the Witt algebra with its intermediate series modules}
\author{Lucas Buzaglo}
\author{Girish S Vishwa}
\date{}

\usepackage{graphicx}
\usepackage[font=small, labelfont=bf, skip=5pt]{caption}
\usepackage{subcaption}
\usepackage{wrapfig}
\usepackage{float}
\usepackage{tikz}
\usetikzlibrary{decorations.markings}
\usetikzlibrary{cd}

\usepackage{url}
\usepackage{hyperref}
\hypersetup{
    colorlinks=true,
    linkcolor=NavyBlue,
    filecolor=magenta,      
    urlcolor=Navy,
    citecolor=DarkSlateBlue
}
\makeatletter
\g@addto@macro{\UrlBreaks}{\UrlOrds}   %% Command to break long URLs $$
\makeatother
\usepackage{cleveref}

\usepackage{array}
\usepackage{tabularx}
\usepackage{multirow}
\renewcommand{\arraystretch}{1.3}
\newcommand{\PreserveBackslash}[1]{\let\temp=\\#1\let\\=\temp}
\newcolumntype{C}[1]{>{\PreserveBackslash\centering}p{#1}}
\newcolumntype{R}[1]{>{\PreserveBackslash\raggedleft}p{#1}}
\newcolumntype{L}[1]{>{\PreserveBackslash\raggedright}p{#1}}

\usepackage[symbol]{footmisc}


\address[Buzaglo]{Department of Mathematics, UC San Diego, La Jolla, CA 92093-0112, USA}
\email{\href{mailto:lbuzaglo@ucsd.edu}{lbuzaglo@ucsd.edu}, ORCID: \href{https://orcid.org/0000-0002-7662-1802}{0000-0002-7662-1802}}

\address[Vishwa]{Maxwell Institute and School of Mathematics, The University
  of Edinburgh, James Clerk Maxwell Building, Peter Guthrie Tait Road, Edinburgh EH9 3FD, Scotland, United Kingdom}
\email{\href{mailto:G.S.Vishwa@sms.ed.ac.uk}{G.S.Vishwa@sms.ed.ac.uk}, ORCID: \href{https://orcid.org/0000-0001-5867-7207}{0000-0001-5867-7207}}

\begin{document}

\begin{abstract}
    Lie algebras formed via semi-direct sums of the Witt algebra $\W = \Der(\CC[t,t^{-1}])$ and its modules have become increasingly prominent in both physics and mathematics in recent years. In this paper, we complete the study of (Leibniz) central extensions, derivations and automorphisms of the Lie algebras formed from the semi-direct sum of the Witt algebra and its indecomposable intermediate series modules (that is, graded modules with one-dimensional graded components). Our techniques exploit the internal grading of the Witt algebra, which can be applied to a wider class of graded Lie algebras.
\end{abstract}

\maketitle
\tableofcontents

\section{Introduction}\label{sec:intro}

Throughout, we work over the field of complex numbers $\CC$. All vector spaces are $\CC$-vector spaces.

The Witt algebra $\W \coloneqq \Der(\CC[t,t^{-1}]) = \CC[t,t^{-1}]\diff{t}$ of vector fields on $\CC^\times$, and its universal central extension, the Virasoro algebra, are ubiquitous in both mathematics and theoretical physics. In this paper, we study Lie algebras obtained from taking semi-direct sums of the Witt algebra with its indecomposable intermediate series modules (see Definition \ref{def: intermediate series module}). Intermediate series modules of the Witt algebra arise naturally, since they are precisely those modules which have the same Hilbert series as the Witt algebra. In some sense, they are the simplest infinite-dimensional representations of $\W$. Defining an abelian Lie algebra structure on these modules, one can then take the semi-direct sum of the Witt algebra with its intermediate series modules to obtain new families of Lie algebras.

Indecomposable intermediate series modules of the Witt algebra were classified by Kaplansky and Santharoubane \cite{KaplanskySantharoubane} into are three families -- one parametrised by $\CC^2$ and the other two parametrised by $\PP^1$. The modules parametrised by $\CC^2$ are the well-known \emph{tensor density modules} $I(a,b)$, where $a,b \in \CC$. We denote the other two families by $A(\lambda)$ and $B(\lambda)$, where $\lambda \in \PP^1$ (see \eqref{eq:defining modules} for the definitions of the modules). These correspond to blowing up $\CC^2$ at the points $(0,0)$ and $(0,1)$. The modules $I(0,0)$ and $I(0,1)$ are the only reducible tensor density modules. Indeed, we have $I(0,0) \cong B(0)$ and $I(0,1) \cong A(0)$. Thus, the indecomposable intermediate series modules of the Witt algebra are parametrised by the plane blown up at two points. In some sense, the $A(\lambda)$s and $B(\lambda)$s are infinitely close to $I(0,1)$ and $I(0,0)$ respectively, in the moduli space of intermediate series modules of $\W$. The resulting semi-direct sums are denoted $\W(a,b) \coloneqq \W \ltimes I(a,b)$ and $\W_X(\lambda) \coloneqq \W \ltimes X(\lambda)$, where $X = A$ or $B$.

The only reducible tensor density modules $I(0,0)$ and $I(0,1)$ are intimately linked to the irreducible bounded modules of the Witt algebra, where we say that a $\ZZ$-graded module is \emph{bounded} if its homogeneous components are all finite-dimensional. The unique nontrivial proper submodule of $I(0,0) = \CC[t,t^{-1}]$ is $\CC$. The quotient $\widetilde{I} \coloneqq I(0,0)/\CC$ is irreducible, self-dual, and isomorphic to the nontrivial proper submodule $\widetilde{J} \coloneqq \spn\{t^n \ dt \mid n \in \ZZ \setminus \{-1\}\}$ of $I(0,1) = \CC[t,t^{-1}]dt$. It was shown independently by Mathieu \cite{Mathieu}, and Martin and Piard \cite{MartinPiard} that $\CC$, $\widetilde{I}$ and $\{I(a,b)\ | \ (a,b)\neq(0,0),(0,1)\}$ are the only irreducible bounded $\W$-modules.

In \cite{GaoJiangPei}, the authors computed central extensions, derivations, and automorphisms of the Lie algebras $\W(a,b)$. With such a strong interest in the properties of semi-direct sums of the Witt algebra and its tensor density modules, it is only natural to consider the algebras formed from the semi-direct sum with its two other families of indecomposable intermediate series modules. In this paper, we commence the study of the resulting two families of Lie algebras by computing their central extensions, derivations, and automorphisms. We further compute the central extensions and derivations of the Lie algebras $\widetilde{\W} \coloneqq \W \ltimes \widetilde{I}$, thereby completing the computation of these properties for Lie algebras formed from the semi-direct sum of the Witt algebra with its indecomposable intermediate series modules and with its irreducible bounded modules.

Our strategy to compute cohomology heavily exploits the \emph{internally graded} structure of $\W_X(\lambda)$ (see Definition \ref{def:internally graded}), which means that the cohomology of $\W_X(\lambda)$ is entirely supported in degree zero. In other words, the degree zero cohomology of $\W_X(\lambda)$ is isomorphic to its full cohomology. This significantly simplifies our computations, since we only have to consider cocycles of degree zero. These techniques can be applied in many more cases, since most of the representations of the Witt algebra that one might be interested in are internally graded. For example, our programme can be readily extended to semi-direct sums of the Virasoro algebra with its length 3 uniserial modules which were classified by Martin and Piard in \cite{MartinPiard2}.

For the interested reader, we note that special cases of $\W(a,b)$ have been of interest in theoretical physics (see \cite{BondivdBurgMetzner, BrownHenneaux, FarahmandSafariSheikhJabbari, GPSJTZ, BBBM, FigueroaVishwa} and references therein). Larger semi-direct sums with the Witt algebra and multiple copies of its intermediate series modules are also gaining interest \cite{GaoLiuPei, BatlleCarlesCampelloGomis, BatlleFigueroaGomisVishwa}, owing largely to their appearances in modern non-Loretnzian physics \cite{BagchiGopakumar, Aizawa}.

We present the organisation and main results of our paper. After introducing some preliminaries in Section \ref{sec:preliminaries}, we compute (Leibniz) central extensions of $\W_x(\lambda)$ in Sections \ref{sec:central extensions}--\ref{sec:Leibniz}.

\begin{theorem}[Theorems \ref{thm:main} and \ref{thm: second Leibniz cohomology}]\label{thm:intro central extensions}
    Let $\lambda \in \PP^1$. Then the spaces of (Lie) central extensions of $\W_A(\lambda)$ and $\W_B(\lambda)$ are given by
    \begin{align*}
        \HH^2(\W_A(\lambda)) &=
        \begin{cases}
            \CC \overline{\Omega}_{\Vir} \oplus \CC \overline{\Omega}_0^A \oplus \CC \overline{\Omega}_{\Mix}^A, &\text{if } \lambda=0.\\
            \CC \overline{\Omega}_{\Vir} \oplus \CC \overline{\Omega}_{\Mix}^A, &\text{if } \lambda \neq 0.
        \end{cases} \\
        \HH^2(\W_B(\lambda)) &= \CC \overline{\Omega}_{\Vir} \oplus \CC \overline{\Omega}_{\Ab}^B \oplus \CC \overline{\Omega}_{\Mix}^B,
    \end{align*}
    where the cocycles are defined in Theorem \ref{thm:main} and \eqref{eq:Virasoro cocycle}.

    Furthermore, $\W_A(\lambda)$ has one extra nontrivial Leibniz central extension (up to multiplication by a scalar) given by the Leibniz 2-cocycle induced by the symmetric, invariant bilinear form
    $$\theta_A(A_n,A_m) = \begin{cases}
        1, &\text{if } n = m = 0, \\
        0, &\text{otherwise},
    \end{cases}$$
    while $\W_B(\lambda)$ does not have any Leibniz central extensions which are not already extensions of Lie algebras. In other words,
    \begin{align*}
        \HL^2(\W_A(\lambda)) &\cong \HH^2(\W_A(\lambda))) \oplus \CC \overline{\theta}_A, \\
        \HL^2(\W_B(\lambda)) &\cong \HH^2(\W_B(\lambda)),
    \end{align*}
    where $\overline{\theta}_A$ is the image of $\theta_A$ in $\HL^2(\W_A(\lambda))$.
\end{theorem}

We then move on to computing automorphisms of $\W_A(\lambda)$ and $\W_B(\lambda)$ in Section \ref{sec:automorphisms}.

\begin{theorem}[Theorem \ref{thm:automorphisms}]\label{thm:intro automorphisms}
    Let $X \in \{A,B\}$ and $\lambda \in \PP^1$. Then
    $$\Aut(\W_X(\lambda)) \cong \begin{cases}
        \CC^\infty \rtimes (\ZZ/2\ZZ \ltimes (\CC \rtimes (\CC^\times \times \CC^\times))), &\text{if } X = A \text{ and } \lambda \in \{0,-1\}, \\
        \CC^\infty_0 \rtimes (\ZZ/2\ZZ \ltimes (\CC^2 \rtimes (\CC^\times \times \CC^\times))), &\text{if } X = B \text{ and } \lambda = 0, \\
        \CC^\infty_0 \rtimes (\ZZ/2\ZZ \ltimes (\CC \rtimes (\CC^\times \times \CC^\times))), &\text{if } X = B \text{ and } \lambda = -1, \\
        \CC^\infty \rtimes (\CC \rtimes (\CC^\times \times \CC^\times)), &\text{if } X = A \text{ and } \lambda \notin \{0,-1\}, \\
        \CC^\infty_0 \rtimes (\CC \rtimes (\CC^\times \times \CC^\times)), &\text{if } X = B \text{ and } \lambda \notin \{0,-1\},
    \end{cases}$$
    where $\CC^\infty$ and $\CC_0^\infty$ are defined in \eqref{eq:Cinfty} and \eqref{eq:C0infty}, and the explicit group structures are given in Theorem \ref{thm:automorphisms}.
\end{theorem}

In Section \ref{sec:derivations}, we compute derivations of $\W_A(\lambda)$ and $\W_B(\lambda)$. Our computation of automorphisms allows us to define derivations of $\W_A(\lambda)$ and $\W_B(\lambda)$ by differentiating one-parameter subgroups of the corresponding automorphism group. All outer derivations of $\W_A(\lambda)$ and $\W_B(\lambda)$ can be obtained in this way.

\begin{theorem}[Theorem \ref{thm:derivations}]\label{thm:intro derivations}
    Let $\lambda \in \PP^1$. Then
    \begin{align*}
        \Der(\W_A(\lambda)) &= \Inn(\W_A(\lambda)) \oplus \CC d_{\Ab} \oplus \CC d_\lambda^A; \\
        \Der(\W_B(\lambda)) &= \begin{cases}
            \Inn(\W_B(\lambda)) \oplus \CC d_{\Ab} \oplus \CC d_0^B \oplus \CC \del_0^B, &\text{if } \lambda = 0, \\
            \Inn(\W_B(\lambda)) \oplus \CC d_{\Ab} \oplus \CC d_\lambda^B, &\text{if } \lambda \neq 0,
        \end{cases}
    \end{align*}
    where the derivations are defined in Notation \ref{ntt:derivations} and Lemma \ref{lem:inner derivations}. Consequently,
    $$\dim(\HH^1(\W_A(\lambda);\W_A(\lambda))) = 2, \quad \dim(\HH^1(\W_B(\lambda);\W_B(\lambda))) = \begin{cases}
        3, &\text{if } \lambda = 0, \\
        2, &\text{if } \lambda \neq 0.
    \end{cases}$$
\end{theorem}

A key technique that we use throughout is the spitting of cohomology groups of semi-direct sum Lie algebras, summarised in the following proposition:

\begin{proposition}[Proposition \ref{prop: cohomology splitting}] \label{prop:intro cohomology splitting}
    For any Lie algebra $\g$ and $\g$-module $\M$, the cohomology groups of $\g\ltimes \M$, where $\M$ is endowed with an abelian Lie bracket, split as follows:
    \begin{equation*}
        \HH^n (\g\ltimes \M) \cong \bigoplus_{k + \ell = n} \HH^k(\g;\extp^{\ell}\M^*),
    \end{equation*}
    where $\M^* = \Hom(\M, \CC)$ is the dual of $\M$.
\end{proposition}

In particular, we exploit the fact the cohomology of $\W_X(\lambda)$ with values in the trivial module can be split into sums over cohomology groups of $\W$ with values in antisymemtric powers of the \emph{restricted dual} of $X$. We need only consider the restricted dual since the modules with which we are working are $\ZZ$-graded. This allows us to use the fact that $A(\lambda)$ is the restricted dual of $B(\lambda)$ to minimise the number of computations required. For example, when computing derivations of $\W_B(\lambda)$, we encounter equations that appear when computing central extensions of $\W_A(\lambda)$, so we only have to compute solutions to the equations once. This is not surprising in light of Proposition \ref{prop:intro cohomology splitting}: setting $n=2$, we observe that one of the components of $\HH^2(\W_X(\lambda))$ is $\HH^1(\W;X(\lambda)')$, which is part of the computation of derivations of $\W_X(\lambda)$.

Finally, we end the paper by computing (Leibniz) central extensions and derivations of $\widetilde{\W}$ in Section \ref{sec:W tilde}. The techniques and proofs are almost identical to those for $\W_X(\lambda)$.

\begin{theorem}[Theorems \ref{thm:central extensions W tilde}, \ref{thm:Leibniz W tilde}, and \ref{thm:derivations W tilde}]\label{thm:intro W tilde}
    The space of central extensions of $\widetilde{\W}$ is given by
    $$\HL^2(\widetilde{\W}) \cong \HH^2(\widetilde{\W}) = \CC \overline{\Omega}_{\Vir} \oplus \CC \overline{\Omega}_1 \oplus \CC \overline{\Omega}_2 \oplus \CC \overline{\Omega}_{\Ab},$$
    where $\Omega_{\Vir}$ is the Gelfand--Fuchs cocycle in \eqref{eq:Virasoro cocycle}, and the other cocycles are defined in Theorem \ref{thm:central extensions W tilde}. Furthermore,
    $$\Der(\widetilde{\W}) = \Inn(\widetilde{\W}) \oplus \CC \delta_1 \oplus \CC \delta_2 \oplus \CC d_{\Ab},$$
    where $\delta_1$ and $\delta_2$ are defined in Section \ref{sec:preliminaries}, and $d_{\Ab}$ is defined in Theorem \ref{thm:derivations W tilde}. Consequently,
    $$\dim(\HH^2(\widetilde{\W})) = 4, \quad \dim(\HH^1(\widetilde{\W})) = 3.$$
\end{theorem}

One of the ingredients in our computations of central extension and derivations is the cohomology space $\HH^1(\W;X(\lambda))$. Although we compute this directly, it is worth noting that it could be analysed by using the long exact sequences in cohomology arising from the short exact sequences
\begin{align*}
    &0 \to \widetilde{I} \to A(\lambda) \to \CC \to 0 \\
    &0 \to \CC \to B(\lambda) \to \widetilde{I} \to 0.
\end{align*}
Together with work of Gao, Jiang, and Pei \cite{GaoJiangPei} (see Theorem \ref{thm:GJP}),
\begin{enumerate}
    \item Theorems \ref{thm:intro central extensions}, \ref{thm:intro automorphisms}, and \ref{thm:intro derivations} complete the study of derivations, central extensions and automorphisms of the Lie algebras formed from the semi-direct sum of the Witt algebra with all its indecomposable intermediate series modules, while
    \item Theorem \ref{thm:intro W tilde} completes the study of derivations and central extensions of the Lie algebras formed from the semi-direct sum of the Witt algebra with all its irreducible bounded modules.
\end{enumerate}
 We compile all these results into the following tables, where we have given simplified forms of the groups of automorphisms for conciseness and for ease of comparison with \cite[Theorem 5.2]{GaoJiangPei}.

\begin{table}[ht]
\centering
\begin{tabular}{|c|c|ccc|}
\hline
\multirow{2}{*}{\textbf{Lie algebra} $\g$}     & \multirow{2}{*}{\textbf{Value of parameters}}          & \multicolumn{3}{c|}{\textbf{Dimension of}}                                                     \\ \cline{3-5} 
                                 &                                               & \multicolumn{1}{c|}{$\HH^2(\g)$}          & \multicolumn{1}{c|}{$\HL^2(\g)$}         & $\HH^1(\g;\g)$       \\ \hline
\multirow{5}{*}{$\W(a,b)$}       & $(a,b) \in \CC^2 \setminus \{(\frac{1}{2},0), (0,1),$      & \multicolumn{1}{c|}{\multirow{2}{*}{1}} & \multicolumn{1}{c|}{\multirow{2}{*}{1}} & \multirow{2}{*}{1} \\
                                 & \multicolumn{1}{r|}{$(0,2), (0,1),(0,0)\}$}   & \multicolumn{1}{c|}{}                   & \multicolumn{1}{c|}{}                   &                    \\ \cline{2-5} 
                                 & $a = \frac{1}{2}, b = 0$                      & \multicolumn{1}{c|}{2}                  & \multicolumn{1}{c|}{2}                  & 1                  \\ \cline{2-5} 
                                 & $a = 0, b = -1$                               & \multicolumn{1}{c|}{2}                  & \multicolumn{1}{c|}{2}                  & 1                  \\ \cline{2-5} 
                                 & $a = 0, b = 2$                                & \multicolumn{1}{c|}{1}                  & \multicolumn{1}{c|}{2}                  & 2                  \\ \hline
$\W(0,1) \cong \W_A(0)$          & $a = 0, b = 1; \quad \lambda = 0$             & \multicolumn{1}{c|}{3}                  & \multicolumn{1}{c|}{4}                  & 2                  \\ \hline
$\W(0,0) \cong \W_B(0)$          & $a = 0, b = 0; \quad \lambda = 0$             & \multicolumn{1}{c|}{3}                  & \multicolumn{1}{c|}{3}                  & 3                  \\ \hline
\multirow{2}{*}{$\W_A(\lambda)$} & \multirow{2}{*}{$\lambda \in \PP^1 \nonzero$} & \multicolumn{1}{c|}{\multirow{2}{*}{2}} & \multicolumn{1}{c|}{\multirow{2}{*}{3}} & \multirow{2}{*}{2} \\
                                 &                                               & \multicolumn{1}{c|}{}                   & \multicolumn{1}{c|}{}                   &                    \\ \hline
\multirow{2}{*}{$\W_B(\lambda)$} & \multirow{2}{*}{$\lambda \in \PP^1 \nonzero$} & \multicolumn{1}{c|}{\multirow{2}{*}{3}} & \multicolumn{1}{c|}{\multirow{2}{*}{3}} & \multirow{2}{*}{2} \\
                                 &                                               & \multicolumn{1}{c|}{}                   & \multicolumn{1}{c|}{}                   &                    \\ \hline
\multirow{2}{*}{$\widetilde{\W}$} & \multirow{2}{*}{$-$} & \multicolumn{1}{c|}{\multirow{2}{*}{4}} & \multicolumn{1}{c|}{\multirow{2}{*}{4}} & \multirow{2}{*}{3} \\
                                 &                                               & \multicolumn{1}{c|}{}                   & \multicolumn{1}{c|}{}                   &                    \\ \hline
\end{tabular}
\caption{The dimensions of the cohomology spaces of $\W(a,b)$, $\W_X(\lambda)$, and $\widetilde{\W}$}
\end{table}

\begin{table}[ht]
\centering
\begin{tabular}{|c|c|c|}
\hline
\textbf{Lie algebra}             & \textbf{Value of parameters}                    & \textbf{Automorphism group}                                            \\ \hline
\multirow{2}{*}{$\W(a,b)$}       & $a \in \CC \setminus \frac{1}{2}\ZZ, b \in \CC$ & $\CC^\infty \rtimes (\CC^\times \times \CC^\times)$                    \\ \cline{2-3} 
                                 & $a \in \frac{1}{2}\ZZ, b \in \CC$               & $\CC^\infty \rtimes (\ZZ/2\ZZ \ltimes (\CC^\times \times \CC^\times))$ \\ \hline
$\W(0,1) \cong \W_A(0)$          & $a = 0, b = 1; \quad \lambda = 0$               & $\CC^\infty \rtimes (\ZZ/2\ZZ \ltimes (\CC^\times \times \CC^\times))$ \\ \hline
$\W(0,0) \cong \W_B(0)$          & $a = 0, b = 0; \quad \lambda = 0$               & $\CC^\infty \rtimes (\ZZ/2\ZZ \ltimes (\CC^\times \times \CC^\times))$ \\ \hline
\multirow{2}{*}{$\W_A(\lambda)$} & $\lambda \in \PP^1 \setminus \{0,-1\}$          & $\CC^\infty \rtimes (\CC^\times \times \CC^\times)$                    \\ \cline{2-3} 
                                 & $\lambda = -1$                                  & $\CC^\infty \rtimes (\ZZ/2\ZZ \ltimes (\CC^\times \times \CC^\times))$ \\ \hline
\multirow{2}{*}{$\W_B(\lambda)$} & $\lambda \in \PP^1 \setminus \{0,-1\}$          & $\CC^\infty \rtimes (\CC^\times \times \CC^\times)$                    \\ \cline{2-3} 
                                 & $\lambda = -1$                                  & $\CC^\infty \rtimes (\ZZ/2\ZZ \ltimes (\CC^\times \times \CC^\times))$ \\ \hline
\end{tabular}
\caption{Structures (simplified) of the automorphism groups of $\W(a,b)$ and $\W_X(\lambda)$}
\end{table}

\newpage

\noindent \textbf{Acknowledgements:} This work was completed as part of the authors' PhD studies at the University of Edinburgh. The authors would like to thank Susan Sierra for extensive comments on an earlier version of the paper, Martin Schlichenmaier for interesting discussions, and Friedrich Wagemann for pointing out that the reference \cite{Fuchs} contains useful results concerning internally graded Lie algebras. The second author would additionally like to thank Jos\'{e} Figueroa-O'Farrill for insights into the various physical contexts in which special cases of these Lie algebras appear in physics. The authors would also like to thank the anonymous referee for their detailed comments, and in particular for suggesting Proposition \ref{prop: cohomology splitting} to us. The second author is supported by the Science and Technologies Facilities Council [grant number 2615874].

\section{Preliminaries}\label{sec:preliminaries}

\subsection{Lie algebra cohomology}

In this section, we provide a brief review of Lie algebra cohomology. We refer the reader to \cite[Chapter 1]{Fuchs} for a more thorough introduction to the subject. Let $\g$ be a Lie algebra and $\rho \colon \g \rightarrow \End\M$ be a $\g$-module, both possibly infinite-dimensional.

\begin{defn}
    Let $p \in \NN$. The \emph{$p$-cochains of $\g$ with values in $\M$}, denoted $C^p(\g; \M)$, are the $\M$-valued alternating $p$-linear functionals on $\g$. Under the identification
    $$C^p(\g;\M) = \Hom\left(\extp^p \g, \M\right) \cong \extp^p \g^* \medotimes \M,$$ the space of cochains is a $\g$-module by natural extension of the coadjoint action of $\g$ on $\g^*$. For completeness, also note that $C^p(\g;\M) = 0$ if $p < 0$ and $C^0(\g;\M)\cong \M$.
    
    The \emph{Lie algebra cohomology of $\g$ with values in $\M$}, denoted $\HH^*(\g;\M)$, is the cohomology of the cochain complex
    $$\cdots \to C^{p-1}(\g;\M) \to C^p (\g;\M) \to C^{p+1}(\g;\M) \to \cdots$$
    with differential
    \begin{equation} \label{eq: LA cohomology d}
    \begin{split}
        d\varphi(x_1, x_2,\dots, x_{p+1})
            &=\sum_{i=1}^{p+1} (-1)^{i+1}\rho(x_i) \varphi (x_1, \dots, \hat{x}_i, x_{i+1}, \dots, x_{p+1}) \\
            &+\sum_{1\leq i< j \leq p+1} (-1)^{i+j} \varphi ([x_i, x_j], x_1, \dots, \hat{x}_i, \dots, \hat{x}_j, \dots, x_{p+1}),       
    \end{split}
    \end{equation}
    where $\hat{\ }$ denotes omission, for all $\varphi \in C^p(\g;\M)$ and $x_1,x_2,\dots,x_{p+1} \in \g$.

    We denote the space of $p$-cocycles (respectively, $p$-coboundaries) of $\g$ with values in $\M$ by $Z^p(\g;\M)$ (respectively, $B^p(\g;\M)$), and so
    $$\HH^p(\g;\M) \coloneqq Z^p(\g;\M)/B^p(\g;\M).$$
\end{defn}

If $\M = \CC$ regarded as the trivial $\g$-module, the first term in \eqref{eq: LA cohomology d} vanishes. For convenience, we write $C^p(\g) \coloneqq C^p(\g;\CC)$, $Z^p(\g) \coloneqq Z^p(\g;\CC)$, $B^p(\g) \coloneqq B^p(\g;\CC)$, and $\HH^p(\g) \coloneqq \HH^p(\g;\CC)$.

We briefly discuss some interpretations of the low-dimensional cohomology spaces of a Lie algebra. Firstly, it is easy to see from the definition of Lie algebra cohomology that $\HH^1(\g) \cong \g/[\g,\g]$ is the abelianisation of the Lie algebra $\g$.

The cohomology space $\HH^1(\g;\M)$ of $\g$ with coefficients in a $\g$-module $\M$ also has an interpretation: it is the space of \emph{outer derivations of $\g$ into $\M$}: we have
$$\HH^1(\g;\M) \cong \Der(\g,\M)/\Inn(\g,\M).$$
We recall the notion of a derivation of a Lie algebra below.

\begin{defn}\label{def:derivation}
    A \emph{derivation} of a Lie algebra $\g$ with values in a $\g$-module $\M$ is a linear map $d \colon \g \to M$ such that
    $$d([x,y]) = x \cdot d(y) - y \cdot d(x)$$
    for $x,y \in \g$. We write $\Der(\g,\M)$ for the set of derivations $\g$ with values in $\M$, and simply write $\Der(\g)$ instead of $\Der(\g,\g)$.
    
    A derivation $d \in \Der(\g,\M)$ is \emph{inner} if there exists $m \in \M$ such that $d(x) = x \cdot m$. We write $\Inn(\g,\M)$ for the set of inner derivations of $\g$ with values in $\M$. As before, we write $\Inn(\g)$ instead of $\Inn(\g,\g)$.
\end{defn}

The second cohomology $\HH^2(\g)$ classifies one-dimensional \emph{central extensions} of $\g$, as we now explain.

\begin{defn}
    Let $0 \to \mathfrak{a} \to \widehat{\g} \to \g \to 0$ be a short exact sequence of Lie algebras. We say that $\widehat{\g}$ is a \emph{central extension} of $\g$ (by $\mathfrak{a}$) if $\mathfrak{a}$ is contained in the center of $\widehat{\g}$.
\end{defn}

It is easy to check that every one-dimensional central extension corresponds to a 2-cocycle (with values in $\CC$) and vice versa \cite[Chapter 4]{Schottenloher}. Central extensions corresponding to 2-coboundaries are always trivial (i.e. give rise to split exact sequences), and $\omega_1, \omega_2 \in Z^2(\g; \mathfrak{a})$ generate equivalent central extensions if and only if $\omega_1 - \omega_2 \in B^2(\g; \mathfrak{a})$. Thus, the second cohomology group $\HH^2(\g; \mathfrak{a})$ is in one-to-one correspondence with the set of equivalence classes of central extensions by $\mathfrak{a}$.
Given that any $\mathfrak{a}$ can be written as a direct sum of one-dimensional Lie algebras, it is natural to ask if studying the one-dimensional central extensions of $\g$ informs us about the \emph{universal} central extension of $\g$, defined below.

\begin{defn} \label{def: universal central extension}
    A central extension $0 \to \fa \to \widehat{\g} \to \g \to 0$ of a Lie algebra $\g$ is \emph{universal} if for every other central extension $0 \to \mathfrak{b} \to \widetilde{\g} \to \g \to 0$, there exist unique maps $\widehat{\g} \to \widetilde{\g}$ and $\mathfrak{a} \to \mathfrak{b}$ such that the following diagram commutes:
    \begin{center}
        \begin{tikzcd}
            0 \arrow[r] & \mathfrak{a} \arrow[r] \arrow[d]   & \widehat{\g} \arrow[r] \arrow[d]   & \g \arrow[r] \arrow[d, "\id"] & 0 \\
            0 \arrow[r] & \mathfrak{b} \arrow[r] & \widetilde{\g} \arrow[r] & \g \arrow[r, rightarrow]      & 0
        \end{tikzcd}.
    \end{center}
\end{defn}

It is well-known that a Lie algebra $\g$ has a universal central extension if and only if $\g$ is perfect \cite[Theorem 7.9.2]{Weibel}.

The following well-known result makes explicit the relationship between $\HH^2(\g)$ (which classifies \emph{one-dimensional} central extensions of $\g$) and universal central extensions of a perfect Lie algebra.
\begin{proposition}
    Let $\g$ be a perfect Lie algebra and let $n = \dim \HH^2(\g)$. Let $\Omega_1,\ldots,\Omega_n \in Z^2(\g)$ be 2-cocycles whose images in cohomology form a basis for $\HH^2(\g)$. Define
    $$\widehat{\g} = \g \oplus \CC c_1 \oplus \ldots \oplus \CC c_n,$$
    where the $c_i$ are central, and the other brackets are given by
    \begin{equation}\label{eq: universal extension}
        [x,y]_{\widehat{\g}} = [x,y]_\g + \sum_{i=1}^n \Omega_i(x,y)c_i,
    \end{equation}
    where $x,y \in \g$. Then $\widehat{\g}$ is the universal central extension of $\g$.
\end{proposition}
\begin{proof}
    This follows immediately from \cite[Theorem 7.9.2]{Weibel} and the duality between $H_2(\g)$ and $H^2(\g)$ (see \cite[p.16]{Fuchs}).
\end{proof}

\subsection{Intermediate series modules of the Witt algebra}

We now introduce the Lie algebras of interest in this paper. 

\begin{defn}
\label{def: Witt algebra}
Let $\W$ be the \emph{Witt algebra}, with basis $\{L_n\}_{n \in \ZZ}$ and Lie bracket
$$[L_n,L_m] = (m - n)L_{n+m}.$$
The Witt algebra has a natural $\ZZ$-grading, with $\deg(L_n) = n$. It is the Lie algebra of derivations of the ring of Laurent polynomials. Under the identification $\W = \Der(\CC[t,t^{-1}]) = \CC[t,t^{-1}]\diff{t}$, we have $L_n = t^{n+1}\diff{t}$.
\end{defn}

In this paper, we study semi-direct sums of $\W$ with its intermediate series modules, defined below.

\begin{defn} \label{def: intermediate series module}
    An \emph{intermediate series module} of a $\ZZ$-graded Lie algebra $\g = \bigoplus_{n \in \ZZ} \g_n$ is a $\ZZ$-graded $\g$-module $\M = \bigoplus_{n \in \ZZ} \M_n$ such that $\dim \M_n = 1$ for all $n \in \ZZ$. 
\end{defn}

Indecomposable intermediate series modules of $\W$ were classified in \cite{KaplanskySantharoubane}: they consist of a family $\{I(a,b) \mid a,b \in \CC\}$, whose elements are known as \emph{tensor density modules}, and two families $\{A(\lambda) \mid \lambda \in \PP^1\}$ and $\{B(\lambda) \mid \lambda \in \PP^1\}$. An analogous classification of the intermediate series modules of \emph{generalised} (or \emph{higher rank}) \emph{Witt algebras} was obtained in \cite{Su}.

In this paper, we identify $\PP^1 = \CC \cup \{\infty\}$, where, for $a,b \in \CC$, the point $[a:b] \in \PP^1$ corresponds to $\frac{a}{b}$ if $b \neq 0$, or $\infty$ if $b = 0$. Fixing $a,b,\lambda \in \CC$ and letting $\{I_n\}_{n \in \ZZ}, \{A_n\}_{n \in \ZZ}, \{B_n\}_{n \in \ZZ}$ be bases for $I(a,b), A(\lambda), B(\lambda)$, respectively, the action of $\W$ on these modules is given by
\begin{equation}\label{eq:defining modules}
    \begin{gathered}
        L_n \cdot I_m = (a + bn +m)I_{n+m}, \\
        L_n \cdot A_m = (n + m + n(n + 1)\lambda \delta_m^0)A_{n+m}, \\
        L_n \cdot B_m = (m - n(n + 1)\lambda \delta_{n+m}^0)B_{n+m},
    \end{gathered}
\end{equation}
where $\delta_n^0$ is the Kronecker delta symbol. When $\lambda = \infty$, the $n(n + 1)\lambda$ term becomes $n^2$, in other words, the action of $\W$ on $A(\infty)$ and $B(\infty)$ is given by
\begin{align*}
    L_n \cdot A_m &= (n + m + n^2 \delta_m^0)A_{n+m}, \\
    L_n \cdot B_m &= (m - n^2 \delta_{n+m}^0)B_{n+m}.
\end{align*}
Note that $A(0) \cong I(0,1)$ and $B(0) \cong I(0,0)$, while $I(a,b) \cong I(a + k,b)$ for all $k \in \ZZ$ and $I(\frac{1}{2},0) \cong I(\frac{1}{2},1)$ \cite[Lemma 2.1]{GaoJiangPei}. For this reason, following the conventions of \cite{GaoJiangPei}, we always assume that $a = 0$ in $I(a,b)$ if $a \in \ZZ$.

Since the modules $A(\lambda)$ and $B(\lambda)$ seem much more mysterious than the tensor density modules, we now spend some time explaining how they arise, and describe their relationship to the other intermediate series modules. The tensor density modules can be viewed as
$$I(a,b) = t^{a - b}\CC[t,t^{-1}]dt^b.$$
Under this identification, the action of $\W$ on $I(a,b)$ is the following:
$$f\del \cdot (g \ dt^b) = (fg' + b f'g)dt^b.$$
Thus, $I(0,0)$ is $\CC[t,t^{-1}]$ with its usual action of $\W$ by derivations, $I(0,-1)$ is $\CC[t,t^{-1}]dt^{-1} \cong \W$ with the adjoint action of $\W$, and $I(0,1)$ is $\CC[t,t^{-1}]dt$, the space of 1-forms.

Recall that the tensor density modules $I(a,b)$ are irreducible for all $a,b \in \CC$, except for the functions $I(0,0)$ and 1-forms $I(0,1)$. These two modules are mutually dual with unique nontrivial proper submodules $\CC$ and $\widetilde{J} \coloneqq \spn\{t^n \ dt \mid n \in \ZZ \setminus \{-1\}\}$ respectively. Note that when we say ``dual'', we mean that $I(0,0)$ is the \emph{restricted dual} of $I(0,1)$, and vice-versa. The restricted dual of a graded $\W$-module $\M = \bigoplus_{n \in \ZZ} \M_n$ with $\dim \M_n < \infty$ for all $n \in \ZZ$ is the graded $\W$-module $\M' = \bigoplus_{n \in \ZZ}\M'_n$ with $\M'_n = \Hom_\CC(\M_{-n},\CC)$. Furthermore, the quotient $\widetilde{I} \coloneqq I(0,0)/\CC$ is irreducible, self-dual, and isomorphic to the nontrivial proper submodule $\widetilde{J}$.

This relationship between $I(0,0)$ and $I(0,1)$ can be explained by introducing the following natural $\W$-maps. First, we have the \emph{de Rham differential}
\begin{align*}
    d \colon I(a,0) &\to I(a,1) \\
    f &\mapsto df = f' \ dt.
\end{align*}
Using this map, we see that the action of $\W$ on $I(0,1)$ can be described as follows:
$$f\del \cdot (g \ dt) = (fg' + f'g)dt = (fg)' \ dt = d(fg).$$
When $a = 0$, it is clear that the kernel of $d$ consists of the constant functions $\CC$, and that the image of $d$ is $\widetilde{J}$, which yields the isomorphism $\widetilde{I} \cong \widetilde{J}$.

Another natural $\W$-map is the \emph{residue map}, defined by
\begin{align*}
    \Res \colon I(0,1) &\to \CC \\
    f \ dt &\mapsto \frac{1}{2\pi i}\oint_0 f \ dt. 
\end{align*}
More concretely, the residue map extracts the coefficient of $t^{-1}$ in the Laurent polynomial $f$. It is clear that $\ker(\Res) = \widetilde{J}$.

Additionally, we have the \emph{multiplication map}
\begin{align*}
    \cdot \ \colon I(a_1,b_1) \otimes I(a_2,b_2) &\to I(a_1 + a_2, b_1 + b_2) \\
    f \ dt^{b_1} \otimes g \ dt^{b_2} &\mapsto fg \ dt^{b_1 + b_2}.
\end{align*}
The self-duality of $\widetilde{I}$ is witnessed by the invariant skew-symmetric form
\begin{equation}\label{eq:self duality of I}
    \ang{f,g} \coloneqq \Res(f \cdot dg) = \oint_0 fg' \ dt,
\end{equation}
where an invariant form is defined below.

\begin{defn}
    Let $\g$ be a Lie algebra and $\M$ be a $\g$-module. A bilinear form $\theta \colon \M \times \M \to \CC$ is said to be \emph{invariant} if it satisfies
    $$\theta(x \cdot u,v) + \theta(u,x \cdot v) = 0$$
    for all $x \in \g$ and $u,v \in \M$.
\end{defn}

Similarly, we can see the duality between $I(0,b)$ and $I(0,1 - b)$ via the form
\begin{equation}\label{eq:B form}
    \mathcal{B}(f \ dt^b, g \ dt^{1 - b}) \coloneqq \Res(f \ dt^b \cdot g \ dt^{1 - b}) = \oint_0 fg \ dt.
\end{equation}
Finally, we have the \emph{Lie bracket}
\begin{align*}
    [-,-] \colon I(a,b_1) \otimes I(a,b_2) &\to I(a, b_1 + b_2 + 1) \\
    f \ dt^{b_1} \otimes g \ dt^{b_2} &\mapsto (-b_1fg' + b_2f'g)dt^{b_1 + b_2 + 1}.
\end{align*}
The multiplication map and the Lie bracket equip the space $\mathcal{F} \coloneqq \bigoplus_{b \in \ZZ} I(0,b)$ with a Poisson algebra structure. See \cite[Theorem 2.10]{Schlichenmaier:2012ika} for more details.

The modules $A(\lambda)$ and $B(\lambda)$ arise as follows: $\HH^1(\W; \widetilde{I})$ is two-dimensional. Noting that $\HH^1(\W; \widetilde{I}) \cong \Der(\W,\widetilde{I})/\Inn(\W,\widetilde{I})$, we can view an element of $\HH^1(\W; \widetilde{I})$ as a derivation of $\W$ with coefficients in $\widetilde{I}$. Given $\delta \in \Der(\W,\widetilde{I})$, we can form a new $\W$-module $\widetilde{I}_\delta$ as follows: as a vector space, $\widetilde{I}_\delta = \widetilde{I} \oplus \CC \delta$ with $\W$-action given by letting $\widetilde{I}$ be a submodule of $\widetilde{I}_\delta$, and defining
$$w \cdot \delta = -\delta(w),$$
where $w \in \W$. This yields a family of extensions of $\CC$ by $\widetilde{I}$, but not all of these are indecomposable. The indecomposable members of this family are the modules $A(\lambda)$, parametrised by $\PP^1$. We then get the dual family $B(\lambda)$ of indecomposable extensions of $\widetilde{I}$ by $\CC$ \cite[Remark 3.1]{SierraSpenko}.

Therefore, we may view $A(\lambda)$ and $B(\lambda)$ as the spaces of 1-forms $\CC[t,t^{-1}]dt$ and functions $\CC[t,t^{-1}]$ respectively, with deformed actions of $\W$. Under this perspective, we have $A_n = t^{n - 1} \ dt$ and $B_n = t^n$. For $\lambda \in \CC$, to define the cocycle giving rise to $A(\lambda)$ as an extension of $\CC$ by $\widetilde{I} \cong \widetilde{J}$, we introduce the following derivation
\begin{align*}
    \delta_1 \colon \W &\to \widetilde{J} \\
    f\del &\mapsto df' = f'' \ dt.
\end{align*}
We can then write the action of $\W$ on $A(\lambda) = \CC[t,t^{-1}]dt$ as
$$f\del \cdot (g \ dt) = \Big(fg' + f'g + \lambda \Res(g \ dt) f''\Big)dt = d(fg + \lambda \Res(g \ dt) f').$$
It is easy to check that the above action coincides with the definition of $A(\lambda)$ as a deformation of $I(0,1)$ given in \eqref{eq:defining modules}, provided $\lambda \neq \infty$. For the case $\lambda = \infty$, we must also use the derivation
\begin{align*}
    \delta_2 \colon \W &\to \widetilde{J} \\
    f\del &\mapsto d(t^{-1}f) = (t^{-1}f - t^{-2}f)dt.
\end{align*}
Then the action of $\W$ on $A(\infty)$ is
\begin{align*}
    f\del \cdot (g \ dt) &= \Big(fg' + f'g + \Res(g \ dt) (f'' - t^{-1}f' + t^{-2}f)\Big)dt \\
    &= d\Big(fg + \Res(g \ dt) (f' - t^{-1}f)\Big),
\end{align*}
which we get from the derivation $\delta_1 - \delta_2 \in \Der(\W,\widetilde{J})$. We can easily see that $\delta_1$ and $\delta_2$ are the compositions of the derivations
\begin{align*}
    \widetilde{\delta}_1 \colon \W &\to I(0,0), & \widetilde{\delta}_2 \colon \W &\to I(0,0) \\
    f\del &\mapsto f', & f\del &\mapsto t^{-1}f
\end{align*}
with the de Rham differential $d$. The classes of $\delta_1$ and $\delta_2$ span the space $\HH^1(\W;\widetilde{J}) \cong \HH^1(\W;\widetilde{I})$ (see Corollary \ref{cor:derivations W to I tilde}).

For $\lambda \in \CC$, the cocycle defining $B(\lambda)$ is dual to $\delta_1$, and the action of $\W$ on $B(\lambda) = \CC[t,t^{-1}]$ can be written as
$$f\del \cdot g = fg' + \lambda \ang{f',g},$$
where $\ang{-,-}$ is the skew-symmetric bilinear form defined in \eqref{eq:self duality of I}. For $\lambda = \infty$, we use the cocycle dual to $\delta_1 - \delta_2$, giving the formula
$$f\del \cdot g = fg' + \ang{f' - t^{-1}f,g}$$
for the action of $\W$ on $B(\infty)$.

By the discussion above, we see that the module $A(\lambda)$ has an irreducible submodule $\W \cdot A(\lambda) = \spn\{A_n \mid n \in \ZZ \nonzero\} \cong \widetilde{I}$ with a one-dimensional trivial quotient $A(\lambda)/(\W \cdot A(\lambda)) \cong \CC$, while $B(\lambda)$ has a one-dimensional trivial submodule $\CC B_0$ with an irreducible quotient $B(\lambda)/\CC B_0$ isomorphic to $\widetilde{I}$. This can be seen by noticing that the de Rham differential is still a $\W$-map under the deformed $\W$-actions on $\CC[t,t^{-1}]$ and $\CC[t,t^{-1}]dt$, as we prove next.

\begin{lemma}
    Let $\lambda, \mu \in \PP^1$. Then the de Rham differential $d \colon B(\lambda) \to A(\mu)$ is a homomorphism of $\W$-representations, where we view $A(\mu) = \CC[t,t^{-1}]dt$ and $B(\lambda) = \CC[t,t^{-1}]$ as vector spaces. The map $d$ induces an isomorphism $B(\lambda)/\CC B_0 \cong \W \cdot A(\mu)$.
\end{lemma}
\begin{proof}
    First, note that $d(B_n) = nA_n$ for all $n \in \ZZ$. Let $n,m \in \ZZ$ such that $n + m \neq 0$. Then
    $$\begin{gathered}
        d(L_n \cdot B_m) = md(B_{n+m}) = m(n + m)A_{n+m}, \\
        L_n \cdot d(B_m) = m L_n \cdot A_m = m(n + m + n(n + 1)\mu\delta_m^0)A_{n+m} = m(n + m)A_{n+m},
    \end{gathered}$$
    where we used that $m\delta_m^0 = 0$ in the last equality, so we see that $d(L_n \cdot B_m) = L_n \cdot d(B_m)$ whenever $n + m \neq 0$. We also have
    $$\begin{gathered}
        d(L_n \cdot B_{-n}) = -(n + n(n + 1)\lambda)d(B_0) = 0, \\
        L_n \cdot d(B_{-n}) = -n L_n \cdot A_{-n} = 0,
    \end{gathered}$$
    where we used that $L_n \cdot A_{-n} = 0$ in the last equality. Combining the above, we conclude that $d(L_n \cdot B_m) = L_n \cdot d(B_m)$ for all $n,m \in \ZZ$, so $d$ is a homomorphism of $\W$-modules, as required.

    For the final sentence, we simply note that $\ker(d) = \CC B_0$, while $\im(d) = \spn\{A_n \mid n \in \ZZ \nonzero\} = \W \cdot A(\mu)$.
\end{proof}

Semi-direct sums of $\W$ with its tensor density modules $I(a,b)$, which we denote
$$\W(a,b) \coloneqq \W \ltimes I(a,b),$$
have been studied in \cite{GaoJiangPei}. In particular, the authors computed (Leibniz) central extensions, derivations, and automorphisms of $\W(a,b)$. We summarise their results below.

\begin{theorem}[{\cite[Theorems 2.3, 3.5, 4.7, and 5.2]{GaoJiangPei}}]\label{thm:GJP}
    Let $a,b \in \CC$. The space of central extensions of $\W(a,b)$ has the following dimension:
    $$\dim(\HH^2(\W(a,b))) = \begin{cases}
        3, &\text{if } (a,b) = (0,0) \text{ or } (0,1), \\
        2, &\text{if } (a,b) = (0,-1) \text{ or } (\frac{1}{2},0), \\
        1, &\text{otherwise}.
    \end{cases}$$
    Furthermore, there are only two values of $(a,b) \in \CC^2$ for which $\W(a,b)$ has extra non-trivial Leibniz central extensions:
    $$\dim(\HL^2(\W(a,b))) = \begin{cases}
        \dim(\HH^2(\W(a,b))) + 1, &\text{if } (a,b) = (0,1) \text{ or } (a,b) = (0,2), \\
        \dim(\HH^2(\W(a,b))), &\text{otherwise}.
    \end{cases}$$
    The space of outer derivations of $\W(a,b)$ has the following dimension:
    $$\dim(\HH^1(\W(a,b);\W(a,b))) = \begin{cases}
        3, &\text{if } (a,b) = (0,0), \\
        2, &\text{if } (a,b) = (0,1) \text{ or } (0,2), \\
        1, &\text{otherwise}.
    \end{cases}$$
    Finally, the automorphism group of $\W(a,b)$ has the following structure:
    $$\Aut(\W(a,b)) \cong \begin{cases}
        \CC^\infty \rtimes (\CC^\times \times \CC^\times), &\text{if } a \notin \frac{1}{2}\ZZ, \\
        \CC^\infty \rtimes (\ZZ/2\ZZ \ltimes (\CC^\times \times \CC^\times)), &\text{otherwise}.
    \end{cases}$$
\end{theorem}

\begin{remark}\label{rem:error}
    Note that the Leibniz cohomology result for $\W(0,1)$ in Theorem \ref{thm:GJP} is different to the one in \cite[Theorem 3.5]{GaoJiangPei}. In particular, they claim that $\dim(\HL^2(\W(0,1))) = \dim(\HH^2(\W(0,1))) + 2$. This is due to an error in \cite[Lemma 3.2]{GaoJiangPei}, which we correct in Section \ref{sec:Leibniz}.
\end{remark}

We focus on semi-direct sums of $\W$ with the intermediate series modules $A(\lambda)$ and $B(\lambda)$ for $\lambda \in \PP^1$, and prove analogous results to those of Theorem \ref{thm:GJP}. We denote the resulting semi-direct sums by
$$\W_A(\lambda) \coloneqq \W \ltimes A(\lambda), \quad \W_B(\lambda) \coloneqq \W \ltimes B(\lambda),$$
spanned by $\{L_n,A_n\}_{n \in \ZZ}$ and $\{L_n, B_n\}_{n \in \ZZ}$, respectively. For $n,m \in \ZZ$, the Lie brackets are given by
\begin{equation*}
    \begin{gathered}
        [L_n,L_m] = (m - n)L_{n+m} \\
        [L_n,A_m] = (n + m + n(n + 1)\lambda \delta_m^0)A_{n+m}, \quad [L_n,B_m] = (m - n(n + 1)\lambda \delta_{n+m}^0)B_{n+m} \\
        [A_n,A_m] = 0, \quad [B_n,B_m] = 0.
    \end{gathered}
\end{equation*}
Note that $\W_B(\lambda)$ is perfect (in other words, $[\W_B(\lambda),\W_B(\lambda)] = \W_B(\lambda)$), while $\W_A(\lambda)$ is not: we have 
$$[\W_A(\lambda),\W_A(\lambda)] = \spn\{L_n,A_m \mid n \in \ZZ, m \in \ZZ \nonzero\} = \W \ltimes (\W \cdot A(\lambda)) \cong \W \ltimes \widetilde{I}.$$
When discussing $\W_A(\lambda)$ and $\W_B(\lambda)$ simultaneously, we write $\W_X(\lambda)$, where $X = A$ or $B$. In such cases, we also write $X_n$ instead of $A_n$ or $B_n$, where $n \in \ZZ$. Furthermore, we define
$$\omega_{A(\lambda)}(n,m) = n + m + n(n + 1)\lambda \delta_m^0, \quad \omega_{B(\lambda)}(n,m) = m - n(n + 1)\lambda \delta_{n+m}^0.$$
Thus, when we discuss the brackets $[L_n,A_m]$ and $[L_n,B_m]$ simultaneously, we write
$$[L_n,X_m] = \omega_{X(\lambda)}(n,m)X_{n+m}.$$

\section{Central extensions}\label{sec:central extensions}

Our first goal is to compute $\HH^2(\W_X(\lambda))$ for $\lambda \in \PP^1$, which classifies central extensions of these Lie algebras. In other words, we want to look for non-trivial 2-cocycles on $\W_X(\lambda)$, which are antisymmetric bilinear maps $\W_X(\lambda) \times \W_X(\lambda)\rightarrow \mathbb{C}$.

Since $\W_B(\lambda)$ is perfect, it follows by \cite[Theorem 7.9.2]{Weibel} that $\W_B(\lambda)$ has a universal central extension, while $\W_A(\lambda)$ does not. Regardless, we can still classify one-dimensional central extensions of $\W_A(\lambda)$ by computing $\HH^2(\W_A(\lambda))$, although there is no universal one. 

The computation of 2-cocycles is simplified by the decomposition formula given in \cite[Proposition 1]{GaoLiuPei}. We present a more general version of this result here.

\begin{proposition} \label{prop: cohomology splitting}
    Let $\M$ be a $\g$-module over some Lie algebra $\g$. The cohomology groups of $\g\ltimes \M$, where $\M$ is endowed with an abelian Lie bracket, split as follows:
    \begin{equation}\label{eq:splitting of cohomology}
        \HH^n (\g\ltimes \M) \cong \bigoplus_{k + \ell = n} \HH^k(\g;\extp^{\ell}\M^*),
    \end{equation}
    where $\M^* = \Hom(\M, \CC)$ is the dual of $\M$.
\end{proposition}
\begin{proof}
    Given $\psi \in C^n(\g \ltimes \M)$ and $k,\ell \in \NN$ with $k + \ell = n$, we define $\psi_{k\ell} \in C^n(\g \ltimes \M)$ as follows:
    $$\psi_{k\ell}(x_1,\ldots,x_r,v_1,\ldots,v_s) \coloneqq \begin{cases}
        \psi(x_1,\ldots,x_k,v_1,\ldots,v_\ell), &\text{if } r = k, s = \ell, \\
        0, &\text{otherwise},
    \end{cases}$$
    where $r + s = n$, $x_i \in \g$, and $v_i \in \M$. Then
    $$\psi = \sum_{k + \ell = n} \psi_{k\ell}$$
    for all $\psi \in C^n(\g \ltimes \M)$. Define
    $$C^{k,\ell}(\g \ltimes \M) \coloneqq \{\psi_{k\ell} \mid \psi \in C^{k + \ell}(\g \ltimes \M)\}.$$
    It is clear that $C^{k,\ell}(\g \ltimes \M) \cong \Hom\left(\extp^{k}\g \otimes \extp^{\ell}\M, \CC\right)$. The space of cochains $C^n(\g\ltimes \M) = \Hom(\extp^{n}(\g\ltimes \M), \CC)$ splits as
    $$C^n(\g\ltimes \M) = \bigoplus_{k+\ell=n} C^{k,\ell}(\g\ltimes\M).$$
    Now we show that $d$ preserves $\ell$, leading to a subcomplex $(C^{\dotr,\ell}(\g \ltimes \M),d)$ for each $\ell \in \NN$. Using the fact that $[x_i,x_j] \in \g$, $[x_i,v_j] \in \M$, $[v_i,v_j] = 0$, the formula \eqref{eq: LA cohomology d} applied to $\psi_{k\ell}\in C^{k,\ell}(\g\ltimes\M)$ reduces to
    \begin{multline}\label{eq:dpsi}
        d\psi_{k\ell}(x_1,\ldots,x_r,v_1,\ldots,v_s) \\
        = \sum_{1\leq i<j\leq r} (-1)^{i+j}\psi_{k\ell}([x_i,x_j],x_1,\dots, \hat{x}_i, \dots, \hat{x}_j, \dots, x_r, v_1, \dots, v_s) \\
        + \sum_{i = 1}^{k + 1} \sum_{j = 1}^{\ell} (-1)^{i + j + r} \psi_{k\ell}([x_i,v_j],x_1,\dots,\hat{x}_i,\dots,x_r,v_1,\dots, \hat{v}_j,\dots,v_s),
    \end{multline}
    where $r + s = n + 1$, $x_i \in \g$, and $v_i \in \M$. In both of the summations in \eqref{eq:dpsi}, the number of entries in $\g$ is reduced by one while the number of entries in $\M$ is left the same. It follows by definition of $\psi_{k\ell}$ that \eqref{eq:dpsi} is non-trivial only if $r = k + 1$ and $s = \ell$. This proves that$d\psi_{k\ell} \in C^{k+1,\ell}(\g \ltimes \M)$, as claimed.
    
    We therefore define $\HH^{\dotr,\ell}(\g\ltimes\M)$ to be the cohomology of the cochain complex $(C^{\dotr,\ell}(\g\ltimes\M),d)$. It follows that
    $$\HH^n(\g \ltimes \M) \cong \bigoplus_{k + \ell = n} \HH^{k,\ell}(\g \ltimes \M).$$
    To prove the isomorphism in the statement of the proposition, we show that the cochain complexes $C^{\dotr,\ell}(\g \ltimes \M)$ and $C^{\dotr}(\g; \extp^\ell \M^*)$ are isomorphic, which will yield
    $$\HH^{k,\ell}(\g \ltimes \M) \cong \HH^k(\g; \extp^\ell \M^*).$$
    This isomorphism is constructed as follows:
    $$\Phi \colon C^{k,\ell}(\g \ltimes \M) \to C^k(\g;\extp^\ell \M^*),$$
    defined by $\Big(\Phi(\psi_{k\ell})(x_1,\dots,x_k)\Big)(v_1,\dots, v_\ell) = \psi_{k\ell}(x_1,\dots,x_k,v_1,\dots,v_\ell)$, where $\psi_{k\ell} \in C^{k,\ell}(\g \ltimes \M)$, $x_i \in \g$, and $v_i \in \M$. It is clear that $\Phi$ is bijective for all $k,\ell \in \NN$.
    
    We now show that $\Phi$ is a chain map. To that end, let $\psi_{k\ell} \in C^{k,\ell}(\g \ltimes \M)$. The definition of the differential in \eqref{eq: LA cohomology d} gives
    \begin{multline}\label{eq:dPhi}
        d\Phi(\psi_{k\ell})(x_1,\dots,x_{k+1}) = \sum_{i = 1}^{k + 1} (-1)^{i+1} x_i \cdot \Phi(\psi_{k\ell})(x_1, \dots, \hat{x}_i, \dots, x_{k+1}) \\
        +\sum_{1\leq i < j \leq k+1} (-1)^{i+j} \Phi(\psi_{k\ell}) ([x_i, x_j], x_1, \dots, \hat{x}_i, \dots, \hat{x}_j, \dots, x_{k+1}).
    \end{multline}
    We briefly remark that one should, strictly speaking, use a different symbol for the differential in \eqref{eq:dPhi}, since it is meant to refer to the differential of the cochain complex whose cochains are given by $C^{\dotr}(\g;\extp^{\ell}\M^*)$. However, since it is constructed canonically from the one defined in \eqref{eq: LA cohomology d}, we abuse notation by using the same symbol, thereby also avoiding superfluous notation.
    On the other hand, using \eqref{eq:dpsi}, we get
    \begin{multline*}
        \Big(\Phi(d\psi_{k\ell})(x_1,\dots,x_{k+1})\Big)(v_1,\dots,v_\ell) = d\psi_{k\ell} (x_1,\dots,x_{k+1},v_1,\dots,v_\ell)\\
        = \sum_{1\leq i<j\leq k+1} (-1)^{i+j}\psi_{k\ell}([x_i,x_j],x_1,\dots, \hat{x}_i, \dots, \hat{x}_j, \dots, x_{k+1}, v_1, \dots, v_\ell) \\
        + \sum_{i = 1}^{k + 1} \sum_{j = 1}^{\ell} (-1)^{i+j+k+1}\psi_{k\ell}([x_i,v_j],x_1,\dots,\hat{x}_i,\dots,x_{k + 1},v_1,\dots, \hat{v}_j,\dots,v_\ell).
    \end{multline*}
    In the second summation, we can move the $[x_i,v_j]$ term to the right as follows:
    \begin{multline*}
        \sum_{i = 1}^{k + 1} \sum_{j = 1}^{\ell} (-1)^{i+j+k+1}\psi_{k\ell}([x_i,v_j],x_1,\dots,\hat{x}_i,\dots,x_{k + 1},v_1,\dots, \hat{v}_j,\dots,v_\ell) \\
        = \sum_{i = 1}^{k + 1} \sum_{j = 1}^{\ell} (-1)^{i}\psi_{k\ell}(x_1,\dots,\hat{x}_i,\dots,x_{k+1},v_1,\dots,v_{j-1}, [x_i,v_j],v_{j+1},\dots,v_\ell).
    \end{multline*}
    By definition of $\Phi$, we get
    \begin{multline*}
        \Big(\Phi(d\psi_{k\ell})(x_1,\dots,x_{k+1})\Big)(v_1,\dots,v_\ell) \\
        = \sum_{1\leq i<j\leq k+1} (-1)^{i+j}\Big(\Phi(\psi_{k\ell})([x_i,x_j],x_1\dots \hat{x}_i, \dots, \hat{x}_j, \dots, x_{k+1})\Big)(v_1, \dots, v_\ell) \\
        +\sum_{i = 1}^{k + 1} \sum_{j = 1}^{\ell} (-1)^{i}\Big(\Phi(\psi_{k\ell})(x_1,\dots,\hat{x}_i,\dots,x_{k+1})\Big)(v_1,\dots,v_{j-1}, [x_i,v_j],v_{j+1},\dots,v_\ell).
    \end{multline*}
    The definition of the action of $\g$ on $\Hom(\extp^\ell \M, \CC)$ now gives
    \begin{multline*}
        \Big(\Phi(d\psi_{k\ell})(x_1,\dots,x_{k+1})\Big)(v_1,\dots,v_\ell) \\
        = \sum_{i = 1}^{k + 1} (-1)^{i+j}\Big(\Phi(\psi_{k\ell})([x_i,x_j],x_1\dots \hat{x}_i, \dots, \hat{x}_j, \dots, x_{k+1})\Big)(v_1, \dots, v_\ell) \\
        + \sum_{i = 1}^{k + 1} (-1)^{i+1}\Big(x_i \cdot \Phi(\psi_{k\ell})(x_1,\dots,\hat{x}_i,\dots,x_{k+1})\Big)(v_1,\dots,v_\ell).
    \end{multline*}
    We conclude that
    \begin{multline}\label{eq:Phi d}
        \Phi(d\psi_{k\ell})(x_1,\dots,x_{k+1}) = \sum_{i = 1}^{k + 1} (-1)^{i+j}\Phi(\psi_{k\ell})([x_i,x_j],x_1\dots \hat{x}_i, \dots, \hat{x}_j, \dots, x_{k+1}) \\
        + \sum_{i = 1}^{k + 1} (-1)^{i+1}x_i \cdot \Phi(\psi_{k\ell})(x_1,\dots,\hat{x}_i,\dots,x_{k+1}).
    \end{multline}
    Combining \eqref{eq:dPhi} and \eqref{eq:Phi d}, we see that $\Phi(d\psi_{k\ell}) = d\Phi(\psi_{k\ell})$, which proves that $\Phi$ is a chain map. Since $\Phi$ is bijective for all $k,\ell \in \NN$, we conclude that $\Phi$ is an isomorphism of cochain complexes
    $$C^{\dotr,\ell}(\g\ltimes\M) \overset{\Phi}{\cong} C^{\dotr}(\g;\extp^{\ell}\M^*)$$
    for all $\ell \in \NN$. It follows that $\HH^{k,\ell}(\g \ltimes \M) \cong \HH^k(\g; \extp^\ell \M^*)$ for all $k,\ell \in \NN$, as claimed, which concludes the proof.
\end{proof}

As an immediate consequence of Proposition \ref{prop: cohomology splitting} in the case $n = 2$, we recover \cite[Proposition 1]{GaoLiuPei}.

\begin{corollary}[{\cite[Proposition 1]{GaoLiuPei}}] \label{cor:H2 splitting}
    Let $\g$ be a Lie algebra and $\M$ be a $\g$-module. Then
    \begin{equation}
        \HH^2(\g \ltimes \M) \cong \HH^2(\g) \oplus \HH^1(\g;\M^*)\oplus \B_\g (\M),
    \end{equation}
    where $\B_\g (\M)$ is the space of $\g$-invariant skew-symmetric bilinear forms on $\M$.\qed
\end{corollary}
\begin{remark}
    Proposition \ref{prop: cohomology splitting} does not hold when $\M$ is endowed with a non-abelian Lie algebra structure. This failure is evidently seen when considering the case $n=2$ in the following situation: if $[w,w']$ were non-zero for some $w,w' \in \M$, $\Omega\in\extp^2(\g \ltimes \M)^*$ is a 2-cocycle only if
    $$\Omega(w,[w',v]) + \Omega(w',[v,w]) + \Omega(v, [w,w']) = 0$$
    for all $v \in \g$. In particular, $\restr{\Omega}{\extp^2 \M} \notin \B_\g(\M)$ in general. Consequently, the full cohomology space does not split into direct sums of vector spaces described by the proposition.
    In fact, in \cite[Proposition 1]{GaoLiuPei} and its proof, it is implicitly assumed that the module $V$ has an abelian Lie algebra structure. This result was also used in \cite{OvsienkoRoger}.
\end{remark}

To formulate the cohomology splitting given by Corollary \ref{cor:H2 splitting} in terms of the generators of $\W_X(\lambda)$, we make the following definition.

\begin{defn}
    A 2-cocycle $\Omega$ on $\W_X(\lambda)$ is called
    \begin{itemize}
        \item a \emph{Virasoro} cocycle if $\Omega (L_n, X_m) = \Omega (X_n, X_m) = 0$;
        \item an \emph{abelian} cocycle if $\Omega (L_n, L_m) = \Omega (L_n, X_m) = 0$;
        \item a \emph{mixing} cocycle if $\Omega(L_n, L_m) = \Omega (X_n, X_m) = 0$;
    \end{itemize}
    for all $n,m \in \ZZ$.
\end{defn}

We can now make the following identifications with the decomposition in Corollary \ref{cor:H2 splitting}:
\begin{align*}
    \HH_{\Vir}^2(\W_X(\lambda)) &\cong \HH^2(\W), \\
    \HH_{\Mix}^2(\W_X(\lambda)) &\cong \HH^1(\W;X(\lambda)'), \\
    \HH_{\Ab}^2(\W_X(\lambda)) &\cong \B_{\W}(X(\lambda)),
\end{align*}
where $\HH_{\Vir}^2(\W_X(\lambda))$, $\HH_{\Ab}^2(\W_X(\lambda))$ and $\HH_{\Mix}^2(\W_X(\lambda))$ denote the images of the spaces of Virasoro, abelian and mixing cocycles in $\HH^2(\W_X(\lambda))$, respectively. Note that the second isomorphism simplifies the computation of outer derivations of $\W_X(\lambda)$. We elaborate on this in Section \ref{sec:derivations}.

We therefore arrive at the following decomposition of $\HH^2(\W_X(\lambda))$ which we use throughout this and the next two sections:
\begin{equation}\label{eq: H^2 vector space splitting}
  \HH^2(\W_X(\lambda)) = \HH_{\Vir}^2(\W_X(\lambda)) \oplus \HH_{\Ab}^2(\W_X(\lambda)) \oplus \HH_{\Mix}^2(\W_X(\lambda)).  
\end{equation}
Proposition \ref{prop: cohomology splitting} implies that the Gelfand--Fuchs cocycle
\begin{equation}\label{eq:Virasoro cocycle}
    \Omega_{\Vir} (L_n, L_m) = \frac{1}{12}n(n^2-1)\delta^0_{m+n}.
\end{equation}
can always be pulled back to a non-trivial cocycle on $\W_X(\lambda)$. In terms of Laurent polynomials, the Gelfand--Fuchs cocycle has the following form: given $f,g \in \CC[t,t^{-1}]$, we have
$$\Omega_{\Vir}(f\del,g\del) = \frac{1}{12}\ang{f',g'} = \frac{1}{12}\Res(f'g'' \ dt).$$
It is a well-known fact that the Gelfand--Fuchs cocycle produces the Virasoro algebra $\Vir$, the universal central extension of the Witt algebra \cite{Schottenloher}. Since $X(\lambda)$ is an ideal in $\W_X(\lambda)$, it follows that the Gelfand--Fuchs cocycle extends to the unique cohomologically non-trivial Virasoro cocycle on $\W_X(\lambda)$. Thus, $\dim(\HH_{\Vir}^2(\W_X(\lambda))) = 1$, so $\dim(\HH^2 (\W_X(\lambda))) \geq 1$ for all $X \in \{A,B\}$ and $\lambda\in\PP^1$.

The main result of this section is the computation of central extensions of $\W_X(\lambda)$ for $X \in \{A,B\}$ and $\lambda \in \PP^1$.

\begin{theorem}\label{thm:main}
    Let $\lambda \in \PP^1$. Then
    \begin{align*}
        \dim(\HH^2(\W_A(\lambda))) &= \begin{cases}
            3, &\text{if } \lambda = 0,\\
            2, &\text{if } \lambda \neq 0,
        \end{cases}\\
        \dim(\HH^2(\W_B(\lambda))) &= 3.
    \end{align*}
    More explicitly, for $X = A$ and $\lambda \in \PP^1$ define $\Omega_{\Mix}^A \in Z^2(\W_A(\lambda))$, for $X = A$ and $\lambda = 0$ define $\Omega_0^A \in Z^2(\W_A(0))$, and for $X = B$ and $\lambda \in \PP^1$ define $\Omega_{\Ab}^B, \Omega_{\Mix}^B \in Z^2(\W_B(\lambda))$ as follows, where we specify only the non-zero components of the cocycles:
    \begin{align*}
        \Omega_0^A (L_n, A_m) &= n\delta^0_{n+m}; & \Omega_{\Mix}^A (L_n, A_m) &= \begin{cases}
            (\lambda + 1)\delta_m^0, &\text{if } n = 0 \text{ and } \lambda \neq \infty, \\
            \delta_{n+m}^0, &\text{otherwise};
        \end{cases} \\
        \Omega_{\Ab}^B (B_n, B_m) &= n\delta^0_{n+m}; & \Omega_{\Mix}^B(L_n, B_m) &= \begin{cases}
            n\delta^0_{n+m}, &\text{if } \lambda \neq 0, \\
            n(n + 1)\delta^0_{n+m}, &\text{if } \lambda = 0;
        \end{cases}
    \end{align*}
    for all $n,m \in \ZZ$. The spaces of one-dimensional central extensions are then given by
    \begin{align*}
        \HH^2(\W_A(\lambda)) &=
        \begin{cases}
            \CC \overline{\Omega}_{\Vir} \oplus \CC \overline{\Omega}_0^A \oplus \CC \overline{\Omega}_{\Mix}^A, &\text{if } \lambda=0.\\
            \CC \overline{\Omega}_{\Vir} \oplus \CC \overline{\Omega}_{\Mix}^A, &\text{if } \lambda \neq 0.
        \end{cases} \\
        \HH^2(\W_B(\lambda)) &= \CC \overline{\Omega}_{\Vir} \oplus \CC \overline{\Omega}_{\Ab}^B \oplus \CC \overline{\Omega}_{\Mix}^B,
    \end{align*}
    where $\overline{\Omega}$ denotes the image of the 2-cocycle $\Omega \in Z^2(\W_X(\lambda))$ in the cohomology space $\HH^2(\W_X(\lambda))$. Here, $\Omega_{\Vir}$ is the Gelfand-Fuchs cocycle in \eqref{eq:Virasoro cocycle}.
\end{theorem}

\begin{remark}\label{rem:basis-free cocycles}
    Under the identification $B(\lambda) = \CC[t,t^{-1}]$ mentioned in Section \ref{sec:preliminaries}, the cocycle $\Omega_{\Ab}^B$ is the invariant skew-symmetric form $\ang{-,-}$ defined in \eqref{eq:self duality of I}. Specifically, letting $f,g \in \CC[t,t^{-1}] = B(\lambda)$, we have
    $$\Omega_{\Ab}^B(f,g) = \ang{g,f} = \Res(f'g \ dt).$$
    On the other hand, the form $\ang{-,-}$ on $\widetilde{I}$ does not extend to $A(\lambda)$.

    The cocycle $\Omega_{\Mix}^B$ is given by
    $$\Omega_{\Mix}^B(f\del,g) = \begin{cases}
        \ang{g,t^{-1}f}, &\text{if } \lambda \neq 0, \\
        \ang{g,f'}, &\text{if } \lambda = 0.
    \end{cases}$$
    Similarly, using the identification $A(\lambda) = \CC[t,t^{-1}]$, we get a description of the cocycles $\Omega_0^A$ and $\Omega_{\Mix}^A$ in terms of polynomials. We have
    \begin{align*}
        \Omega_0^A(f\del, g \ dt) &= \Res((f' - t^{-1}f)g \ dt), \\
        \Omega_{\Mix}^A(f\del, g \ dt) &= \begin{cases}
            \Res(t^{-1}fg \ dt) + \lambda \ang{t^{-1},f}\Res(g \ dt), &\text{if } \lambda \neq \infty, \\
            \Res(t^{-1}fg \ dt), &\text{if } \lambda = \infty.
        \end{cases}
    \end{align*}
    Using the form $\mathcal{B}$ defined in \eqref{eq:B form}, we can also write the above as
    \begin{align*}
        \Omega_0^A(f\del,g \ dt) &= \mathcal{B}(f' - t^{-1}f, g \ dt) \\
        \Omega_{\Mix}^A(f\del, g \ dt) &= \begin{cases}
            \mathcal{B}(t^{-1}f + \lambda \ang{t^{-1},f}, g \ dt), &\text{if } \lambda \neq \infty, \\
            \mathcal{B}(t^{-1}f, g \ dt), &\text{if } \lambda = \infty,
        \end{cases} \\
        \Omega_{\Mix}^B(f\del, g) &= \begin{cases}
            \mathcal{B}(g, d(t^{-1}f)), &\text{if } \lambda \neq 0, \\
            \mathcal{B}(g, d(f')), &\text{if } \lambda = 0.
        \end{cases}
    \end{align*}
\end{remark}

Theorem \ref{thm:main} implies that we can define central extensions of the Lie algebras $\W_X(\lambda)$ as follows: as vector spaces, let $\Vir_A(\lambda) = \W_A(\lambda) \oplus \CC c_{\Vir} \oplus \CC c_{\Mix}^A$ and $\Vir_B(\lambda) = \W_B(\lambda) \oplus \CC c_{\Vir} \oplus \CC c_{\Ab}^B \oplus \CC c_{\Mix}^B$, where $\lambda \in \PP^1 \nonzero$. The Lie brackets are then given by
$$\begin{gathered}
    [L_n,L_m] = (m - n)L_{n+m} + \frac{1}{12}n(n^2 - 1)\delta_{n+m}^0 c_{\Vir}, \\
    [L_n,A_m] = \begin{cases}
        m A_m + (\lambda + 1)\delta_m^0 c_{\Mix}^A, &\text{if } n = 0 \text{ and } \lambda \neq \infty, \\
        (n + m + n(n + 1)\lambda \delta_m^0)A_{n+m} + \delta_{n+m}^0 c_{\Mix}^A, &\text{otherwise},
    \end{cases} \\
    [A_n,A_m] = 0, \\
    [L_n,B_m] = (m - n(n + 1)\lambda \delta_{n+m}^0)B_{n+m} + n\delta_{n+m}^0 c_{\Mix}^B, \\
    [B_n,B_m] = n\delta_{n+m}^0 c_{\Ab}^B,
\end{gathered}$$
with $c_{\Vir}, c_{\Mix}^A, c_{\Mix}^B, c_{\Ab}^B$ central. The central extension $\Vir_B(0)$ of $\W_B(0)$ is defined similarly, but the bracket $[L_n,B_m]$ becomes $[L_n,B_m] = m B_{n+m} + n^2\delta_{n+m}^0 c_{\Mix}^B$. The central extension $\Vir_B(\lambda)$ of $\W_B(\lambda)$ is universal for all $\lambda \in \PP^1$, while $\Vir_A(\lambda)$ is not universal.

\begin{remark}
    Let $\lambda \in \PP^1$ and let $\widehat{B}(\lambda) = B(\lambda) \oplus \CC c_{\Ab}^B$. As a Lie algebra, $\widehat{B}(\lambda)$ is the infinite-dimensional Heisenberg algebra $\mathfrak{H}$. It follows that the one-dimensional central extension
    $$\W_B(\lambda) \oplus \CC c_{\Ab}^B \cong \W \ltimes \widehat{B}(\lambda) \cong \W \lsemidirect{\lambda} \mathfrak{H}$$
    is a semi-direct sum of the Witt algebra with the Heisenberg algebra, where the symbol $\lsemidirect{\lambda}$ emphasises the dependence of the $\W$-module structure of $\mathfrak{H}$ on the parameter $\lambda$.

    Similarly, the two-dimensional central extension
    $$\W_B(\lambda) \oplus \CC c_{\Vir} \oplus \CC c_{\Ab}^B \cong \Vir \ltimes \widehat{B}(\lambda) \cong \Vir \lsemidirect{\lambda} \mathfrak{H}$$
    is a semi-direct sum of the Virasoro algebra with the Heisenberg algebra.
\end{remark}

To prove Theorem \ref{thm:main}, we proceed with the search for abelian and mixing cocycles on $\W_X(\lambda)$. First, we require a definition which will allow us to greatly simplify our computations.

\begin{defn}\label{def:internally graded}
    Let $\g$ be a Lie algebra. We say that $\g$ is \emph{internally graded} if $\g$ has an element $x_0$ and a basis consisting of eigenvectors for $\ad_{x_0}$. The Lie algebra $\g$ is then naturally graded as follows:
    $$\g_\lambda = \{x \in \g \mid [x_0,x] = \lambda x\},$$
    where $\lambda \in \CC$. 
    
    If $\g$ is internally graded, we say that a $\g$-module $\M$ is \emph{internally graded} if $\M$ has a basis consisting of eigenvectors for the action of $x_0$. Then $\M$ is naturally graded as follows:
    $$\M_\lambda = \{z \in \M \mid x_0 \cdot z = \lambda z\},$$
    where $\lambda \in \CC$.
\end{defn}

Certainly, the Lie algebras $\W_X(\lambda)$ are internally $\ZZ$-graded, with homogeneous components given by the eigenspaces of $L_0$: we have 
$$\W_X(\lambda)_n = \CC L_n + \CC X_n$$
for all $n \in \ZZ$, since $L_n$ and $X_n$ are eigenvectors of $L_0$ with eigenvalue $n$. Throughout this section, we will exploit this internal grading of $\W_X(\lambda)$ to make our search space significantly smaller.

\begin{defn}
    Let $\g = \bigoplus_{n \in \ZZ} \g_n$ be a $\ZZ$-graded Lie algebra and $\M = \bigoplus_{n \in \ZZ} \M_n$ be a $\ZZ$-graded $\g$-module. For $p \in \ZZ$, the \emph{degree 0 $p$-cochains of $\g$ with values in $\M$} are defined as follows:
    $$C_0^p(\g;\M) = \left\{\varphi \in C^p(\g;\M) \mid \varphi(x_1,\ldots,x_p) \in \M_{\deg(x_1) + \ldots + \deg(x_p)}, \text{ where } x_i \in \g \text{ are homogeneous}\right\}.$$
    One can easily verify that $d(C_0^p(\g;\M)) \subseteq C_0^{p+1}(\g;\M)$ for all $p \in \ZZ$. Thus, we define the \emph{degree 0 cohomology of $\g$} as the cohomology of the cochain complex $(C_0^*(\g;\M),d)$, which we denote $\HH_0^*(\g;\M)$. As usual, we write $C_0^*(\g)$ and $\HH_0^*(\g)$ instead of $C_0^*(\g;\CC)$ and $\HH_0^*(\g;\CC)$.
\end{defn}

In fact, when $\g$ is internally graded, the degree 0 cohomology of $\g$ is exactly the same as the usual cohomology of $\g$.

\begin{theorem}[{\cite[Theorem 1.5.2]{Fuchs}}]\label{thm:internal grading}
    Let $\g$ be an internally graded Lie algebra and let $\M$ be an internally graded $\g$-module. Then the inclusion $C_0^*(\g;\M) \to C^*(\g;\M)$ induces an isomorphism in cohomology $\HH_0^*(\g;\M) \cong \HH^*(\g;\M)$.
\end{theorem}

By Theorem \ref{thm:internal grading}, the internally graded structure of $\W_X(\lambda)$ means that any 2-cocycle with coefficients in $\CC$ can only be non-trivial on the degree zero part of $\W_X(\lambda) \times \W_X(\lambda)$. In other words, if $\Omega \in Z^2(\g)$, we may assume that
$$\Omega(L_n,L_m) = \Omega(L_n,X_m) = \Omega(X_n,X_m) = 0$$
for all $n,m \in \ZZ$ such that $n + m \neq 0$.

\section{Abelian cocycles}\label{sec:abelian cocycles}

The goal of this section is to compute $\HH_{\Ab}^2(\W_X(\lambda))$ for $\lambda \in \PP^1$ and $X = A$ or $B$.

\begin{proposition}\label{prop:W_X abelian cocycles}
    Let $\lambda \in \PP^1$. Then
    \begin{align*}
        \HH_{\Ab}^2&(\W_A(\lambda)) = 0, \\
        \dim(\HH_{\Ab}^2&(\W_B(\lambda))) = 1.
    \end{align*}
\end{proposition}

By the isomorphism $\HH^2_{\Ab}(\W_X(\lambda)) \cong \B_\W(X(\lambda))$, Proposition \ref{prop:W_X abelian cocycles} implies that
\begin{align*}
    &\B_\W(A(\lambda)) = 0, \\
    \dim(&\B_\W(B(\lambda))) = 1.
\end{align*}
By Theorem \ref{thm:internal grading}, we only need to compute abelian cocycles of degree 0 to prove Proposition \ref{prop:W_X abelian cocycles}. In other words, we may assume that any abelian cocycle of $\W_X(\lambda)$ is of the form
\begin{equation}
    \Omega (X_n, X_m) = \alpha(n) \delta_{n+m}^0,
\end{equation}
where $\alpha \colon \ZZ \to \CC$. By the antisymmetry of $\Omega$, the function $\alpha$ must satisfy $\alpha(-n) = - \alpha(n)$ for all $n \in \ZZ$, and in particular, $\alpha(0) = 0$.

\begin{defn}
    A function $\alpha \colon \ZZ \to \CC$ is called an \emph{abelian cocycle function} on $\W_X(\lambda)$ if the antisymmetric bilinear map $\Omega \colon \W_X(\lambda) \times \W_X(\lambda) \to \CC$ defined by
    $$\Omega(X_n,X_m) = \alpha(n)\delta_{n+m}^0$$
    and $\Omega(L_n,L_m) = \Omega(L_n,X_m) = 0$ is a 2-cocycle, where $n,m \in \ZZ$.
    
    The space of all abelian cocycle functions on $\W_X(\lambda)$ is denoted $\Ab_X(\lambda)$, where $X = A$ or $B$.
\end{defn}

Note that there are no nonzero abelian coboundaries (that is, abelian cocycles which are coboundaries): this is because if $\varphi \colon \W_X(\lambda) \to \CC$ gives rise to an abelian cocycle $d\varphi$, then
$$d\varphi(L_n,L_m) = d\varphi(L_n,X_m) = 0,$$
and $d\varphi(X_n,X_m) = \varphi([X_n,X_m]) = \varphi(0) = 0$ for all $n,m \in \ZZ$, so $d\varphi = 0$. Therefore, we see that
$$\HH_{\Ab}^2(\W_X(\lambda)) \cong \Ab_X(\lambda).$$

If $\Omega$ is a cocycle on $\W_X(\lambda)$, then
\begin{equation}\label{eq:W_X abelian cocycle cond}
    \Omega(X_n,[X_m,L_{-n-m}]) + \Omega(X_m,[L_{-n-m},X_n]) + \Omega(L_{-n-m},[X_n,X_m]) = 0
\end{equation}
for all $n,m \in \ZZ$.

We first look for abelian cocycles of $\W_A (\lambda)$. We aim to prove the following:

\begin{proposition}\label{prop:W_A has no abelian cocycles}
    Let $\lambda \in \PP^1$. Then there are no non-trivial abelian cocycles on $\W_A(\lambda)$, in other words, $\Ab_A(\lambda) = 0$. Consequently, $\HH_{\Ab}^2(\W_A(\lambda)) = 0$.
\end{proposition}

We split this into two cases: $\lambda\neq\infty$ (in other words, $\lambda\in\mathbb{C}$) and $\lambda=\infty$.

\begin{lemma}
    Let $\lambda \in \CC$ and $\alpha \in \Ab_A(\lambda)$. Then $\alpha = 0$.
\end{lemma}
\begin{proof}
    The cocycle condition \eqref{eq:W_X abelian cocycle cond} reads
    \begin{equation}
    \label{eq: W_A abelian cocycle cond}
        (m + (m + n)(1 - m - n) \lambda \delta^0_n) \alpha(m) - (n + (m + n)(1 - m - n) \lambda \delta^0_m) \alpha(n) = 0,
    \end{equation}
    for all $n,m \in \ZZ$. Consider $m\neq 0,\ n=0$. Then \eqref{eq: W_A abelian cocycle cond} gives
    \begin{equation}
        m\big(1+(1-m)\lambda\big)\alpha(m) = 0 \iff (1+(1-m)\lambda\big)\alpha(m) = 0.
    \end{equation}
    Setting $m=1$ above forces $\alpha(1)=0$, and thus $\alpha(-1)=0$ due to the antisymmetry of $\alpha: \ZZ \rightarrow \CC$. Likewise, for any $m\in\mathbb{Z}$, we obtain $\alpha(m)=0$ from the above, which finishes the proof.
\end{proof}

Similarly to the case $\lambda \in \CC$, it is straightforward to prove that $\W_A(\infty)$ also does not have any non-trivial abelian cocycles.

\begin{lemma}
    Let $\alpha \in \Ab_A(\infty)$. Then $\alpha = 0$.
\end{lemma}
\begin{proof}
    The cocycle condition reads
    \begin{equation}
    \label{eq: W_A abelian cocycle cond infty}
        (-n+(m+n)^2\delta^0_m)\alpha(n)-(-m+(m+n)^2\delta^0_n)\alpha(m)=0,
    \end{equation}
    for all $n,m \in \ZZ$. Once again, we consider $m\neq 0,\ n=0$, so \eqref{eq: W_A abelian cocycle cond infty} gives 
    \begin{equation}
        m(1-m)\alpha(m)=0.
    \end{equation}
    By the same argument as for the $\lambda\neq\infty$ case, the above tells us that $\alpha(m)=0$ for all $m\in\ZZ$.
\end{proof}

This concludes the proof of Proposition \ref{prop:W_A has no abelian cocycles}.

The situation for $\W_B(\lambda)$ is very different to that of $\W_A(\lambda)$: we will see that there is a non-trivial abelian cocycle, which we define below.

\begin{Notation}\label{ntt:abelian cocycle B}
    Define $\iota \colon \ZZ \to \CC$ by $\iota(n) = n$ for all $n \in \ZZ$.
\end{Notation}

We now prove that $\iota$ is an abelian cocycle function for $\W_B(\lambda)$, and that $\Ab_B(\lambda)$ is one-dimensional, spanned by $\iota$.

\begin{proposition}\label{prop:W_B abelian cocycles}
    Let $\lambda \in \PP^1$. Then $\iota$ is the unique (up to scalar multiplication) abelian cocycle function on $\W_B(\lambda)$, in other words, $\Ab_B(\lambda) = \CC \iota$. In particular, if $\Omega$ is an abelian cocycle on $\W_B(\lambda)$ of degree zero, then, up to multiplication by a scalar, we have
    $$\Omega(B_n,B_m) = n\delta_{n+m}^0.$$
    Consequently, $\dim(\HH_{\Ab}^2(\W_B(\lambda))) = 1$.
\end{proposition}
\begin{proof}
    Let $\alpha \in \Ab_B(\lambda)$. The cocycle condition \eqref{eq:W_X abelian cocycle cond} reads
    \begin{equation}
        \label{eq: W_B abelian cocycle cond}
        (m+(m+n)(1-m-n)\lambda\delta^0_n)\alpha(n)-(n+(m+n)(1-m-n)\delta^0_m)\alpha(m)=0
    \end{equation}
    for $\lambda\neq\infty$, with $- (m+n)^2$ instead of $(m+n)(1-m-n)\lambda$ for $\lambda=\infty$. The only non-trivial case to consider is $m,n\neq0$, in which case, equation \eqref{eq: W_B abelian cocycle cond} gives
    \begin{equation}
    \label{eq: reduced W_B abelian cocycle cond}
        m\alpha(n)-n\alpha(m)=0.
    \end{equation}
    Any abelian cocycle on $\W_B(\lambda)$ has to satisfy \eqref{eq: reduced W_B abelian cocycle cond} for $m,n\neq 0$, with no other constraints. To show that there is a unique (up to scalar multiplication) solution, we set $m=1$ in \eqref{eq: reduced W_B abelian cocycle cond}, which gives $\alpha(n)=n\alpha(1)$. We split the rest of the proof in two cases: $\alpha(1) = 0$ and $\alpha(1) \neq 0$.
    
    If $\alpha(1)=0$, the above gives $\alpha(n)=0$ for all $n \neq 0$. But the antisymmetry of $\alpha \colon \ZZ\rightarrow\CC$ implies $\alpha(0)=0$, so $\alpha(n)=0$ for all $n\in\ZZ$.

    If $\alpha(1)\neq 0$, we obtain the nonzero solution $\alpha = \alpha(1) \iota$. This concludes the proof.
\end{proof}

Proposition \ref{prop:W_X abelian cocycles} now follows by combining Propositions \ref{prop:W_A has no abelian cocycles} and \ref{prop:W_B abelian cocycles}.

\section{Mixing cocycles}\label{sec:mixing cocycles}

In this section, we now turn our attention to mixing cocycles on $\W_X(\lambda)$, with the goal of computing $\HH_{\Mix}^2(\W_X(\lambda))$.

\begin{proposition}\label{prop:W_X mixing cocycles}
    Let $\lambda \in \PP^1$. Then
    $$\dim(\HH_{\Mix}^2(\W_A(\lambda))) = \begin{cases}
        2, &\text{if } \lambda = 0,\\
        1, &\text{if } \lambda \neq 0,
    \end{cases}$$
    while $\dim(\HH_{\Mix}^2(\W_B(\lambda))) = 1$.
\end{proposition}

By the isomorphism $\HH^2_{\Mix}(\W_X(\lambda)) \cong \HH^1(\W;X(\lambda)')$, Proposition \ref{prop:W_X abelian cocycles} implies that
\begin{align*}
    \dim(\HH^1(\W;A(\lambda))) &= 1, \\
    \dim(\HH^1(\W;B(\lambda))) &= \begin{cases}
        2, &\text{if } \lambda = 0, \\
        1, &\text{if } \lambda \neq 0,
    \end{cases}.
\end{align*}
Where we have used that $A(\lambda)' \cong B(\lambda)$ and $B(\lambda)' \cong A(\lambda)$.

Similarly to before, we may assume that a mixing cocycle $\Omega \colon \W_X(\lambda) \times \W_X(\lambda) \to \CC$ is of the form
\begin{equation}
    \Omega(L_n,X_m) = \beta(n)\delta_{n+m}^0,
\end{equation}
where $\beta \colon \ZZ \to \CC$. Note that in this case, we no longer have $\beta(-n) = -\beta(n)$.

\begin{defn}
    A function $\beta \colon \ZZ \to \CC$ is called a \emph{mixing cocycle function} on $\W_X(\lambda)$ if the antisymmetric bilinear map $\Omega \colon \W_X(\lambda) \times \W_X(\lambda) \to \CC$ defined by
    $$\Omega(L_n,X_m) = \beta(n)\delta_{n+m}^0$$
    and $\Omega(L_n,L_m) = \Omega(X_n,X_m) = 0$ is a 2-cocycle, where $n,m \in \ZZ$. If $\Omega$ is a coboundary, then we call $\beta$ a \emph{mixing coboundary function}.
    
    The space of all mixing cocycle functions on $\W_X(\lambda)$ is denoted $\ZMix_X(\lambda)$, while the space of mixing coboundary functions on $\W_X(\lambda)$ is denoted $\BMix_X(\lambda)$.
\end{defn}

In this case, we have
$$\HH_{\Mix}^2(\W_X(\lambda)) \cong \ZMix_X(\lambda)/\BMix_X(\lambda),$$
by Theorem \ref{thm:internal grading}. The cocycle condition gives
\begin{equation}\label{eq:W_X mixing cocycle cond}
    \Omega(L_n,[L_m,X_{-n-m}]) + \Omega(L_m,[X_{-n-m},L_n]) + \Omega(X_{-n-m},[L_n,L_m]) = 0,
\end{equation}
where $n,m \in \ZZ$, which gives a condition that mixing cocycle functions must satisfy. As we did for abelian cocycles, we begin by considering $\W_A(\lambda)$. We first define a mixing cocycle function for $\W_A(\lambda)$.

\begin{Notation}\label{ntt:mixing cocycle A}
    For $\lambda \in \CC$, define $\beta_\lambda \colon \ZZ \to \CC$ by
    $$\beta_\lambda(n) = \begin{cases}
        \lambda + 1, &\text{if $n  = 0$}, \\
        1, &\text{otherwise}.
    \end{cases}$$
    We define $\beta_\infty \colon \ZZ \to \CC$ by $\beta(n) = 1$ for all $n \in \ZZ$.
\end{Notation}

Proposition \ref{prop:W_X mixing cocycles} for $\W_A(\lambda)$ will follow from the refinement in the next result.

\begin{proposition}\label{prop:W_A mixing cocycles}
    Let $\lambda \in \PP^1$. Then $\ZMix_A(0) = \CC \beta_0 \oplus \CC \iota$, where $\iota$ is as in Notation \ref{ntt:abelian cocycle B}, and $\ZMix_A(\lambda) = \CC \beta_\lambda$ if $\lambda \neq 0$. In particular, if $\Omega$ is a mixing cocycle on $\W_A(\lambda)$ of degree zero, then for $\lambda \neq 0$ we have, up to multiplication by a scalar,
    $$\Omega(L_n,A_m) = \begin{cases}
        (\lambda + 1)\delta_{n+m}^0, &\text{if } n = 0 \text{ and } \lambda \neq \infty, \\
        \delta_{n+m}^0, &\text{otherwise},
    \end{cases}$$
    while for $\lambda = 0$, there exist $a,b \in \CC$ such that
    $$\Omega(L_n,A_m) = (an + b)\delta_{n+m}^0.$$
    Furthermore, $\BMix_A(\lambda) = 0$, so
    $$\dim(\HH_{\Mix}^2(\W_A(\lambda))) = \begin{cases}
        2, &\text{if } \lambda = 0,\\
        1, &\text{if } \lambda \neq 0.
    \end{cases}$$
\end{proposition}

Since it is not entirely obvious, we first explicitly prove that $\beta_\lambda \in \ZMix_A(\lambda)$.

\begin{lemma}\label{lem:mixing cocycle function A}
    Let $\lambda \in \PP^1$. Then the map $\beta_\lambda$ is a mixing cocycle function on $\W_A(\lambda)$. In other words, there is a mixing cocycle $\Omega \in Z^2(\W_A(\lambda))$ defined by
    $$\Omega(L_n,A_m) = \begin{cases}
        (\lambda + 1)\delta_{n+m}^0, &\text{if } n = 0 \text{ and } \lambda \neq \infty, \\
        \delta_{n+m}^0, &\text{otherwise},
    \end{cases}$$
    for all $n,m \in \ZZ$.
\end{lemma}
\begin{proof}
    Consider $\lambda \in \CC$. In this case, the cocycle condition \eqref{eq:W_X mixing cocycle cond} gives
    \begin{equation}\label{eq:mixing cocycle A}
        (-n + m(m + 1)\lambda \delta_{n+m}^0)\beta(n) - (-m + n(n + 1)\lambda \delta_{n+m}^0)\beta(m) + (n - m)\beta(n + m) = 0.
    \end{equation}
    If $n = 0$ or $m = 0$, then \eqref{eq:mixing cocycle A} for $\beta_\lambda$ is trivially satisfied. Furthermore, if $m \neq -n$ with $n,m \in \ZZ \nonzero$, then \eqref{eq:mixing cocycle A} is once again easily checked.
    
    Therefore, it suffices to check \eqref{eq:mixing cocycle A} for $m = -n \neq 0$. In this case, we have
    \begin{align*}
        (-n + m(m + 1)\lambda \delta_{n+m}^0)\beta_0(n) - &(-m + n(n + 1)\lambda \delta_{n+m}^0)\beta_0(m) + (n - m)\beta_0(n + m) \\
        &= -n + n(n - 1)\lambda - (n + n(n + 1)\lambda) + 2(\lambda + 1)n \\
        &= -n + \lambda n^2 - \lambda n - n - \lambda n^2 - \lambda n + 2\lambda n + 2n = 0,
    \end{align*}
    as required.

    Now let $\lambda = \infty$. Here, the cocycle condition gives
    \begin{equation}\label{eq:mixing cocycle A infinity}
        (-n + m^2\delta_{n+m}^0)\beta(n) - (-m + n^2\delta_{n+m}^0)\beta(m) + (n - m)\beta(n + m) = 0.
    \end{equation}
    We need to check that $\beta_\infty$ satisfies \eqref{eq:mixing cocycle A infinity}. For $n,m \in \ZZ$, we have
    \begin{align*}
        (-n + m^2\delta_{n+m}^0)\beta_\infty(n) - &(-m + n^2\delta_{n+m}^0)\beta_\infty(m) + (n - m)\beta_\infty(n + m) \\
        &= -n + m^2\delta_{n+m}^0 - (-m + n^2\delta_{n+m}^0) + n - m \\
        &= (n + m)(m - n)\delta_{n+m}^0 = 0,
    \end{align*}
    as required.
\end{proof}

We now prove that there are no nonzero mixing coboundary functions on $\W_A(\lambda)$.

\begin{lemma}\label{lem:no mixing coboundaries}
    Let $\lambda \in \PP^1$. Then there does not exist any nonzero mixing coboundary function on $\W_A(\lambda)$. In other words, $\BMix_A(\lambda) = 0$. Consequently, $\HH_{\Mix}^2(\W_A(\lambda)) \cong \ZMix_A(\lambda)$.
\end{lemma}
\begin{proof}
    Let $\beta \in \BMix_A(\lambda)$, so there exists some $\varphi \colon \W_A(\lambda) \to \CC$ such that $d\varphi(L_n, A_{-n}) = \beta(n)$ for all $n \in \ZZ$. Then $\beta(n) = d\varphi(L_n, A_{-n}) = -\varphi([L_n, A_{-n}]) = 0$ for all $n \in \ZZ$, so $\beta = 0$.
\end{proof}

We now complete the proof of Proposition \ref{prop:W_A mixing cocycles} by proving that $\beta_\lambda$ is the unique (up to scalar multiplication) mixing cocycle function on $\W_A(\lambda)$ for $\lambda \neq 0$. Note that when $\lambda = 0$, we can apply \cite[Theorem 2.3]{GaoJiangPei}, since $A(0) \cong I(0,1)$ as representations of $\W$.

\begin{proof}[Proof of Proposition \ref{prop:W_A mixing cocycles}]
    Let $\beta \in \ZMix_A(\lambda)$. As mentioned above, the proof for $\lambda = 0$ follows by \cite[Theorem 2.3]{GaoJiangPei}: we have
    $$\dim(\ZMix_A(0)) = \dim(\HH_{\Mix}^2(\W_A(0))) = \dim(\HH^2_{\Mix} (\W(0,1))) = 2.$$
    Indeed, $\beta_0$ gives an element of $\ZMix_A(0)$ by Lemma \ref{lem:mixing cocycle function A}, while $\iota$ from Notation \ref{ntt:abelian cocycle B} is also a mixing cocycle function on $\W_A(0)$. Therefore, $\ZMix_A(0)$ is spanned by $\beta_0$ and $\iota$.
    
    Thus, we may assume that $\lambda \neq 0$. We analyse the cases $\lambda \in \CC\setminus\{0\}$ and $\lambda = \infty$ separately.

    First, consider $\lambda \in \CC \nonzero$. Let $\beta' = \beta - \beta(1)\beta_\lambda$, so that $\beta'$ is still a mixing cocycle function on $\W_A(\lambda)$ and $\beta'(1) = 0$. We claim that $\beta'(-1) = 0$.

    Assume, for a contradiction, that $\beta'(-1) \neq 0$. Rescaling $\beta'$ if necessary, we may assume that $\beta'(-1) = 1$. Taking $n = 2$ and $m = -1$ in \eqref{eq:mixing cocycle A}, we get $\beta'(2) = -\frac{1}{2}$, while $n = -2$ and $m = 1$ gives $\beta'(-2) = \frac{3}{2}$. Now we have two ways of computing $\beta'(0)$: we can substitute $n = 1, m = -1$ or $n = 2, m = -2$ into \eqref{eq:mixing cocycle A}. The former gives $\beta'(0) = \lambda + \frac{1}{2}$, while the latter gives $\beta'(0) = \frac{1}{2}(5\lambda + 1)$. Combining these two equations, we conclude that $\lambda = 0$. But we assumed that $\lambda \in \CC \nonzero$, so this is a contradiction. This proves the claim.

    Since $\beta'(-1) = 0$, we can take $n = 1$ and $m = -1$ in \eqref{eq:mixing cocycle A} to get $\beta'(0) = 0$. Letting $n \geq 2$ and $m = -1$, it follows that
    $$n\beta'(n) = (n + 1)\beta'(n - 1).$$
    By induction, we conclude that $\beta'(n) = 0$ for all $n \geq 2$. Similarly, taking $n \leq -2$ and $m = 1$, we get
    $$n\beta'(n) = (n - 1)\beta'(n + 1),$$
    so $\beta'(n) = 0$ for $n \leq -2$ by induction. Therefore, $\beta' = 0$ and $\beta = \beta(1)\beta_\lambda$, as required. This concludes the proof for $\lambda \neq \infty$.

    Now consider $\lambda = \infty$, and let $\beta' = \beta - \beta(0)\beta_\infty$. Then $\beta' \in \ZMix_A(\infty)$ and $\beta'(0) = 0$. We claim that $\beta'(1) = 0$.

    Assume, for a contradiction, that $\beta'(1) \neq 0$. Rescaling $\beta'$ if necessary, we may assume that $\beta'(1) = 1$. Taking $n = 1$ and $m = -1$ in \eqref{eq:mixing cocycle A infinity}, we get $-2\beta'(-1) = 0$, so that $\beta'(-1) = 0$. Now, substituting $n = 2$ and $m = -1$ in \eqref{eq:mixing cocycle A infinity} gives $\beta'(2) = \frac{3}{2}$, while taking $n = -2$ and $m = 1$ gives $\beta'(-2) = -\frac{1}{2}$.

    Finally, consider $n = 2$ and $m = -2$ in \eqref{eq:mixing cocycle A infinity}:
    $$(-2 + 2^2)\beta'(2) - (2 + 2^2)\beta'(-2) + 4\beta'(0) = 3 + 3 + 0 = 6 \neq 0,$$
    a contradiction. This proves the claim.

    Since $\beta'(1) = 0$, we can take $n = 1$ and $m = -1$ in \eqref{eq:mixing cocycle A infinity} to get $-2\beta'(-1) = 0$, so that $\beta'(-1) = 0$. Taking $n \geq 2$ and $m = -1$ in \eqref{eq:mixing cocycle A infinity}, we get
    $$n \beta'(n) = (n + 1)\beta'(n - 1).$$
    Therefore, it follows by induction that $\beta'(n) = 0$ for all $n \geq 2$. Similarly, if we take $n \leq -2$ and $m = 1$ in \eqref{eq:mixing cocycle A infinity}, we get
    $$n\beta'(n) = (n - 1)\beta'(n + 1),$$
    so we get $\beta'(n) = 0$ for all $n \leq -2$. We conclude that $\beta' = 0$, so $\beta = \beta(0) \beta_\infty$, as required.
\end{proof}

The final goal of this section is to compute $\HH_{\Mix}^2(\W_B(\lambda))$. Note that, unlike $\W_A(\lambda)$, the Lie algebra $\W_B(\lambda)$ does have nonzero mixing coboundary functions for all $\lambda \in \PP^1$.

We start by introducing notation for some mixing cocycle functions on $\W_B(\lambda)$ and stating the result on mixing cocycles of $\W_B(\lambda)$.

\begin{Notation}\label{ntt:mixing cocycles B}
    Let $\lambda \in \PP^1$. Define $\eta_\lambda, \gamma_1, \gamma_2 \colon \ZZ \to \CC$ as follows:
    $$\eta_\lambda(n) = n + n(n + 1)\lambda, \quad \gamma_1(n) = n, \quad \gamma_2(n) = n^2$$
    for all $n \in \ZZ$, where $n(n + 1)\lambda$ is understood to be $n^2$ if $\lambda = \infty$. In other words, $\eta_\lambda = (\lambda + 1) \gamma_1 + \lambda \gamma_2$ if $\lambda \neq \infty$ and $\eta_\infty = \gamma_1 + \gamma_2$.
\end{Notation}

\begin{proposition}\label{prop:W_B mixing cocycles}
    Let $\lambda \in \PP^1$. Then $\ZMix_B(\lambda) = \CC \gamma_1 \oplus \CC \gamma_2$, while $\BMix_B(\lambda) = \CC \eta_\lambda$, where $\gamma_1,\gamma_2$ and $\eta_\lambda$ are as in Notation \ref{ntt:mixing cocycles B}. In other words, if $\Omega$ is a mixing cocycle of degree 0 on $\W_B(\lambda)$, then
    $$\Omega(L_n,B_m) = (an^2 + bn)\delta_{n+m}^0$$
    for some $a,b \in \CC$. Furthermore, if $\lambda \neq \infty$, then $\Omega$ is a coboundary if and only if there exists $c \in \CC$ such that $a = c\lambda$ and $b = c(\lambda + 1)$, while if $\lambda = \infty$, then $\Omega$ is a coboundary if and only if $a = b$.
    Consequently,
    $$\dim(\HH_{\Mix}^2(\W_B(\lambda))) = \dim(\ZMix_B(\lambda)) - \dim(\BMix_B(\lambda)) = 1.$$
\end{proposition}

We first compute the space of mixing coboundary functions on $\W_B(\lambda)$.

\begin{lemma}
    Let $\lambda \in \PP^1$. Then $\BMix_B(\lambda) = \CC \eta_\lambda$, where $\eta_\lambda$ is as in Notation \ref{ntt:mixing cocycles B}. In particular, the mixing cocycle on $\W_B(\lambda)$ defined by
    $$\Omega(L_n,B_m) = (n + n(n + 1)\lambda)\delta_{n+m}^0$$
    is a coboundary.
\end{lemma}
\begin{proof}
    Let $\chi \in \BMix_B(\lambda)$. Let $\psi \colon \W_B(\lambda) \to \CC$ be a linear map giving rise to the mixing coboundary defined by $\chi$. In other words, $\psi \in \W_B(\lambda)^*$ has the property that
    $$d\psi(L_n,L_m) = d\psi(B_n,B_m) = 0, \quad d\psi(L_n,B_m) = \chi(n)\delta_{n+m}^0$$
    for all $n,m \in \ZZ$. This implies that, for all $n,m \in \ZZ$,
    $$0 = d\psi(L_n,L_m) = (n - m)\psi(L_{n+m}),$$
    so $\psi(L_n) = 0$ for all $n \in \ZZ$.
    
    If $n,m \in \ZZ$ with $m \neq -n$, we have
    $$0 = d\psi(L_n,B_m) = -\psi([L_n,B_m]) = -m\psi(B_{n+m}).$$
    Therefore, $\psi(B_n) = 0$ for all $n \in \ZZ \nonzero$. Hence, the only possibly nonzero value of $\psi$ is $\psi(B_0)$. Then, for all $n \in \ZZ$, we have
    $$\chi(n) = d\psi(L_n,B_{-n}) = -\psi([L_n,B_{-n}]) = (n + n(n + 1)\lambda)\psi(B_0) = \psi(B_0)\eta_\lambda(n),$$
    where, as before, $n(n + 1)\lambda$ is understood to be $n^2$ if $\lambda = \infty$. It follows that $\chi = \psi(B_0)\eta_\lambda$, which concludes the proof.
\end{proof}

We now compute $\ZMix_B(\lambda)$. For $\lambda \in \PP^1$ and $\beta \in \ZMix_B(\lambda)$, the cocycle condition \eqref{eq:W_X mixing cocycle cond} gives
\begin{equation}\label{eq:mixing cocycle B}
    (-(n + m) - m(m + 1)\lambda \delta_n^0)\beta(n) + (n + m + n(n + 1)\lambda \delta_m^0)\beta(m) + (n - m)\beta(n + m) = 0,
\end{equation}
for all $n,m \in \ZZ$, where once again $n(n + 1)\lambda$ is understood to be $n^2$ when $\lambda = \infty$.

Suppose $m = -n \in \ZZ \nonzero$. Then \eqref{eq:mixing cocycle B} becomes $2n\beta(0) = 0$, and thus $\beta(0) = 0$. It follows that $n(n + 1)\lambda \delta_m^0 \beta(m) = 0$ for all $n,m \in \ZZ$, so \eqref{eq:mixing cocycle B} becomes
\begin{equation}\label{eq:simple mixing cocycle B}
    (n + m)(\beta(m) - \beta(n)) + (n - m)\beta(n + m) = 0.
\end{equation}
Therefore, the space of mixing cocycle functions of $\W_B(\lambda)$ does not depend on $\lambda$. Given \eqref{eq:simple mixing cocycle B}, we immediately get two possible forms that $\beta$ can take.

\begin{lemma}\label{lem:obvious cocycles B}
    Let $\lambda \in \PP^1$. Then $\gamma_1, \gamma_2 \in \ZMix_B(\lambda)$. In other words, there are mixing cocycles $\Omega_1, \Omega_2 \in Z^2(\W_B(\lambda))$ defined by
    $$\Omega_1(L_n,B_m) = n\delta_{n+m}^0, \quad \Omega_2(L_n,B_m) = n^2\delta_{n+m}^0$$
    for all $n,m \in \ZZ$.
\end{lemma}
\begin{proof}
    Follows immediately from \eqref{eq:simple mixing cocycle B}.
\end{proof}

We now prove Proposition \ref{prop:W_B mixing cocycles} by showing that $\gamma_1$ and $\gamma_2$ span $\ZMix_B(\lambda)$. Recall that there is a linear combination of $\gamma_1$ and $\gamma_2$ which gives a coboundary function: we have $(\lambda + 1)\gamma_1 + \lambda \gamma_2 = \eta_\lambda \in \BMix_B(\lambda)$ for $\lambda \in \CC$, while $\gamma_1 + \gamma_2 = \eta_\infty \in \BMix_B(\infty)$.

\begin{proof}[Proof of Proposition \ref{prop:W_B mixing cocycles}]
    As mentioned above, it is enough to prove that $\ZMix_B(\lambda)$ is two-dimensional. We claim that $\gamma_1$ and $\gamma_2$ form a basis for $\ZMix_B(\lambda)$.
    
    Let $\beta \in \ZMix_B(\lambda)$ and define
    $$\beta' = \beta + \beta(1)(\gamma_2 - 2\gamma_1) - \frac{\beta(2)}{2}(\gamma_2 - \gamma_1),$$
    so $\beta' \in \ZMix_B(\lambda)$. Note that $\beta'(1) = \beta'(2) = 0$. Since $\beta'$ is a mixing cocycle function for $\W_B(\lambda)$, the function $\beta'$ satisfies \eqref{eq:simple mixing cocycle B}. Substituting $m = 1$ into \eqref{eq:simple mixing cocycle B} for $\beta'$, we get
    \begin{equation}\label{eq:beta' induction}
        (n - 1)\beta'(n + 1) = (n + 1)\beta'(n)
    \end{equation}
    for all $n \in \ZZ$. Since $\beta'(2) = 0$, applying \eqref{eq:beta' induction} inductively implies that $\beta'(n) = 0$ for all $n \in \NN$.
    
    Now, substituting $n = -1, m = 2$ into \eqref{eq:simple mixing cocycle B} for $\beta'$, we get $\beta'(-1) = 0$. Therefore, applying \eqref{eq:beta' induction} inductively, we get $\beta'(n) = 0$ for all $n \in \ZZ$. Therefore, $\beta' = 0$, which means that $\beta$ is a linear combination of $\gamma_1$ and $\gamma_2$, proving the claim.
\end{proof}

Combining Propositions \ref{prop:W_A mixing cocycles} and \ref{prop:W_B mixing cocycles}, we have proved Proposition \ref{prop:W_X mixing cocycles}. We now put together all our previous computations to prove Theorem \ref{thm:main}.

\begin{proof}[Proof of Theorem \ref{thm:main}]
    We will use the decomposition
    $$\HH^2(\W_X(\lambda)) = \HH_{\Vir}^2(\W_X(\lambda)) \oplus \HH_{\Ab}^2(\W_X(\lambda)) \oplus \HH_{\Mix}^2(\W_X(\lambda)).$$
    As stated before, $\HH_{\Vir}^2(\W_X(\lambda))$ is one-dimensional, spanned by the Virasoro cocycle $\overline{\Omega}_{\Vir}$.
    
    We now analyse $\W_A(\lambda)$. By Proposition \ref{prop:W_A has no abelian cocycles}, $\HH_{\Ab}^2(\W_A(\lambda)) = 0$. Note that, for $n,m \in \ZZ$, we have $\Omega_0^A(L_n,A_m) = \iota(n)\delta_{n+m}^0$ and $\Omega_{\Mix}^A(L_n,A_m) = \beta_\lambda(n)\delta_{n+m}^0$. Therefore, Proposition \ref{prop:W_A mixing cocycles} implies that
    $$\HH_{\Mix}^2(\W_A(\lambda)) =
    \begin{cases}
        \CC \overline{\Omega}_0^A \oplus \CC \overline{\Omega}_{\Mix}^A, &\text{if } \lambda=0.\\
        \CC \overline{\Omega}_{\Mix}^A, &\text{if } \lambda \neq 0.
    \end{cases}$$
    This completes the proof for $\W_A(\lambda)$. We now move on to $\W_B(\lambda)$.

    Note that $\Omega_{\Ab}^B(B_n,B_m) = \iota(n)\delta_{n+m}^0$, so Proposition \ref{prop:W_B abelian cocycles} implies that $\HH_{\Ab}^2(\W_B(\lambda)) = \CC \overline{\Omega}_{\Ab}^B$.

    For the mixing cocycle, we can take the 2-cocycle induced by the function
    \begin{align*}
        \theta_\lambda \colon \ZZ &\to \CC \\
        n &\mapsto \begin{cases}
            n, &\text{if } \lambda \neq 0, \\
            n(n + 1), &\text{if } \lambda = 0,
        \end{cases}
    \end{align*}
    as the representative in cohomology, since Proposition \ref{prop:W_B mixing cocycles} implies that $\theta_\lambda \notin \BMix_B(\lambda)$. Now, $\theta_\lambda$ is a mixing cocycle function which is not a coboundary function and
    $$\Omega_{\Mix}^B(L_n,B_m) = \theta_\lambda(n)\delta_{n+m}^0,$$
    so $\Omega_{\Mix}^B$ has nonzero image in $\HH^2(\W_B(\lambda))$. Since $\HH_{\Mix}^2(\W_B(\lambda))$ is one-dimensional by Proposition \ref{prop:W_B mixing cocycles}, it follows that $\HH_{\Mix}^2(\W_B(\lambda)) = \CC \overline{\Omega}_{\Mix}^B$, which finishes the proof.
\end{proof}

\section{Leibniz central extensions}\label{sec:Leibniz}

The study of Leibniz cohomology was pioneered by the works of Loday and Pirashvili \cite{LodayPirashvili, LodayBook} as a non-commutative analogue of Chevalley--Eilenberg cohomology. Leibniz cohomology can be defined for a more general class of objects called Leibniz algebras, of which Lie algebras are a special case. Analogously to Lie central extensions, the second Leibniz cohomology $\HL^2$ with coefficients in $\CC$ is in one-to-one correspondence with equivalence classes of one-dimensional Leibniz central extensions. We refer the reader to \cite{LodayPirashvili} and \cite{FeldvossWagemann} for a more thorough introduction to Leibniz algebras than the one presented here.

\begin{defn}
    A \emph{right Leibniz algebra} $\mathfrak{l}$ is a vector space with a bilinear map $[-,-] \colon \mathfrak{l} \times \mathfrak{l} \rightarrow \CC$ satisfying the Leibniz identity
    $$[[x,y],z] = [x,[y,z]] + [[x,z],y]$$
    for all $x,y,z \in \mathfrak{l}$. In other words, taking brackets on the right with an element $z \in \mathfrak{l}$ is a derivation of $\mathfrak{l}$.

    A \emph{Leibniz 2-cocycle} on $\mathfrak{l}$ is a bilinear map $\alpha \colon \mathfrak{l} \times \mathfrak{l} \to \CC$ such that
    \begin{equation}\label{eq:Leibniz cocycle}
        \alpha([x,y],z) = \alpha(x,[y,z]) + \alpha([x,z],y)
    \end{equation}
    for all $x,y,z \in \mathfrak{l}$.
\end{defn}

\begin{remark}
    An analogous definition exists for left Leibniz algebras, but we only consider right Leibniz algebras to maintain consistency with \cite{GaoJiangPei}. When the product is alternating (i.e. $[x,x] = 0$ for all $x \in \mathfrak{l}$), then the Leibniz identity is equivalent to the Jacobi identity and $\mathfrak{l}$ is a Lie algebra.

    We also note that if a Leibniz 2-cocycle is alternating (in other words, if $\alpha \in \Hom(\extp^2\mathfrak{l},\CC)$), then the Leibniz cocycle condition \eqref{eq:Leibniz cocycle} is equivalent to the Chevalley--Eilenberg cocycle condition of Lie algebra cohomology.
\end{remark}

In this section, we compute the inequivalent Leibniz central extensions of $\W_X(\lambda)$, given by the second Leibniz cohomology $\HL^2(\W_X (\lambda))$. The main result of this section is summarised in the following theorem.
\begin{theorem}\label{thm: second Leibniz cohomology}
    Let $\lambda \in \PP^1$. We have
    \begin{align*}
        \dim(\HL^2(\W_A(\lambda)))&= \dim(\HH^2(\W_A(\lambda)))+ 1 =\begin{cases}
            4, \quad \lambda = 0,\\
            3, \quad \lambda \neq 0.
        \end{cases}\\
        \dim(\HL^2(\W_B(\lambda)))&= \dim(\HH^2(\W_B(\lambda))) = 3,
    \end{align*}
    In particular,
    \begin{align*}
        \HL^2(\W_A(\lambda))&= \HH^2(\W_A(\lambda)) \oplus \CC \overline{\theta}_A \\
        \HL^2(\W_B(\lambda))&= \HH^2(\W_B(\lambda)),
    \end{align*}
    where $\overline{\theta}_A$ is the image of the Leibniz 2-cocycle defined by the symmetric, invariant bilinear form
    $$\theta_A(A_n,A_m)=\begin{cases}
        1, \quad n=m=0,\\
        0, \quad \text{otherwise}
    \end{cases}$$
    in $\HL^2(\W_A(\lambda))$.
\end{theorem}

We start by stating a useful result for computing Leibniz cohomology of Lie algebras.

\begin{proposition}[{\cite[Proposition 3.2]{HuPeiLiu}}]\label{prop:Leibniz InvForm exact sequence}
Let $\Inv(\g)$ be the space of symmetric invariant bilinear forms on a Lie algebra $\g$. Then there exists an exact sequence
    \begin{equation*}
        \begin{tikzcd}
            0\arrow[r] & \HH^2(\g) \arrow[r, hookrightarrow, "\iota"] & \HL^2(\g) \arrow[r, "\phi"] & \Inv(\g) \arrow[r, "h"] & \HH^3(\g),
        \end{tikzcd}
    \end{equation*}
    where $\phi(\chi)(x,y) = \chi(x,y) + \chi(y,x)$ and $h(\theta)(x,y,z) = \theta([x,y],z)$ for all $\chi \in \HL^2(\g)$, $\theta \in \Inv(\g)$, and $x,y \in \g$. The map $h$ is known as the \emph{Cartan--Koszul map}.
\end{proposition}

Thus, it suffices to look for symmetric invariant bilinear forms on $\W_X (\lambda)$ and then check which of these satisfy the Leibniz cocycle condition \eqref{eq:Leibniz cocycle}.

\begin{lemma}\label{lem: components of invariants forms on W_X()}
    Let $X \in \{A,B\}$ and $\lambda \in \PP^1$. Then, for all $\theta \in \Inv(\W_X(\lambda))$, we have
    \begin{align*}
        \theta(L_n, L_m) = 0 \quad &\text{for all } n,m \in \ZZ; \\
        \theta(L_n, X_m) = 0 \quad &\text{if } n + m \neq 0; \\
        \theta(X_n, X_m) = 0 \quad &\text{for all } (n,m) \in \ZZ^2 \setminus \{(0,0)\}.
    \end{align*}
\end{lemma}
\begin{proof}
    Let $n,m \in \ZZ$. We repeatedly use the invariance of $\theta$ on specific combinations of generators. First,
    $$\theta([L_0,L_n], L_m) = -\theta([L_n,L_0],L_m) = -\theta(L_n, [L_0,L_m]).$$
    Since $[L_0,L_n] = nL_n$ and $[L_0,L_m] = mL_m$, we get
    $$n \theta(L_n, L_m) = -m\theta(L_n, L_m).$$
    Therefore, $(n+m)\theta (L_n, L_m) = 0$, which implies that
    \begin{equation}\label{eq: InvForm L(n,m)=0 unless n+m=0}
        \theta(L_n, L_m) = 0 \quad \text{if } n + m \neq 0.
    \end{equation}
    On the other hand, $\theta([L_0,L_n], L_m) = \theta(L_0,[L_n, L_m])$, so
    \begin{equation}\label{eq: InvForm L(n,m) vs L(0,n+m)}
        n\theta(L_n,L_m) = (m-n)\theta(L_0, L_{m+n}) \quad \text{for all } m,n \in \ZZ.
    \end{equation}
    Letting $m = -n \neq 0$ in \eqref{eq: InvForm L(n,m) vs L(0,n+m)} gives
    \begin{equation}\label{eq: InvForm L(n,-n) = L(0,0)}
        \theta(L_n,L_{-n}) = -2\theta(L_0,L_0).
    \end{equation}
    Furthermore, 
    $$\theta([L_2,L_{-1}],L_{-1}) = -3\theta(L_1, L_{-1}) = 6\theta(L_0, L_0),$$
    where we used \eqref{eq: InvForm L(n,-n) = L(0,0)} for the last equality. On the other hand,
    $$\theta([L_2,L_{-1}],L_{-1}) = \theta(L_2,[L_{-1},L_{-1}]) = \theta(L_2,0) = 0.$$
    It follows that
    \begin{equation}
    \label{eq: InvForm L(0,0) = 0}
        \theta(L_0,L_0) = 0.
    \end{equation}
    Thus, \eqref{eq: InvForm L(n,m)=0 unless n+m=0}, \eqref{eq: InvForm L(n,-n) = L(0,0)} and \eqref{eq: InvForm L(0,0) = 0} imply that $\theta(L_n, L_m) = 0$ for all $n,m\in\ZZ$, as required.
    
    Next, proceeding as above and using the fact that $[L_0, X_m] = m X_m$ for both $X=A$ and $X=B$, we get $\theta(L_n, X_m) = 0$ if $n + m \neq 0$.
    
    Finally, we have $\theta([L_0,X_n],X_m) = \theta(L_0,[X_n, X_m]) = \theta(L_0,0) = 0$, and thus
    $$n\theta(X_n, X_m) = 0 \quad \text{for all } n,m \in \ZZ.$$
    Since $\theta$ is symmetric, this proves the last statement of the lemma, and thereby completes the proof.
\end{proof}

Lemma \ref{lem: components of invariants forms on W_X()} narrows down which components of any symmetric invariant bilinear form on $\W_X(\lambda)$ can be non-zero. 

\begin{proposition}\label{prop:InvForm}
Let $X \in \{A,B\}$ and $\lambda \in \PP^1$. We have
    $$\dim(\Inv(\W_X(\lambda)))=\begin{cases}
        1, \quad &\text{if } X = A, \\
        0, \quad &\text{if } X = B.
    \end{cases}$$
\end{proposition}
\begin{proof}
    Equipped with Lemma \ref{lem: components of invariants forms on W_X()}, we now treat $X = A$ and $X = B$ as separate cases.
    
    \case{1}{$X = A$.}
    
    For any $\theta \in \Inv(\W_A(\lambda))$ and $n \in \ZZ \nonzero$, we have
    $$n\theta(L_n, A_{-n}) = \theta([L_0, L_{n}], A_{-n}) = \theta(L_{0}, [L_n,A_{-n}]) = 0,$$
    where the last equality follows from the fact that $[L_n, A_{-n}] = 0$ for all $n \in \ZZ$. Thus,
    \begin{equation}\label{eq: InvForm A}
        \theta(L_n, A_{-n}) = 0 \quad \text{for all } n \in \ZZ \nonzero.
    \end{equation}
    Now, for all $n \in \ZZ\nonzero$,
    $$2\theta(L_0, A_0) = \theta([L_{-1}, L_1], A_{0}) = \theta(L_{-1}, [L_1,A_{0}]) = \omega_{A(\lambda)}(1,0) \theta(L_{-1}, A_1) = 0,$$
    where we used \eqref{eq: InvForm A} for the final equality. Thus $\theta(L_0,A_0) = 0$, so we conclude that $\theta(L_n, A_m) = 0$ for all $n,m \in \ZZ$, by Lemma \ref{lem: components of invariants forms on W_X()}.
    
    The only component that is unconstrained is $\theta(A_0, A_0)$. Hence, $\theta(A_0, A_0) = 1$ defines the unique symmetric invariant bilinear form on $\W_A (\lambda)$, up to rescaling. In other words, $\dim(\Inv(\W_A(\lambda))) = 1$.

    \case{2}{$X = B$.}
    
    Again, for any $\theta \in \Inv(\W_B(\lambda))$ and $n \in \ZZ \nonzero$, we have
    $$n\theta(L_n, B_{-n}) = \theta([L_0, L_{n}], B_{-n}) = \theta(L_{0}, [L_n,B_{-n}]) = \omega_{B(\lambda)}(n,-n)\theta(L_0,B_0),$$
    and thus, choosing $N \in \ZZ \nonzero$ such that $\omega_{B(\lambda)}(N,-N) = -(N + N(N + 1)\lambda) \neq 0$, we get
    \begin{equation}\label{eq: InvForm LB part 1}
        \theta(L_0, B_0) = \frac{N}{\omega_{B(\lambda)}(N,-N)} \theta(L_N, B_{-N}).
    \end{equation}
    Furthermore,
    $$-2n\theta(B_n, L_{-n}) = \theta([B_{2n},L_{-n}], L_{-n}) = \theta(B_{2n},[L_{-n},L_{-n}]) = 0,$$
    which implies that
    \begin{equation}\label{eq: InvForm LB part 2}
         \theta(B_n, L_{-n}) = 0 \quad \text{for all } n \in \ZZ \nonzero.
    \end{equation}
    Letting $n = -N$ in \eqref{eq: InvForm LB part 2}, we get $\theta(L_N, B_{-N}) = 0$, so \eqref{eq: InvForm LB part 1} becomes $\theta(L_0,B_0) = 0$. Hence, $\theta(L_n, B_m) = 0$ for all $n,m \in \ZZ$, by Lemma \ref{lem: components of invariants forms on W_X()}.
    
    Again, the only component left to determine is $\theta(B_0, B_0)$. Unlike the case where $X=A$, this component can be constrained because $B_0 \in [\W_B(\lambda),\W_B(\lambda)]$. As before, choose $N \in \ZZ \nonzero$ such that $\omega_{B(\lambda)}(N,-N) \neq 0$. Then $\theta([L_N, B_{-N}], B_0) = \omega_{B(\lambda)}(N,-N)\theta(B_0,B_0)$, so
    $$\theta(B_0,B_0) = \frac{\theta([L_N, B_{-N}], B_0)}{\omega_{B(\lambda)}(N,-N)} = \frac{\theta(L_N, [B_{-N}, B_0])}{\omega_{B(\lambda)}(N,-N)} = 0.$$
    Together with Lemma \ref{lem: components of invariants forms on W_X()}, we have shown that for any $\theta \in \Inv(\W_B(\lambda))$, we have
    $$\theta(L_n,L_m) = \theta(L_n, B_m) = \theta(B_n,B_m) = 0$$
    for all $n,m \in \ZZ$, meaning $\dim(\Inv(\W_B(\lambda))) = 0$, as required. This completes the proof.
\end{proof}
It follows form Propositions \ref{prop:Leibniz InvForm exact sequence} and \ref{prop:InvForm} that
    $$\begin{gathered}
        \HL^2(\W_B(\lambda)) = \HH^2(\W_B(\lambda)); \\
        \dim(\HH^2(\W_A(\lambda))) \leq \dim(\HL^2(\W_A(\lambda))) \leq \dim(\HH^2(\W_A(\lambda))) + 1.
    \end{gathered}$$
We now finish the proof of Theorem \ref{thm: second Leibniz cohomology} by constructing a nontrivial Leibniz 2-cocycle on $\W_A(\lambda)$ which maps to a nonzero element of $\Inv(\W_A(\lambda))$, proving that $\dim(\HL^2(\W_A(\lambda))) = \dim(\HH^2(\W_A(\lambda))) + 1$.

\begin{proof}[Proof of Theorem \ref{thm: second Leibniz cohomology}]
    As mentioned above, we have already shown the result for $X = B$, so assume that $X = A$.
    
    It is straightforward to check that $\theta_A$ satisfies the Leibniz cocycle condition \eqref{eq:Leibniz cocycle}. Since $\theta_A$ maps to a nonzero element of $\Inv(\W_A(\lambda))$ under the map $\phi$ of Proposition \ref{prop:Leibniz InvForm exact sequence}, it follows that $\theta_A$ gives an element of $\HL^2(\W_A(\lambda)) \setminus \HH^2(\W_A(\lambda))$. This completes the proof.
\end{proof}

\begin{remark}
    Theorem \ref{thm: second Leibniz cohomology} is consistent with the analogous results for $\W(0,1)$ and $\W(0,0)$ (see Theorem \ref{thm:GJP}), which are the $\lambda=0$ cases of $\W_A(\lambda)$ and $\W_B(\lambda)$, respectively. As pointed out in Remark \ref{rem:error}, our result for $\W(0,1)$ is different to that of \cite[Theorem 3.5]{GaoJiangPei}. This is because the map $f^{0,1}$ from \cite[Lemma 3.2]{GaoJiangPei} is not an invariant bilinear form, which can be seen in the proof of Proposition \ref{prop:InvForm} for $X = A$ and $\lambda = 0$.
\end{remark}

\section{Automorphisms}\label{sec:automorphisms}

In this section, we compute the automorphism groups of the Lie algebras $\W_X(\lambda)$. We start by defining some automorphisms.

\begin{Notation}\label{ntt:automorphisms}
    Let $X \in \{A,B\}$ and $\lambda \in \PP^1$. For $\alpha \in \CC^\times$, define $\sigma_\alpha \in \Aut(\W_X(\lambda))$ as follows:
    $$\sigma_\alpha(L_n) = \alpha^n L_n, \quad \sigma_\alpha(X_n) = \alpha^n X_n.$$
    Define $\tau_0 \in \Aut(\W_X(0))$ as follows:
    $$\tau_0(L_n) = -L_{-n}, \quad \tau_0(X_n) = X_{-n}.$$
    Define $\tau_{-1} \in \Aut(\W_X(-1))$ as follows:
    $$\tau_{-1}(L_n) = -L_{-n}, \quad \tau_{-1}(X_n) = \begin{cases}
        -X_0, &\text{if } n = 0, \\
        X_{-n}, &\text{if } n \neq 0.
    \end{cases}$$
    For $\xi \in \CC^\times$, we define $\mu_\xi \in \Aut(\W_X(\lambda))$ as follows:
    $$\mu_\xi(L_n) = L_n, \quad \mu_\xi(X_n) = \xi X_n.$$
    For $X = A$ and $a \in \CC$, we define $\varphi^{A(\lambda)}_a, \psi^{A(\lambda)}_a \in \Aut(\W_A(\lambda))$ as follows:
    $$\begin{gathered}
        \varphi^{A(\lambda)}_a(L_n) = L_n + an^2 A_n, \quad \varphi^{A(\lambda)}_a(A_n) = A_n; \\
        \psi^{A(\lambda)}_a(L_n) = L_n + an A_n, \quad \psi^{A(\lambda)}_a(A_n) = A_n.
    \end{gathered}$$
    For $X = B$ and $a \in \CC$, we define $\varphi^{B(\lambda)}_a \in \Aut(\W_B(\lambda))$ and $\psi^{B(0)}_a \in \Aut(\W_B(0))$ as follows:
    $$\begin{gathered}
        \varphi^{B(\lambda)}_a(L_n) = \begin{cases}
            L_0 + (\lambda + 1)a B_0, &\text{if } n = 0 \text{ and } \lambda \neq \infty, \\
            L_n + a B_n, &\text{otherwise},
        \end{cases}
    \quad \varphi^{B(\lambda)}_a(B_n) = B_n; \\
    \psi^{B(0)}_a(L_n) = L_n + an B_n, \quad \psi^{B(0)}_a(B_n) = B_n.
    \end{gathered}$$
    Finally, letting $k \in \ZZ/2\ZZ$, $a,b \in \CC$, and $\alpha, \xi \in \CC^\times$, we define
    \begin{equation}\label{eq:outer automorphism general form}
        \Phi_{k;a,b;\alpha,\xi}^{X(\lambda)} = \tau_\lambda^k \circ \varphi_a^{X(\lambda)} \circ \psi_b^{X(\lambda)} \circ \sigma_\alpha \circ \mu_\xi,
    \end{equation}
    where we let $\tau_\lambda = \id_{\W_X(\lambda)}$ if $\lambda \notin \{0,-1\}$ and $\psi_b^{B(\lambda)} = \id_{\W_B(\lambda)}$ if $\lambda \neq 0$.
\end{Notation}

Although it is not entirely obvious from the definitions that the linear maps defined above are automorphisms, it will become clear that they preserve the Lie brackets in the proof of Proposition \ref{prop:outer automorphisms}.

Next, we consider inner automorphisms of $\W_X(\lambda)$, defined below.

\begin{defn}
    Let $\g$ be a Lie algebra and let $x \in \g$ such that $\ad_x$ is locally nilpotent (in other words, for all $y \in \g$, there exists $n \in \NN$ such that $\ad_x^n(y) = 0$). The \emph{exponential} $\exp(\ad_x) \in \Aut(\g)$ of $\ad_x$ is defined by
    $$\exp(\ad_x) = \sum_{n = 0}^\infty \frac{1}{n!}\ad_x^n = \id_\g + \sum_{n = 1}^\infty \frac{1}{n!}\ad_x^n.$$
    The subgroup of $\Aut(\g)$ generated by $\{\exp(\ad_x) \mid x \in \g \text{ such that } \ad_x \text{ is locally nilpotent}\}$ is called the group of \emph{inner automorphisms} of $\g$, denoted $\InnAut(\g)$.
\end{defn}

Given a Lie algebra $\g$, the group of inner automorphisms of $\g$ is a normal subgroup of $\Aut(\g)$. This is because
$$\sigma \circ \exp(\ad_x) \circ \sigma^{-1}(y) = \sigma\left(\sum_{n=0}^\infty \ad_x^n(\sigma^{-1}(y))\right) = \sum_{n=0}^\infty \ad_{\sigma(x)}^n(y) = \exp(\ad_{\sigma(x)})(y)$$
for all $x,y \in \g$ and $\sigma \in \Aut(\g)$, so $\sigma \circ \exp(\ad_x) \circ \sigma^{-1} = \exp(\ad_{\sigma(x)}) \in \InnAut(\g)$.

\begin{remark}\label{rem:exp(L0)}
    Even though we only defined the exponential of a locally nilpotent adjoint map, we can sometimes define exponentials of adjoint maps even when they are not locally nilpotent. For example, this is the case for $\ad_{L_0}$, which is not locally nilpotent, but if $t \in \CC$, then $\exp(t \ad_{L_0}) = \sigma_{e^t}$, where $\sigma_\alpha$ is defined in Notation \ref{ntt:automorphisms}.

    On the other hand, $\exp(\ad_{L_n})$ is not defined for $n \in \ZZ \nonzero$, because the infinite sum
    $$\exp(\ad_{L_n})(L_m) = \sum_{k=0}^\infty \frac{1}{k!}\ad_{L_n}^k(L_m)$$
    is not defined for any $m \in \ZZ \setminus \{n\}$.
\end{remark}

We now compute $\InnAut(\W_X(\lambda))$.

\begin{lemma}\label{lem:inner automorphisms}
    Let $X \in \{A,B\}$ and $\lambda \in \PP^1$. Then the locally nilpotent adjoint maps of $\W_X(\lambda)$ are $\ad_w$ for $w \in X(\lambda)$. Therefore, $\InnAut(\W_X(\lambda))$ is generated by
    $$\{\exp(c \ad_{X_n}) \mid n \in \ZZ, c \in \CC\}.$$
    Letting $c \in \CC$ and $n,m \in \ZZ$ the action of $\InnAut(\W_X(\lambda))$ on $\W_X(\lambda)$ is given by
    \begin{align*}
        \restr{\exp(c\ad_{X_n})}{X(\lambda)} &= \id_{X(\lambda)}, \\
        \exp(c\ad_{X_n})(L_m) &= L_m - c\omega_{X(\lambda)}(m,n)X_{n+m} \\
        &= \begin{cases}
            L_m - c(n + m + m(m + 1)\lambda \delta_n^0)A_{n+m}, &\text{if } X = A,\\
            L_m - c(n - m(m + 1)\lambda \delta_{n+m}^0)B_{n+m}, &\text{if } X = B.
        \end{cases}
    \end{align*}
\end{lemma}
\begin{proof}
    Let $w \in \W_X(\lambda)$ such that $\ad_w$ is locally nilpotent. Then $w = \sum_{i \in \ZZ} (\alpha_i L_i + \beta_i X_i)$ for some $\alpha_i, \beta_i \in \CC$. It is clear that if $\alpha_i \neq 0$ for some $i \in \ZZ$, then $\ad_w$ is not locally nilpotent. Thus, $w \in X(\lambda)$.
    
    Conversely, let $w \in X(\lambda)$. Since $X(\lambda)$ is an ideal of $\W_X(\lambda)$, we have $\ad_w(\W_X(\lambda)) \subseteq X(\lambda)$. Now, $X(\lambda)$ is an abelian Lie algebra, so $\ad_w^2$ is the zero map on $\W_X(\lambda)$. It follows that $\ad_w$ is locally nilpotent, as required.

    For the final sentence, note that
    $$\exp(c \ad_{X_n}) = \id_{\W_X(\lambda)} + c \ad_{X_n}$$
    for all $n \in \ZZ$, since all the higher powers of $\ad_{X_n}$ are zero. The result now follows by the definition of the bracket on $\W_X(\lambda)$.
\end{proof}

By Lemma \ref{lem:inner automorphisms}, the action of an arbitrary element of $\InnAut(\W_X(\lambda))$ on $\W_X(\lambda)$ is given by
$$\prod_{i \in I}\exp(c_i\ad_{X_i})(L_n) = L_n - \sum_{i \in I} c_i\omega_{X(\lambda)}(n,i)X_{n+i},$$
where $I$ is a finite subset of $\ZZ$ and $c_i \in \CC$.

Furthermore, since $X(\lambda)$ is an abelian subalgebra of $\W_X(\lambda)$, it follows from Lemma \ref{lem:inner automorphisms} that $\InnAut(\W_X(\lambda))$ is an abelian normal subgroup of $\Aut(\W_X(\lambda))$. We now compute the group structure of the group of inner automorphisms of $\W_X(\lambda)$.

\begin{lemma}\label{lem:group structure of inner auomorphisms}
    Let $X \in \{A,B\}$ and $\lambda \in \PP^1$. Then $\InnAut(\W_X(\lambda)) \cong \CC^{\infty}$, where
    \begin{equation}\label{eq:Cinfty}
        \CC^\infty \coloneqq \{(z_i)_{i \in \ZZ} \mid z_i \in \CC, \text{ all but finitely many } z_i \text{ are zero}\}
    \end{equation}
    is the additive group of finite $\ZZ$-sequences over $\CC$.
\end{lemma}
\begin{proof}
    Let $e_n$ be the element of $\CC^\infty$ with zeros everywhere except for a one in the $n^\text{th}$ coordinate. The map
    \begin{align*}
        \InnAut(\W_A(\lambda)) &\to \CC^\infty \\
        \exp(c \ad_{A_n}) &\mapsto c e_n
    \end{align*}
    is easily seen to be an isomorphism in the case $X = A$.
    
    In the case $X = B$, we need a slightly different argument, since $B_0$ is central, so $\exp(c \ad_{B_0}) = \id_{\W_B(\lambda)}$ for all $c \in \CC$. Therefore, the map
    \begin{align*}
        \InnAut(\W_B(\lambda)) &\to \CC^\infty \\
        \exp(c \ad_{B_n}) &\mapsto c e_n \quad (n \in \ZZ \nonzero)
    \end{align*}
    is not an isomorphism, since $e_0$ is not in the image. However, the image of the map is the group
    \begin{equation}\label{eq:C0infty}
        \CC_0^\infty \coloneqq \{(z_i)_{i \in \ZZ} \mid z_i \in \CC, \text{ all but finitely many } z_i \text{ are zero}, z_0 = 0\},
    \end{equation}
    which is easily seen to be isomorphic to $\CC^\infty$.
\end{proof}

\begin{lemma}\label{lem:linear combination is inner}
    Let $\lambda \in \PP^1$. Then, given $a \in \CC$, we have $\lambda \varphi_a^{A(\lambda)} + (\lambda + 1)\psi_a^{A(\lambda)} \in \InnAut(\W_A(\lambda))$ if $\lambda \neq \infty$, while $\varphi_a^{A(\infty)} + \psi_a^{A(\infty)} \in \InnAut(\W_A(\infty))$ if $\lambda = \infty$. Defining
    \begin{equation}\label{eq:special automorphism A}
        \varpi_a^{A(\lambda)} \coloneqq \begin{cases}
            \psi_a^{A(\lambda)}, &\text{if } \lambda \neq 0, \\
            \varphi_a^{A(0)}, &\text{if } \lambda = 0,
        \end{cases}
    \end{equation}
    then $\varpi_a^{A(\lambda)} \notin \InnAut(\W_A(\lambda))$ for all $\lambda \in \PP^1$.

    Furthermore, none of the automorphisms of $\W_B(\lambda)$ from Notation \ref{ntt:automorphisms} are inner.
\end{lemma}
\begin{proof}
    Given $a \in \CC$, we have
    $$\exp(-a \ad_{A_0})(L_n) = L_n + a(n + n(n + 1)\lambda)A_n$$
    for all $n \in \ZZ$. Therefore, $\lambda \varphi_a^{A(\lambda)} + (\lambda + 1)\psi_a^{A(\lambda)} = \exp(-a \ad_{A_0})$ if $\lambda \neq \infty$, while $\varphi_a^{A(\infty)} + \psi_a^{A(\infty)} = \exp(-a \ad_{A_0})$ if $\lambda = \infty$. Since $\ad_{A_0}$ is the unique (up to multiplication by a scalar) locally nilpotent adjoint map of $\W_A(\lambda)$, it follows that $\varpi_a^{A(\lambda)}$ is not inner for all $\lambda \in \PP^1$.

    On the other hand, $B_0$ is central in $\W_B(\lambda)$, so $\exp(a \ad_{B_0}) = \id_{\W_B(\lambda)}$ for all $\lambda \in \CC$ and $a \in \CC$. Therefore, there are no inner derivations of $\W_B(\lambda)$ of degree zero, so the result follows.
\end{proof}

Despite the results of Lemma \ref{lem:linear combination is inner}, it will still be useful to consider both $\varphi_a^{A(\lambda)}$ and $\psi_a^{A(\lambda)}$.

Given the observations of Lemma \ref{lem:linear combination is inner}, we introduce the following notation.

\begin{Notation}\label{ntt:outer automorphism group}
    Let $X \in \{A,B\}$ and $\lambda \in \PP^1$. Define
    \begin{equation}\label{eq:outer automorphism general form A}
        \Psi_{k;a;\alpha,\xi}^{A(\lambda)} = \tau_\lambda^k \circ \varpi_a^{A(\lambda)} \circ \sigma_\alpha \circ \mu_\xi,
    \end{equation}
    where $k \in \ZZ/2\ZZ$, $a \in \CC$, and $\alpha, \xi \in \CC^\times$, and as before, we let $\tau_\lambda = \id_{\W_X(\lambda)}$ if $\lambda \notin \{0,-1\}$. In other words,
    $$\Psi_{k;a;\alpha,\xi}^{A(\lambda)} = \begin{cases}
        \Phi_{k;0,a;\alpha,\xi}^{A(\lambda)}, &\text{if } \lambda \neq 0, \\
        \Phi_{k;a,0;\alpha,\xi}^{A(\lambda)}, &\text{if } \lambda = 0.
    \end{cases}$$
    Define 
    $$G_X(\lambda) = \begin{cases}
        \{\Psi_{k;a;\alpha,\xi}^{A(\lambda)} \mid k \in \ZZ/2\ZZ, a \in \CC, \alpha, \xi \in \CC^\times\}, &\text{if } X = A, \\
        \{\Phi_{k;a,b;\alpha,\xi}^{B(\lambda)} \mid k \in \ZZ/2\ZZ, a,b \in \CC, \alpha, \xi \in \CC^\times\}, &\text{if } X = B,
    \end{cases}$$
    where $\Psi_{k;a,b;\alpha,\xi}^{A(\lambda)}$ is defined in \eqref{eq:outer automorphism general form A}, and $\Phi_{k;a,b;\alpha,\xi}^{B(\lambda)}$ is defined in \eqref{eq:outer automorphism general form}.
\end{Notation}

Note that, by Lemma \ref{lem:inner automorphisms}, we have $\InnAut(\W_X(\lambda)) \cap G_X(\lambda) = \{\id_{\W_X(\lambda)}\}$ for all $X \in \{A,B\}$ and $\lambda \in \PP^1$.

Having introduced all the necessary notation, we now state the main result of this section, which computes the group structure of $\Aut(\W_X(\lambda))$.

\begin{theorem}\label{thm:automorphisms}
    Let $X \in \{A,B\}$ and $\lambda \in \PP^1$. Then $G_X(\lambda)$ from Notation \ref{ntt:outer automorphism group} is a subgroup of $\Aut(\W_X(\lambda))$ and
    $$\Aut(\W_X(\lambda)) = \InnAut(\W_X(\lambda)) \rtimes G_X(\lambda).$$
    Consequently,
    $$\Aut(\W_X(\lambda)) \cong \begin{cases}
        \CC^\infty \rtimes (\ZZ/2\ZZ \ltimes (\CC \rtimes (\CC^\times \times \CC^\times))), &\text{if } X = A \text{ and } \lambda \in \{0,-1\}, \\
        \CC^\infty_0 \rtimes (\ZZ/2\ZZ \ltimes (\CC^2 \rtimes (\CC^\times \times \CC^\times))), &\text{if } X = B \text{ and } \lambda = 0, \\
        \CC^\infty_0 \rtimes (\ZZ/2\ZZ \ltimes (\CC \rtimes (\CC^\times \times \CC^\times))), &\text{if } X = B \text{ and } \lambda = -1, \\
        \CC^\infty \rtimes (\CC \rtimes (\CC^\times \times \CC^\times)), &\text{if } X = A \text{ and } \lambda \notin \{0,-1\}, \\
        \CC^\infty_0 \rtimes (\CC \rtimes (\CC^\times \times \CC^\times)), &\text{if } X = B \text{ and } \lambda \notin \{0,-1\},
    \end{cases}$$
    where $\CC^\infty$ and $\CC_0^\infty$ are defined in \eqref{eq:Cinfty} and \eqref{eq:C0infty}. Given $k,\ell \in \ZZ/2\ZZ$, $a,b,c,d,y_i,z_i \in \CC$ for $i \in \ZZ$ with only finitely many $y_i$ and $z_i$ being nonzero, and $\alpha,\beta,\xi,\zeta \in \CC^\times$, the structures of the groups above are as follows:
    \begin{align*}
        \W_B(0) &\colon &((y_i),k,a,b,\alpha,\xi) \cdot ((z_i),
        \ell,c,d,\beta,\zeta) &= ((y_i + \alpha^{\varepsilon_1 i} \xi z_{\varepsilon_1 i}), k + \ell, \varepsilon_2 a + \xi c, b + \xi d, \alpha^{\varepsilon_2} \beta, \xi \zeta); \\
        \W_A(0), \W_B(-1) &\colon &((y_i),k,a,\alpha,\xi) \cdot ((z_i),
        \ell,c,\beta,\zeta) &= ((y_i + \alpha^{\varepsilon_1 i} \xi z_{\varepsilon_1 i}), k + \ell, \varepsilon_2 a + \xi c, \alpha^{\varepsilon_2} \beta, \xi \zeta); \\
        \W_A(-1) &\colon &((y_i),k,b,\alpha,\xi) \cdot ((z_i),\ell,d,\beta,\zeta) &= ((y_i + \varepsilon_1^{\delta_i^0} \alpha^{\varepsilon_1 i} \xi z_{\varepsilon_1 i}), k + \ell, b + \xi d, \alpha^{\varepsilon_2} \beta, \xi \zeta); \\
        \W_X(\lambda) &\colon &((y_i),a,\alpha,\xi) \cdot ((z_i),c,\beta,\xi) &= ((y_i + \alpha^i \xi z_i), a + \xi c, \alpha \beta, \xi \zeta);
    \end{align*}
    where $\varepsilon_1 = (-1)^k, \varepsilon_2 = (-1)^\ell$, and $\lambda \notin \{0,-1\}$, and we assume that $y_0 = z_0 = 0$ if $X = B$.
\end{theorem}

\begin{remark}
    Certainly, we have $\CC^\infty \times \CC \cong \CC_0^\infty \times \CC^2 \cong \CC^\infty_0 \times \CC \cong \CC^\infty$. We have chosen to leave these products of groups in the unsimplified forms to make the group structure clearer. Indeed, the automorphism groups of Theorem \ref{thm:automorphisms} are all isomorphic to $\ZZ/2\ZZ \ltimes (\CC^\infty \rtimes (\CC^\times \times \CC^\times))$ or $\CC^\infty \rtimes (\CC^\times \times \CC^\times)$ with appropriate semi-direct product structures, but writing the groups this way makes the group structure less clear.
\end{remark}

We have already considered inner automorphisms, so we now move on to computing arbitrary automorphisms of $\W_X(\lambda)$ to complete the proof of Theorem \ref{thm:automorphisms}. We begin by analysing the action of an automorphism of $\W_X(\lambda)$ on the abelian ideal $X(\lambda)$ of $\W_X(\lambda)$.

\begin{lemma}\label{lem:automorphism preserves X}
    Let $X \in \{A,B\}$, $\lambda \in \PP^1$, and $\sigma \in \Aut(\W_X(\lambda))$. Then $\sigma(X(\lambda)) \subseteq X(\lambda)$.
\end{lemma}
\begin{proof}
    Note that $\sigma(X(\lambda))$ is an infinite-dimensional abelian ideal of $\W_X(\lambda)$, and therefore must be contained in $X(\lambda)$.
\end{proof}

The next result shows that we can exploit the structure of the automorphism group of the Witt algebra to get information about automorphisms of $\W_X(\lambda)$.

\begin{lemma}\label{lem:induced Witt automorphism}
    Let $X \in \{A,B\}$, $\lambda \in \PP^1$, and $\sigma \in \Aut(\W_X(\lambda))$. Write $\overline{\sigma}$ for the composition
    $$\W \xrightarrow{\sigma} \W_X(\lambda) \twoheadrightarrow \W_X(\lambda)/X(\lambda) \cong \W.$$
    Then $\overline{\sigma} \in \Aut(\W)$.
\end{lemma}
\begin{proof}
    It is straightforward to check that $\overline{\sigma^{-1}}$ is the inverse of $\overline{\sigma}$.
\end{proof}

Given Lemma \ref{lem:induced Witt automorphism}, it will be useful to consider automorphisms of the Witt algebra. These are well-known.

\begin{proposition}[{\cite[Theorem 1.1]{BavulaWitt}}]\label{prop:automorphisms Witt}
    For $\alpha \in \CC^\times$, define automorphisms $\sigma_\alpha, \tau \colon \W \to \W$ by
    $$\sigma_\alpha(L_n) = \alpha^n L_n, \quad \tau(L_n) = -L_{-n},$$
    for $n \in \ZZ$. Then $\Aut(\W) = \{\sigma_\alpha \mid \alpha \in \CC^\times\} \rtimes \{\id_\W, \tau\} \cong \CC^\times \rtimes \ZZ/2\ZZ$.
\end{proposition}

The next result shows that, after composing with an inner automorphism, we can greatly simplify the form which an automorphism of $\W_X(\lambda)$ can take.

\begin{proposition}\label{prop:automorphism nice form}
    Let $X \in \{A,B\}$, $\lambda \in \PP^1$, and $\sigma \in \Aut(\W_X(\lambda))$. Then there exists $\rho \in \InnAut(\W_X(\lambda))$ such that, for all $n \in \ZZ$,
    $$(\rho \circ \sigma)(L_n) = \varepsilon \alpha^n L_{\varepsilon n} + \gamma_n X_{\varepsilon n}, \quad (\rho \circ \sigma)(X_n) = \xi_n X_{\varepsilon n},$$
    for some $\varepsilon \in \{\pm 1\}$, $\alpha, \xi_n \in \CC^\times$, $\gamma_n \in \CC$. Furthermore, if $X = A$, then $\gamma_0 = 0$.
\end{proposition}
\begin{proof}
    By Lemma \ref{lem:induced Witt automorphism} and Proposition \ref{prop:automorphisms Witt}, we have $\overline{\sigma} = \tau^k \circ \sigma_{\alpha}$ for some $\alpha \in \CC^\times$ and $k \in \{0,1\}$ (where we write $\overline{\sigma}$ as in the notation of Lemma \ref{lem:induced Witt automorphism}). Letting $\varepsilon = (-1)^k$, we therefore see that
    \begin{equation}\label{eq:automorphism general form}
        \sigma(L_n) = \varepsilon \alpha^n L_{\varepsilon n} + \sum_{i \in \ZZ} \beta_i^{(n)} X_i,
    \end{equation}
    for some $\beta_i^{(n)} \in \CC$, where, for each $n \in \ZZ$, only finitely many $\beta_i^{(n)}$ are nonzero.

    Let $\rho = \prod_{i \in \ZZ \nonzero} \exp(\frac{\varepsilon \beta_i^{(0)}}{i} \ad_{X_i}) \in \InnAut(\W_X(\lambda))$. By Lemma \ref{lem:inner automorphisms}, it follows that $\rho(X_n) = X_n$ for all $n \in \ZZ$ and
    \begin{equation}\label{eq:rho}
        \rho(L_0) = L_0 - \varepsilon \sum_{i\in \ZZ \nonzero} \beta_i^{(0)}X_i,
    \end{equation}
    since $[L_0,X_n] = nX_n$ for both $X = A$ and $X = B$. Now, \eqref{eq:automorphism general form} and \eqref{eq:rho} imply that
    \begin{align*}
        \rho(\sigma(L_0)) &= \varepsilon \rho(L_0) + \sum_{i \in \ZZ} \beta_i^{(n)} X_i \\
        &= \varepsilon L_0 - \sum_{i\in \ZZ \nonzero} \beta_i^{(0)}X_i + \sum_{i \in \ZZ} \beta_i^{(n)} X_i \\
        &= \varepsilon L_0 + \beta_0^{(0)} X_0.
    \end{align*}
    Letting $\sigma' = \rho \circ \sigma$, we can proceed as before to conclude that
    \begin{equation}\label{eq:sigma'}
        \sigma'(L_n) = \varepsilon \alpha^n L_{\varepsilon n} + \sum_{i \in \ZZ} \gamma_i^{(n)} X_i,
    \end{equation}
    for some $\gamma_i^{(n)} \in \CC$, where, for each $n \in \ZZ$, only finitely many $\gamma_i^{(n)}$ are nonzero. By construction, we have $\gamma_0^{(0)} = \beta_0^{(0)}$ and $\gamma_i^{(0)} = 0$ for $i \in \ZZ \nonzero$.

    We have $n\sigma'(L_n) = \sigma'([L_0,L_n]) = [\sigma'(L_0),\sigma'(L_n)]$ and thus, using \eqref{eq:sigma'} to expand,
    \begin{equation}\label{eq:sigma'([L0,Ln])}
        \varepsilon n \alpha^n L_{\varepsilon n} + n \sum_{i \in \ZZ} \gamma_i^{(n)} X_i = \varepsilon n \alpha^n L_{\varepsilon n} + \varepsilon \sum_{i \in \ZZ} i \gamma_i^{(n)} X_i - \varepsilon \alpha^n \gamma_0^{(0)} \omega_{X(\lambda)}(\varepsilon n,0) X_{\varepsilon n}.
    \end{equation}
    It follows that, for $m \in \ZZ \setminus \{\varepsilon n\}$, we have $n \gamma_m^{(n)} = \varepsilon m \gamma_m^{(n)}$, so $(m - \varepsilon n)\gamma_m^{(n)} = 0$. Therefore, $\gamma_m^{(n)} = 0$ for all $m \in \ZZ \setminus \{\varepsilon n\}$. Furthermore, comparing coefficients of $X_{\varepsilon n}$ in \eqref{eq:sigma'([L0,Ln])}, we see that
    \begin{equation}\label{eq:automorphism gamma0}
        \omega_{X(\lambda)}(n,0) \gamma_0^{(0)} = 0
    \end{equation}
    for all $n \in \ZZ$. Note that $\omega_{B(\lambda)}(n,0) = 0$ for all $n \in \ZZ$, so \eqref{eq:automorphism gamma0} does not give anything in the case $X = B$. On the other hand, if $X = A$, then \eqref{eq:automorphism gamma0} implies that $\gamma_0^{(0)} = 0$.
    
    Letting $\gamma_n = \gamma_{\varepsilon n}^{(n)}$, we have shown that
    \begin{equation}\label{eq:sigma' L}
        \sigma'(L_n) = \varepsilon \alpha^n L_{\varepsilon n} + \gamma_n X_{\varepsilon n}.
    \end{equation}
    Now, by Lemma \ref{lem:automorphism preserves X}, we have
    \begin{equation}\label{eq:sigma' X}
        \sigma'(X_n) = \sum_{i \in \ZZ} \xi_i^{(n)} X_i,
    \end{equation}
    for some $\xi_i^{(n)} \in \CC$, where, for each $n \in \ZZ$, only finitely many $\xi_i^{(n)}$ are nonzero. Since $[L_0,X_n] = nX_n$, we have $[\sigma'(L_0),\sigma'(X_n)] = n\sigma'(X_n)$. Therefore, it follows from \eqref{eq:sigma' L} and \eqref{eq:sigma' X} that
    $$\varepsilon \sum_{i \in \ZZ} i \xi_i^{(n)} X_i = n\sum_{i \in \ZZ} \xi_i^{(n)} X_i.$$
    Hence, $(m - \varepsilon n)\xi_m^{(n)} = 0$ for all $m \in \ZZ$, so $\xi_m^{(n)} = 0$ for all $m \in \ZZ \setminus \{\varepsilon n\}$. Thus, letting $\xi_n = \xi_{\varepsilon n}^{(n)}$, we have
    $$\sigma'(X_n) = \xi_n X_{\varepsilon n}.$$
    This concludes the proof.
\end{proof}

We now show that, together with inner automorphisms, the automorphisms defined in Notation \ref{ntt:automorphisms} generate the whole automorphism group of $\W_X(\lambda)$.

\begin{proposition}\label{prop:outer automorphisms}
    Let $X \in \{A,B\}$, $\lambda \in \PP^1$, and $\sigma \in \Aut(\W_X(\lambda))$. Then there exists $\rho \in \InnAut(\W_X(\lambda))$ such that $\rho \circ \sigma = \Phi_{k;a,b;\alpha,\xi}^{X(\lambda)}$ for some $k \in \ZZ/2\ZZ$, $a,b \in \CC$, and $\alpha, \xi \in \CC^\times$, where $\Phi_{k;a,b;\alpha,\xi}^{X(\lambda)}$ is defined in \eqref{eq:outer automorphism general form}.

\end{proposition}
\begin{proof}
    Let $\sigma \in \Aut(\W_X(\lambda))$. By Proposition \ref{prop:automorphism nice form}, possibly by replacing $\sigma$ with $\rho \circ \sigma$ for some $\rho \in \InnAut(\W_X(\lambda))$, we may assume that
    \begin{equation}\label{eq:sigma nice form}
        \sigma(L_n) = \alpha^n (\varepsilon L_{\varepsilon n} + \gamma_n X_{\varepsilon n}), \quad \sigma(X_n) = \alpha^n \xi_n X_{\varepsilon n},
    \end{equation}
    for some $\varepsilon \in \{\pm 1\}$, $\alpha, \xi_n \in \CC^\times$, $\gamma_n \in \CC$, with $\gamma_0 = 0$ if $X = A$.

    Since $[L_n,L_m] = (m - n)L_{n+m}$, we have $[\sigma(L_n),\sigma(L_m)] = (m - n)\sigma(L_{n+m})$. Therefore, applying \eqref{eq:sigma nice form} and comparing coefficients of $X_{\varepsilon n}$, we deduce that
    \begin{equation}\label{eq:automorphism gamma}
        \varepsilon(\omega_{X(\lambda)}(\varepsilon n, \varepsilon m)\gamma_m - \omega_{X(\lambda)}(\varepsilon m, \varepsilon n)\gamma_n) = (m - n)\gamma_{n+m}
    \end{equation}
    for all $n,m \in \ZZ$. Similarly, we have $[\sigma(L_n),\sigma(X_m)] = \omega_{X(\lambda)}(n,m)\sigma(X_{n+m})$, so we can apply \eqref{eq:sigma nice form} and simplify to get
    \begin{equation}\label{eq:automorphism xi}
        \varepsilon \omega_{X(\lambda)}(\varepsilon n, \varepsilon m)\xi_m = \omega_{X(\lambda)}(n,m)\xi_{n+m}
    \end{equation}
    for all $n,m \in \ZZ$. Furthermore, any bijective linear map $\sigma$ defined as in \eqref{eq:sigma nice form} which satisfies \eqref{eq:automorphism gamma} and \eqref{eq:automorphism xi} is an automorphism of $\W_X(\lambda)$.
    
    We now consider the cases $X = A$ and $X = B$ separately.

    \case{1}{$X = A$.}
    
    Recall that in this case, we have $\gamma_0 = 0$. It follows from \eqref{eq:automorphism gamma} and the definition of $\omega_{A(\lambda)}$ that
    \begin{equation}\label{eq:automorphism gamma A}
        (n + m + n(\varepsilon n + 1)\lambda \delta_m^0)\gamma_m - (n + m + m(\varepsilon m + 1)\lambda \delta_n^0)\gamma_n = (m - n)\gamma_{n+m}.
    \end{equation}
    Since $\gamma_0 = 0$, we have $\delta_n^0 \gamma_n = 0$ for all $n \in \ZZ$. Therefore, we can simplify \eqref{eq:automorphism gamma A} as follows:
    \begin{equation}\label{eq:automorphism gamma A simple}
        (n + m)(\gamma_m - \gamma_n) = (m - n)\gamma_{n+m}
    \end{equation}
    for all $n \in \ZZ$. Note that \eqref{eq:automorphism gamma A simple} has the same form as \eqref{eq:simple mixing cocycle B}, which we already solved in the proof of Proposition \ref{prop:W_B mixing cocycles}: we have $\gamma_n = a n^2 + b n$ for some $a,b \in \CC$ (in fact, $a = \frac{1}{2}\gamma_2 - \gamma_1$ and $b = 2\gamma_1 - \frac{1}{2}\gamma_2$).

    We now compute $\xi_n$. By \eqref{eq:automorphism xi}, we have
    \begin{equation}\label{eq:automorphism xi A}
        (n + m + n(\varepsilon n + 1)\lambda \delta_m^0)\xi_m = (n + m + n(n + 1)\lambda \delta_m^0)\xi_{n+m}
    \end{equation}
    for all $n,m \in \ZZ$. It follows from \eqref{eq:automorphism xi A} that $\xi_m = \xi_{n+m}$ for all $m \in \ZZ \nonzero$ and $n \in \ZZ \setminus \{-m\}$. Letting $\xi = \xi_1$, we therefore conclude that $\xi_n = \xi$ for all $n \in \ZZ \nonzero$.

    To compute $\xi_0$, we substitute $n \in \ZZ \nonzero$ and $m = 0$ into \eqref{eq:automorphism xi A} to get
    \begin{equation}\label{eq:xi0 A}
        (n + n(\varepsilon n + 1)\lambda)\xi_0 = (n + n(n + 1)\lambda)\xi
    \end{equation}
    for all $n \in \ZZ$. If $\varepsilon = 1$, then \eqref{eq:xi0 A} implies that $\xi_0 = \xi$. It follows that $\sigma = \varphi_a^{A(\lambda)} \circ \psi_b^{A(\lambda)} \circ \sigma_\alpha \circ \mu_\xi$ (in other words, $\sigma = \Phi^{A(\lambda)}_{0,a,b,\alpha,\xi}$).
    
    Now assume that $\varepsilon = -1$. In this case, if $\lambda \neq \infty$, \eqref{eq:xi0 A} becomes
    $$(\xi + \xi_0)\lambda n + (\xi - \xi_0)(\lambda + 1) = 0$$
    for all $n \in \ZZ \nonzero$. Therefore, $(\xi + \xi_0)\lambda = 0$ and $(\xi - \xi_0)(\lambda + 1) = 0$. Similarly, if $\lambda = \infty$, then \eqref{eq:xi0 A} gives $(\xi + \xi_0)n + \xi - \xi_0 = 0$ for all $n \in \ZZ \nonzero$, from which it follows that $\xi_0 = \xi = 0$. We conclude that, when $\varepsilon = -1$, there are three possibilities:
    \begin{enumerate}[(a)]
        \item $\lambda = 0$ and $\xi_0 = \xi$;
        \item $\lambda = -1$ and $\xi_0 = -\xi$;
        \item $\lambda \in \PP^1 \setminus \{0,-1\}$ and $\xi_0 = \xi = 0$.
    \end{enumerate}
    However, $\sigma$ is an automorphism, so we cannot have $\xi = 0$. Therefore, the case $\varepsilon = -1$ is not possible if $\lambda \in \PP^1 \setminus \{0,-1\}$. Therefore, in the two cases where $\varepsilon = -1$ is possible ($\lambda = 0$ or $\lambda = -1$), we get $\sigma = \tau_\lambda \circ \varphi_a^{A(\lambda)} \circ \psi_b^{A(\lambda)} \circ \sigma_\alpha \circ \mu_\xi$ (in other words, $\sigma = \Phi_{1;a,b;\alpha,\xi}^{A(\lambda)}$). This concludes the proof when $X = A$.

    \case{2}{$X = B$.}

    Let $\gamma_n' = \gamma_{\varepsilon n}$ and $k = \varepsilon n$, $\ell = \varepsilon m$. Then \eqref{eq:automorphism gamma} becomes
    \begin{equation}\label{eq:automorphism gamma' B}
        \omega_{B(\lambda)}(k,\ell)\gamma_\ell' - \omega_{B(\lambda)}(\ell,k)\gamma_k' = (\ell - k)\gamma_{k+\ell}',
    \end{equation}
    for all $k,\ell \in \ZZ$. By \eqref{eq:automorphism gamma' B} and the definition of $\omega_{B(\lambda)}$, we have
    \begin{equation}\label{eq:automorphism gamma B simple}
        (\ell - k(k + 1)\lambda \delta_{k+\ell}^0)\gamma_\ell' - (k - \ell(\ell + 1)\lambda \delta_{k+\ell}^0)\gamma_k' = (\ell - k)\gamma_{k+\ell}'.
    \end{equation}
    Note that \eqref{eq:automorphism gamma B simple} has the same form as \eqref{eq:mixing cocycle A}, so we already know the solutions, thanks to the proof of Proposition \ref{prop:W_A mixing cocycles}: if $\lambda = 0$, then $\gamma_n' = a'n + b$ for some $a',b \in \CC$, while for $\lambda \notin \{0,\infty\}$, we get
    $$\gamma_n' = \begin{cases}
        (\lambda + 1)c, &\text{if } n = 0; \\
        c, &\text{if } n \neq 0.
    \end{cases}$$
    for some $c \in \CC$. For $\lambda = \infty$, we get $\gamma_n' = c$ for all $n \in \ZZ$. Therefore, letting $a = \varepsilon a'$, we conclude that $\gamma_n = an + b$ if $\lambda = 0$, while
    $$\gamma_n = \begin{cases}
        (\lambda + 1)c, &\text{if } n = 0; \\
        c, &\text{if } n \neq 0.
    \end{cases}$$
    if $\lambda \notin \{0,\infty\}$, and $\gamma_n = c$ for all $n \in \ZZ$ if $\lambda = \infty$.

    We finish by computing $\xi_n$. By \eqref{eq:automorphism xi}, we get
    \begin{equation}\label{eq:automorphism xi B}
        (m - n(\varepsilon n + 1)\lambda \delta_{n+m}^0)\xi_m = (m - n(n + 1)\lambda \delta_{n+m}^0)\xi_{n+m}
    \end{equation}
    for all $n,m \in \ZZ$. Similarly to before, letting $\xi = \xi_1$, \eqref{eq:automorphism xi B} immediately implies that $\xi_n = \xi$ for all $n \in \ZZ \nonzero$.
    
    Finally, in order to compute $\xi_0$, we substitute $m \in \ZZ \nonzero$ and $n = -m$ into \eqref{eq:automorphism xi B}:
    \begin{equation}\label{eq:xi0 B}
        (m - m(\varepsilon m - 1)\lambda)\xi = (m - m(m - 1)\lambda)\xi_0
    \end{equation}
    for all $m \in \ZZ \nonzero$. If $\varepsilon = 1$, then it follows from \eqref{eq:xi0 B} that $\xi_0 = \xi$. Hence,
    $$\sigma = \begin{cases}
        \varphi_a^{B(0)} \circ \psi_b^{B(0)} \circ \sigma_\alpha \circ \mu_\xi, &\text{if } \lambda = 0, \\
        \varphi_c^{B(\lambda)} \circ \sigma_\alpha \circ \mu_\xi, &\text{if } \lambda \neq 0.
    \end{cases}$$
    In other words, $\sigma = \Phi_{0;a,b,\alpha,\xi}^{B(0)}$ if $\lambda = 0$, while $\sigma = \Phi_{0;c,0;\alpha,\xi}^{B(\lambda)}$ if $\lambda \neq 0$.
    
    Therefore, it remains to consider $\varepsilon = -1$. In this case, if $\lambda \neq \infty$, \eqref{eq:xi0 B} becomes
    $$(\xi + \xi_0)\lambda m + (\xi - \xi_0)(\lambda + 1) = 0$$
    for all $m \in \ZZ \nonzero$. For $\lambda = \infty$, we instead get $(\xi + \xi_0)m + \xi - \xi_0 = 0$ for all $m \in \ZZ \nonzero$. Proceeding as we did in Case 1, we conclude that there are three possibilities when $\varepsilon = -1$:
    \begin{enumerate}[(a)]
        \item $\lambda = 0$ and $\xi_0 = \xi$;
        \item $\lambda = -1$ and $\xi_0 = -\xi$;
        \item $\lambda \in \PP^1 \setminus \{0,-1\}$ and $\xi_0 = \xi = 0$.
    \end{enumerate}
    It follows that, as before, the case $\varepsilon = -1$ is not possible if $\lambda \in \PP^1 \setminus \{0,-1\}$. Therefore, in the cases where $\varepsilon = -1$ is possible ($\lambda = 0$ or $\lambda = -1$), we get
    $$\sigma = \begin{cases}
        \tau_0 \circ \varphi_a^{B(0)} \circ \psi_b^{B(0)} \circ \sigma_\alpha \circ \mu_\xi, &\text{if } \lambda = 0, \\
        \tau_{-1} \circ \varphi_c^{B(-1)} \circ \sigma_\alpha \circ \mu_\xi, &\text{if } \lambda = -1.
    \end{cases}$$
    In other words, $\sigma = \Phi_{1;a,b;\alpha,\xi}^{B(0)}$ if $\lambda = 0$, while $\sigma = \Phi_{1;c,0;\alpha,\xi}^{B(-1)}$ if $\lambda = -1$, which concludes the proof.
\end{proof}

\begin{remark}
    As we saw before when we noted that $A(\lambda)$ is the adjoint of $B(\lambda)$, once again we see the close relationship between $\W_A(\lambda)$ and $\W_B(\lambda)$ in the proof of Proposition \ref{prop:outer automorphisms}: when computing automorphisms of $\W_A(\lambda)$, we got an equation \eqref{eq:automorphism gamma A simple} that we already saw when we computed central extensions of $\W_B(\lambda)$ \eqref{eq:simple mixing cocycle B}, and similarly for automorphisms of $\W_B(\lambda)$ with \eqref{eq:automorphism gamma B simple} and \eqref{eq:mixing cocycle A}. We will see this once again when we compute derivations of $\W_X(\lambda)$ in Section \ref{sec:derivations}.
\end{remark}

The next result computes the group structure of the subgroup of $\Aut(\W_X(\lambda))$ generated by the automorphisms defined in Notation \ref{ntt:automorphisms}.

\begin{proposition}\label{prop:group structure of outer automorphisms}
    Let $X \in \{A,B\}$ and $\lambda \in \PP^1$. Let $G_X(\lambda)$ be as in Notation \ref{ntt:outer automorphism group}. Then $G_X(\lambda)$ is a subgroup of $\Aut(\W_X(\lambda))$ and
    $$G_X(\lambda) \cong \begin{cases}
        \ZZ/2\ZZ \ltimes (\CC^2 \rtimes (\CC^\times \times \CC^\times)), &\text{if } X = B \text{ and } \lambda = 0, \\
        \ZZ/2\ZZ \ltimes (\CC \rtimes (\CC^\times \times \CC^\times)), &\text{if } X = A \text{ and } \lambda \in \{0,-1\}, \text{ or } X = B \text{ and } \lambda = -1, \\
        \CC \rtimes (\CC^\times \times \CC^\times), &\text{if } \lambda \notin \{0,-1\}.
    \end{cases}$$
    Given $k,\ell \in \ZZ/2\ZZ$, $a,b,c,d \in \CC$, and $\alpha,\beta,\xi,\zeta \in \CC^\times$, the structure of the groups above is the following:
    \begin{align*}
        G_B(0) &\colon &(k,a,b,\alpha,\xi) \cdot (\ell,c,d,\ell,\beta,\zeta) &= (k + \ell, \varepsilon a + \xi c, b + \xi d, \alpha^{\varepsilon}\beta, \xi \zeta); \\
        G_A(0), G_B(-1) &\colon &(k,a,\alpha,\xi) \cdot (\ell,c,\beta,\zeta) &= (k + \ell, \varepsilon a + \xi c, \alpha^{\varepsilon}\beta, \xi \zeta); \\
        G_A(-1) &\colon &(k,b,\alpha,\xi) \cdot (\ell,d,\beta,\zeta) &= (k + \ell, b + \xi d, \alpha^{\varepsilon}\beta, \xi \zeta); \\
        G_X(\lambda) &\colon &(a,\alpha,\xi) \cdot (c,\beta,\zeta) &= (a + \xi c, \alpha \beta, \xi \zeta);
    \end{align*}
    where $\varepsilon = (-1)^\ell$ and $\lambda \notin \{0,-1\}$. Furthermore, $\Aut(\W_X(\lambda)) \cong \InnAut(\W_X(\lambda)) \rtimes G_X(\lambda)$.
\end{proposition}
\begin{proof}
    Let $k,\ell \in \ZZ/2\ZZ$, $a,b,c,d \in \CC$, and $\alpha,\beta,\xi,\zeta \in \CC^\times$, where we assume that $k = \ell = 0$ if $\lambda \notin \{0,-1\}$. We claim that
    \begin{equation}\label{eq:automorphism composition}
        \Phi_{k;a,b;\alpha,\xi}^{X(\lambda)} \circ \Phi_{\ell;c,d;\beta,\zeta}^{X(\lambda)} = \Phi_{k + \ell; (-1)^\ell a + \xi c, b + \xi d; \alpha^{(-1)^\ell}\beta, \xi \zeta}^{X(\lambda)}.
    \end{equation}
    We will only prove the claim in the case $X = A$, since the case $X = B$ is similar.
    
    Let $\varepsilon_1 = (-1)^k, \varepsilon_2 = (-1)^\ell$, and $\varepsilon = (-1)^{k+\ell}$ (recall the assumption that $k = \ell = 0$ if $\lambda \notin \{0,-1\}$, so $\varepsilon_1 = \varepsilon_2 = \varepsilon = 1$ in these cases). We have
    \begin{align*}
        \Phi_{k;a,b;\alpha,\xi}^{A(\lambda)}(L_n) &= \alpha^n(\varepsilon_1 L_{\varepsilon_1 n} + (an^2 + bn)A_{\varepsilon_1 n}), \\
        \Phi_{k;a,b;\alpha,\xi}^{A(\lambda)}(A_n) &= \begin{cases}
            -\xi A_0, &\text{if } \lambda = -1, n = 0, \text{ and } k = 1, \\
            \alpha^n \xi A_{\varepsilon_1 n}, &\text{otherwise}.
        \end{cases}
    \end{align*}
    for all $n \in \ZZ$. Therefore,
    \begin{align*}
        \Phi_{k;a,b;\alpha,\xi}^{A(\lambda)}(\Phi_{\ell;c,d;\beta,\zeta}^{A(\lambda)}(L_n)) &= \beta^n\left(\varepsilon_2\Phi_{k;a,b;\alpha,\xi}^{A(\lambda)}(L_{\varepsilon_2 n}) + (cn^2 + dn)\Phi_{k;a,b;\alpha,\xi}^{A(\lambda)}(A_{\varepsilon_2 n})\right) \\
        &= (\alpha^{\varepsilon_2}\beta)^n\left(\varepsilon L_{\varepsilon n} + \varepsilon_2(an^2 + \varepsilon_2 b n)A_{\varepsilon n} + (cn^2 + dn)\xi A_{\varepsilon n}\right) \\
        &= (\alpha^{\varepsilon_2}\beta)^n\left(\varepsilon L_{\varepsilon n} + ((\varepsilon_2a + \xi c)n^2 + (b + \xi d)n)A_{\varepsilon n}\right) \\
        &= \Phi_{k+\ell; \varepsilon_2 a + \xi c, b + \xi d; \alpha^{\varepsilon_2} \beta, \xi \zeta}^{A(\lambda)}(L_n), \\
        \Phi_{k;a,b;\alpha,\xi}^{A(\lambda)}(\Phi_{\ell;c,d;\beta,\zeta}^{A(\lambda)}(A_n)) &= \begin{cases}
            -\zeta \Phi_{k;a,b;\alpha,\xi}^{A(\lambda)}(A_0), &\text{if } \lambda = -1, n = 0, \text{ and } \ell = 1, \\
            \beta^n \zeta \Phi_{k;a,b;\alpha,\xi}^{A(\lambda)}(A_{\varepsilon_2 n}), &\text{otherwise}.
        \end{cases} \\
        &= \begin{cases}
            -\xi \zeta A_0, &\text{if } \lambda = -1, n = 0, \text{ and } k + \ell = 1, \\
            (\alpha^{\varepsilon_2}\beta)^n \xi \zeta A_{\varepsilon n}, &\text{otherwise}.
        \end{cases} \\
        &= \Phi_{k+\ell; \varepsilon_2 a + \xi c, b + \xi d; \alpha^{\varepsilon_2} \beta, \xi \zeta}^{A(\lambda)}(A_n).
    \end{align*}
    This proves the the claim.
    
    The computation of the group structures now follows immediately from the claim, the definition of $G_X(\lambda)$ in the statement of Theorem \ref{thm:automorphisms}, and the definitions of $\Phi_{k;a,b;\alpha,\xi}^{X(\lambda)}$ and $\Psi_{k;a;\alpha,\xi}^{A(\lambda)}$ from Notation \ref{ntt:automorphisms} and Lemma \ref{lem:inner automorphisms}. The difference in the group structures of $G_A(-1)$ and $G_A(0)$ comes from the fact that, by definition, $G_A(-1)$ is generated by $\Psi_{k;a;\alpha,\xi}^{A(-1)} = \Phi_{k;0,a;\alpha,\xi}^{A(-1)}$, while $G_A(0)$ is generated by $\Psi_{k;a;\alpha,\xi}^{A(0)} = \Phi_{k;a,0;\alpha,\xi}^{A(0)}$, where $k \in \ZZ/2\ZZ$, $a \in \CC$, and $\alpha, \xi \in \CC^\times$.

    For the final sentence, note that $\Aut(\W_X(\lambda)) = \InnAut(\W_X(\lambda)) G_X(\lambda)$ by Proposition \ref{prop:outer automorphisms}. Since $\InnAut(\W_X(\lambda)) \cap G_X(\lambda) = \{\id_{\W_X(\lambda)}\}$ by Lemma \ref{lem:inner automorphisms}, it follows that this is a semi-direct product. This concludes the proof.
\end{proof}

We are now ready to prove Theorem \ref{thm:automorphisms}.

\begin{proof}[Proof of Theorem \ref{thm:automorphisms}]
    By Proposition \ref{prop:group structure of outer automorphisms}, we have
    $$\Aut(\W_X(\lambda)) = \InnAut(\W_X(\lambda)) \rtimes G_X(\lambda).$$
    Let $a,b,c,d,y_i,z_i \in \CC$ for $i \in \ZZ$, and $\alpha,\beta,\xi,\zeta \in \CC^\times$, where we assume that only finitely many $y_i$ and $z_i$ are nonzero. These parameters allow us to take the product of two elements of $\Aut(\W_X(\lambda))$: we will compute the product
    \begin{equation}\label{eq:composing general automorphisms}
        \left(\prod_{i \in \ZZ} \exp(y_i\ad_{X_i}) \Phi_{0;a,b;\alpha,\xi}^{X(\lambda)}\right) \left(\prod_{j \in \ZZ} \exp(z_j\ad_{X_j}) \Phi_{0;c,d;\beta,\zeta}^{X(\lambda)}\right).
    \end{equation}
    Recall that if $X = B$, we may assume that $y_0 = z_0 = 0$ since $B_0$ is central. Note that
    \begin{equation}\label{eq:moving Phi to the right}
        \Phi_{0;a,b;\alpha,\xi}^{X(\lambda)} \prod_{j \in \ZZ} \exp(z_j\ad_{X_j}) = \prod_{j \in \ZZ}\left(\Phi_{0;a,b;\alpha,\xi}^{X(\lambda)} \exp(z_j\ad_{X_j}) (\Phi_{0;a,b;\alpha,\xi}^{X(\lambda)})^{-1}\right) \Phi_{0;a,b;\alpha,\xi}^{X(\lambda)}.
    \end{equation}
    Letting $z \in \CC$ and $n \in \ZZ$, we have
    \begin{equation}
        \Phi_{0;a,b;\alpha,\xi}^{X(\lambda)} \exp(z\ad_{X_n}) (\Phi_{0;a,b;\alpha,\xi}^{X(\lambda)})^{-1} = \exp\left(z \ad_{\Phi_{0;a,b;\alpha,\xi}^{X(\lambda)}(X_n)}\right) \label{eq:conjugation of inner automorphism} = \exp(z \alpha^n \xi \ad_{X_n}).
    \end{equation}
    Therefore,
    \begin{equation}\label{eq:simplifying composition of automorphisms}
        \Phi_{0;a,b;\alpha,\xi}^{X(\lambda)} \prod_{j \in \ZZ} \exp(z_j\ad_{X_j}) = \prod_{j \in \ZZ} \exp(\alpha^j \xi z_j \ad_{X_j}) \Phi_{0;a,b;\alpha,\xi}^{X(\lambda)},
    \end{equation}
    by \eqref{eq:moving Phi to the right} and \eqref{eq:conjugation of inner automorphism}. Using \eqref{eq:simplifying composition of automorphisms}, the product \eqref{eq:composing general automorphisms} becomes
    $$\prod_{i \in \ZZ} \exp(y_i\ad_{X_i}) \prod_{j \in \ZZ} \exp(z_j \alpha^j \xi \ad_{X_j}) \Phi_{0;a,b;\alpha,\xi}^{X(\lambda)} \Phi_{0.;c,d;\beta,\zeta}^{X(\lambda)} = \prod_{i \in \ZZ} \exp((y_i + \alpha^i \xi z_i)\ad_{X_i})\Phi_{0;a + \xi c, b + \xi d; \alpha \beta, \xi \zeta}^{X(\lambda)},$$
    where we used \eqref{eq:automorphism composition} for the final equality. In summary, we have
    $$\left(\prod_{i \in \ZZ} \exp(y_i\ad_{X_i}) \Phi_{0;a,b;\alpha,\xi}^{X(\lambda)}\right) \left(\prod_{j \in \ZZ} \exp(z_j\ad_{X_j}) \Phi_{0;c,d;\beta,\zeta}^{X(\lambda)}\right) = \prod_{i \in \ZZ} \exp\left((y_i + \alpha^i \xi z_i)\ad_{X_i}\right)\Phi_{0;a + \xi c, b + \xi d; \alpha \beta, \xi \zeta}^{X(\lambda)}.$$
    This proves the result in the case $\lambda \notin \{0,-1\}$.

    To complete the proof, assume that $\lambda \in \{0,-1\}$. For these values of $\lambda$, we need to consider how $\tau_\lambda$ interacts with the other automorphisms. We proceed similarly to above to deduce that
    \begin{align*}
        \tau_\lambda \left(\prod_{i \in \ZZ} \exp(y_i \ad_{X_i})\right) &= \prod_{i \in \ZZ} \left(\tau_\lambda \exp(y_i \ad_{X_i}) \tau_\lambda \right) \tau_\lambda \\
        &= \left(\prod_{i \in \ZZ} \exp(y_i \ad_{\tau_\lambda(X_i)})\right) \tau_\lambda \\
        &= \begin{cases}
            \prod_{i \in \ZZ} \exp(y_{-i} \ad_{X_i}) \tau_0, &\text{if } \lambda = 0, \\
            \prod_{i \in \ZZ} \exp((-1)^{\delta_i^0}y_{-i} \ad_{X_i}) \tau_{-1}, &\text{if } \lambda = -1,
        \end{cases}
    \end{align*}
    Therefore, if we multiply two arbitrary elements of $\Aut(\W_X(\lambda))$, we get
    \begin{multline*}
        \left(\prod_{i \in \ZZ} \exp(y_i \ad_{X_i}) \Phi_{k;a,b;\alpha,\xi}^{X(\lambda)}\right)\left(\prod_{i \in \ZZ} \exp(z_i \ad_{X_i}) \Phi_{\ell;c,d;\beta,\zeta}^{X(\lambda)}\right) \\
        = \begin{cases}
            \prod_{i \in \ZZ} \exp((y_i + \alpha^{\varepsilon_1 i} \xi z_{\varepsilon_1 i}) \ad_{X_i}) \Phi_{k + \ell;\varepsilon_2 a + \xi c, b + \xi d; \alpha^{\varepsilon_2} \beta, \xi \zeta}^{X(\lambda)}, &\text{if } \lambda = 0, \\
            \prod_{i \in \ZZ} \exp((y_i + \varepsilon_1^{\delta_i^0} \alpha^{\varepsilon_1 i} \xi z_{\varepsilon_1 i}) \ad_{X_i}) \Phi_{k + \ell;\varepsilon_2 a + \xi c, b + \xi d; \alpha^{\varepsilon_2} \beta, \xi \zeta}^{X(\lambda)}, &\text{if } \lambda = -1,
        \end{cases}
    \end{multline*}
    where $k,\ell \in \ZZ/2\ZZ$ and $\varepsilon_1 = (-1)^k$, $\varepsilon_2 = (-1)^\ell$. This concludes the proof.
\end{proof}

\section{Derivations}\label{sec:derivations}

Next, we compute derivations of $\W_X(\lambda)$. As remarked in Section \ref{sec:preliminaries}, computing derivations of $\g$ is equivalent to computing $\HH^1(\g;\g)$. There is a close relationship between the automorphisms of a Lie algebra $\g$ and its derivations: if we have a (smooth) one-parameter subgroup of $\Aut(\g)$, differentiating it at $0$ gives a derivation of $\g$. In other words, derivations of $\g$ can be seen as ``infinitesimal automorphisms''. We now introduce notation for the derivations that arise from one-parameter subgroups of $\Aut(\W_X(\lambda))$.

\begin{Notation}\label{ntt:derivations}
    Let $X \in \{A,B\}$ and $\lambda \in \PP^1$. Recall the automorphisms defined in Notation \ref{ntt:automorphisms}. For both $X = A$ and $X = B$, define
    $$d_{\Ab} = \diffeval{t}{0}\mu_{e^t} \in \Der(\W_X(\lambda)).$$
    For $X = A$, define
    $$\delta_\lambda^A = \diffeval{t}{0}\varphi_t^{A(\lambda)}, \quad \del_\lambda^A = \diffeval{t}{0}\psi_t^{A(\lambda)} \in \Der(\W_A(\lambda)).$$
    For $X = B$, define
    $$d_\lambda^B = \diffeval{t}{0}\varphi_t^{B(\lambda)} \in \Der(\W_B(\lambda)).$$
    For $X = B$ and $\lambda = 0$, define
    $$\del^B_0 = \diffeval{t}{0}\psi_t^{B(0)} \in \Der(\W_B(0)).$$
\end{Notation}

\begin{remark}
    Certainly, there is another one-parameter subgroup $\sigma_{e^t}$ of $\Aut(\W_X(\lambda))$, giving rise to another derivation. However, the derivation we obtain from this one-parameter subgroup is inner: we have
    $$\diffeval{t}{0}\sigma_{e^t} = \ad_{L_0} \in \Inn(\W_X(\lambda)).$$
    Indeed, we can immediately see this from the observations of Remark \ref{rem:exp(L0)}.
\end{remark}

Having defined these derivations that arise from automorphisms, we now compute their actions on $\W_X(\lambda)$.

\begin{lemma}\label{lem:derivations}
    Let $X \in \{A,B\}$ and $\lambda \in \PP^1$. Then
    $$\begin{gathered}
        d_{\Ab}(L_n) = 0, \quad d_{\Ab}(X_n) = X_n; \\
        \delta_\lambda^A(L_n) = nA_n, \quad \delta_\lambda^A(A_n) = 0; \\
        \del_\lambda^A(L_n) = n(n + 1)A_n, \quad \del_\lambda^A(A_n) = 0; \\
        d_{\lambda}^B(L_n) = \begin{cases}
        (\lambda + 1)B_0, &\text{if } n = 0 \text{ and } \lambda \neq \infty,\\
        B_n, &\text{otherwise},
    \end{cases} \quad d_{\lambda}^B(B_n) = 0; \\
    \del^B_0(L_n) = nB_n, \quad \del^B_0(B_n) = 0,
    \end{gathered}$$
    for all $n \in \ZZ$.
\end{lemma}
\begin{proof}
    Follows immediately from the definitions of the automorphisms in Notation \ref{ntt:automorphisms}.
\end{proof}

\begin{remark}\label{rem:cocycle to derivation dictionary}
    Using the identifications $A(\lambda) = \CC[t,t^{-1}]dt$ and $B(\lambda) = \CC[t,t^{-1}]$, we now give formulas for the derivations of Notation \ref{ntt:derivations} in terms of polynomials. We have
    \begin{align*}
        \delta_\lambda^A(f\del) &= d(t^{-1}f) = (t^{-1}f' - t^{-2}f) dt, \\
        \del_\lambda^A(f\del) &= d(f') = f'' \ dt, \\
        d_\lambda^B(f\del) &= \begin{cases}
            t^{-1}f + \lambda \ang{t^{-1},f}, &\text{if } \lambda \neq \infty, \\
            t^{-1}f, &\text{if } \lambda = \infty,
        \end{cases} \\
        \del_0^B(f\del) &= f' - t^{-1}f.
    \end{align*}
    By the observations of Remark \ref{rem:basis-free cocycles}, it is now clear that, under the isomorphism $\HH^2_{\Mix}(\W_X(\lambda)) \cong \HH^1(\W;X(\lambda)')$, we have the following correspondences:
    \begin{align*}
        \delta_\lambda^A &\leftrightarrow \Omega^B_{\Mix} \quad (\lambda \neq 0), \\
        \del_0^A &\leftrightarrow \Omega^B_{\Mix} \quad (\lambda = 0), \\
        d_\lambda^B &\leftrightarrow \Omega^A_{\Mix}, \\
        \del_0^B &\leftrightarrow \Omega^A_0,
    \end{align*}
    where we have used that $A(\lambda)' \cong B(\lambda)$ and $B(\lambda)' \cong A(\lambda)$ under the form $\mathcal{B}$.
\end{remark}

We now consider the question of when the derivations defined in Notation \ref{ntt:derivations} are inner.

\begin{lemma}\label{lem:inner derivations}
    Let $\lambda \in \PP^1$. If $\lambda \neq \infty$, then $(\lambda + 1)\delta_\lambda^A + \lambda \del_\lambda^A \in \Inn(\W_A(\lambda))$, while if $\lambda = \infty$, then $\delta_\infty^A + \del_\infty^A = \ad_{-A_0} \in \Inn(\W_A(\infty))$. Defining
    $$d_\lambda^A = \begin{cases}
        \delta_\lambda^A, &\text{if } \lambda \neq 0, \\
        \del_0^A, &\text{if } \lambda = 0,
    \end{cases}$$
    then $d_\lambda^A \notin \Inn(\W_A(\lambda))$ for all $\lambda \in \PP^1$. Certainly, we have
    $$d_\lambda^A = \diffeval{t}{0}\varpi_t^{A(\lambda)},$$
    where $\varpi_t^{A(\lambda)}$ is defined in \eqref{eq:special automorphism A}.

    Furthermore, $d_\lambda^B \notin \Inn(\W_B(\lambda))$, and there is no nonzero linear combination of $d_0^B$ and $\del^B_0$ which produces an inner derivation of $\W_B(0)$.
\end{lemma}
\begin{proof}
    By Lemma \ref{lem:derivations}, we can easily see that $(\lambda + 1)\delta_\lambda^A + \lambda \del_\lambda^A = \ad_{-A_0}$ if $\lambda \neq \infty$, while $\delta_\infty^A + \del_\infty^A = \ad_{-A_0}$ if $\lambda = \infty$. Since $\ad_{A_0}$ and $\ad_{L_0}$ span the space of inner derivations of $\W_A(\lambda)$ of degree zero, it follows immediately that $d_\lambda^A$ is not inner for all $\lambda \in \PP^1$.

    The only nonzero (up to multiplication by a scalar) inner derivation of $\W_B(\lambda)$ of degree zero is $\ad_{L_0}$. This is because $B_0$ is central in $\W_B(\lambda)$, so $\ad_{B_0} = 0$. The result follows.
\end{proof}

The main goal of this section is to prove that, together with inner derivations, the derivations defined in Notation \ref{ntt:derivations} span the entire space of derivations of $\W_X(\lambda)$.

\begin{theorem}\label{thm:derivations}
    Let $\lambda \in \PP^1$. Then
    \begin{align*}
        \Der(\W_A(\lambda)) &= \Inn(\W_A(\lambda)) \oplus \CC d_{\Ab} \oplus \CC d_\lambda^A; \\
        \Der(\W_B(\lambda)) &= \begin{cases}
            \Inn(\W_B(\lambda)) \oplus \CC d_{\Ab} \oplus \CC d_0^B \oplus \CC \del_0^B, &\text{if } \lambda = 0, \\
            \Inn(\W_B(\lambda)) \oplus \CC d_{\Ab} \oplus \CC d_\lambda^B, &\text{if } \lambda \neq 0,
        \end{cases}
    \end{align*}
    where the derivations are defined in Notation \ref{ntt:derivations} and Lemma \ref{lem:inner derivations}. Consequently,
    $$\dim(\HH^1(\W_A(\lambda);\W_A(\lambda))) = 2, \quad \dim(\HH^1(\W_B(\lambda);\W_B(\lambda))) = \begin{cases}
        3, &\text{if } \lambda = 0, \\
        2, &\text{if } \lambda \neq 0.
    \end{cases}$$
\end{theorem}

\begin{remark}
    Theorem \ref{thm:derivations} says that all derivations of $\W(0,1)$ extend to the one-parameter family $\W_A(\lambda)$. On the other hand, all derivations in $\Inn(\W(0,0)) \oplus \CC d_{\Ab} \oplus \CC d_0^B \subseteq \Der(\W(0,0))$ extend to the one-parameter family $\W_B(\lambda)$, while $\del_0^B$ does not.
\end{remark}

The next result greatly simplifies the computation of derivations of $\W_X(\lambda)$.

\begin{proposition}\label{prop:derivation splitting}
    Let $\g$ be a Lie algebra with $H^1(\g;\g) = 0$, and let $\M$  be a $\g$-module such that $\Hom_\g(\M,\g) = 0$. Then
    $$H^1(\g \ltimes \M; \g \ltimes \M) \cong H^1(\g;\M) \oplus \End_\g(\M),$$
    where we endow $\M$ with an abelian Lie algebra structure.
\end{proposition}
\begin{proof}
    Define
    $$\alpha \colon \Der(\g,\M) \hookrightarrow \Der(\g \ltimes \M), \quad \beta \colon \End_\g(\M) \hookrightarrow \Der(\g \ltimes \M)$$
    by
    $$\alpha(D)(x) = D(x), \quad \alpha(D)(v) = 0,$$
    $$\beta(\varphi)(x) = 0, \quad \beta(\varphi)(v) = \varphi(v),$$
    where $D \in \Der(\g,\M)$, $\varphi \in \End_\g(\M)$, $x \in \g$, and $v \in \M$. It is straightforward to check that $\alpha(D)$ and $\beta(\varphi)$ are elements of $\Der(\g \ltimes \M)$.

    Let $D \in \Der(\g \ltimes \M)$. Let $\overline{D}$ be the composition
    $$\g \xrightarrow{D} \g \ltimes \M \twoheadrightarrow (\g \ltimes \M)/\M \cong \g.$$
    It is easy to see that $\overline{D} \in \Der(\g)$. But $H^1(\g;\g) = 0$, so $\Der(\g) = \Inn(g)$. Therefore, $\overline{D} \in \Inn(\g)$, so there exists $w \in \g$ such that $\overline{D} = \ad_w$. Define
    $$D' \coloneqq D - \ad_w \in \Der(\g \ltimes \M).$$
    By construction, $D'(x) \in \M$ for all $x \in \g$. In other words, $\restr{D'}{\g} \in \Der(\g,\M)$.

    We claim that $\restr{D'}{\M}$ is a homomorphism of $\g$-modules. Indeed, for $x \in \g$ and $v \in \M$, we have
    $$D'(x \cdot v) = D'([x,v]) = [x,D'(v)] + [D'(x),v] = [x,D'(v)] = x \cdot D'(v),$$
    where we used that $D'(x) \in \M$ to deduce that $[D'(x),v] = 0$. Now let $\widehat{D}'$ be the composition
    $$\M \xrightarrow{D'} \g \ltimes \M \twoheadrightarrow \g.$$
    Since $\widehat{D}'$ is the composition of two $\g$-module homomorphisms, $\widehat{D}'$ is also a $\g$-module homomorphism. But since $\Hom_\g(\M,\g) = 0$ by assumption, we have $\widehat{D}' = 0$. Therefore, $D'(v) \in \M$ for all $v \in \M$, and thus $\restr{D'}{\M} \in \End_\g(\M)$.

    Define $D_1 \coloneqq \alpha(\restr{D'}{\g})$ and $D_2 \coloneqq \beta(\restr{D'}{\M})$. Certainly, $D = \ad_w + D_1 + D_2$. It follows that
    \begin{equation}\label{eq:derivation splitting}
        \Der(\g \ltimes \M) = \Big(\Inn(\g \ltimes \M) + \Der(\g,\M)\Big) \oplus \End_\g(\M),
    \end{equation}
    where we view $\Der(\g,\M)$ and $\End_\g(\M)$ as subsets of $\Der(\g \ltimes \M)$. Note that $\Der(\g,\M) \cap \Inn(\g \ltimes \M) = \Inn(\g,\M)$, and thus
    $$\frac{\Inn(\g \ltimes \M) + \Der(\g,\M)}{\Inn(\g \ltimes \M)} \cong \frac{\Der(\g,\M)}{\Der(\g,\M) \cap \Inn(\g \ltimes \M)} = \frac{\Der(\g,\M)}{\Inn(\g,\M)} \cong H^1(\g; \M).$$
    Therefore, taking the quotient by $\Inn(\g \ltimes \M)$ on both sides of \eqref{eq:derivation splitting}, we get
    $$H^1(\g \ltimes \M; \g \ltimes \M) \cong H^1(\g; \M) \oplus \End_\g(\M),$$
    which concludes the proof.
\end{proof}

We now show that Proposition \ref{prop:derivation splitting} applies to $\W_X(\lambda)$.

\begin{lemma}\label{lem:no homs X to W}
    Let $X \in \{A,B\}$ and $\lambda \in \PP^1$. Then $\Hom_\W(X(\lambda),\W) = 0$.
\end{lemma}
\begin{proof}
    The internally graded structures of $\W$ and $X(\lambda)$ immediately imply that any element of the space $\Hom_\W(X(\lambda),\W)$ must be graded of degree 0. Therefore, if $\varphi \in \Hom_\W(X(\lambda),\W)$, then
    $$\varphi(X_n) = \eta_n L_n.$$
    We now split the proof into the cases $X = A$ and $X = B$.

    \case{1}{$X = A$.}

    The equality $\varphi([L_n,A_m]) = [L_n,\varphi(A_m)]$ gives
    \begin{equation}\label{eq:eta coefficients A}
        (m - n)\eta_m = (n + m + n(n + 1) \lambda \delta_m^0)\eta_{n+m}
    \end{equation}
    for all $n,m \in \ZZ$. Letting $m \in \ZZ \nonzero$ and $n = -m$ in \eqref{eq:eta coefficients A}, we get $\eta_m = 0$ for all $m \in \ZZ \nonzero$. If we set $m = 0, n \in \ZZ \nonzero$ in \eqref{eq:eta coefficients A}, we also deduce that $\eta_0 = 0$. This proves the case $X = A$.

    \case{2}{$X = B$.}

    The equality $\varphi([L_n,B_m]) = [L_n,\varphi(B_m)]$ gives
    \begin{equation}\label{eq:eta coefficients B}
        (m - n)\eta_m = (m - n(n + 1)\lambda \delta_{n+m}^0)\eta_{n+m}
    \end{equation}
    for all $n,m \in \ZZ$. Setting $m = 0$ and $n \in \ZZ \nonzero$ \eqref{eq:eta coefficients B}, we get $\eta_0 = 0$. If we let $m \in \ZZ \nonzero$ and $n = -m$ in \eqref{eq:eta coefficients B}, we also deduce that $\eta_m = 0$ for $m \neq 0$. This concludes the proof.
\end{proof}

Proposition \ref{prop:derivation splitting}, together with Lemma \ref{lem:no homs X to W}, and the well-known fact that $\HH^1(\W;\W) = 0$ (see \cite[Proposition 7.2]{Buzaglo} or \cite{IkedaKawamoto}), imply the following.

\begin{corollary}\label{cor:derivations of WX split}
    Let $X \in \{A,B\}$ and $\lambda \in \PP^1$. Then
    $$\Der(\W_X(\lambda)) = \Inn(\W_X(\lambda)) + \Der(\W,X(\lambda)) + \End_\W(X(\lambda)),$$
    and thus
    $$\HH^1(\W_X(\lambda)) = \HH^1(\W;X(\lambda)) \oplus \End_\W(X(\lambda)).\qed$$
\end{corollary}

Therefore, to prove Theorem \ref{thm:derivations} it suffices to compute $\Der(\W,X(\lambda))$ and $\End_\W(X(\lambda))$. For $\Der(\W,X(\lambda))$, we have the isomorphism
$$\HH^1(\W;X(\lambda)') \cong \HH_{\Mix}^2(\W_X(\lambda)),$$
so we do not need to perform any further computations, since $\HH_{\Mix}^2(\W_X(\lambda))$ was already computed in Proposition \ref{prop:W_X mixing cocycles}.

\begin{proposition}\label{prop:derivations W to X}
    Let $\lambda \in \PP^1$. Then
    \begin{align*}
        \Der(\W,A(\lambda)) &= \Inn(\W,A(\lambda)) \oplus d_\lambda^A, \\
        \Der(\W,B(\lambda)) &= \begin{cases}
            \Inn(\W,B(\lambda)) \oplus \CC d_0^B \oplus \CC \del_0^B, &\text{if } \lambda = 0, \\
            \Inn(\W,B(\lambda)) \oplus \CC d_0^B, &\text{if } \lambda \neq 0,
        \end{cases}
    \end{align*}
    where the derivations are defined in Notation \ref{ntt:derivations} and Lemma \ref{lem:inner derivations}.
\end{proposition}
\begin{proof}
    By Propositions \ref{prop:W_X mixing cocycles} and \ref{prop:W_B mixing cocycles}, we know that $\HH^2_{\Mix}(\W_B(\lambda))$ is spanned by $\overline{\Omega}_{\Mix}^B$, where $\overline{\Omega}_{\Mix}^B$ is defined in Theorem \ref{thm:main}. By the isomorphism
    \begin{equation}\label{eq:mixing B to derivations A}
        \HH^2_{\Mix}(\W_B(\lambda)) \cong \HH^1(\W;B(\lambda)') \cong \HH^1(\W;A(\lambda)),
    \end{equation}
    it follows that $\HH^1(\W;A(\lambda))$ is also one-dimensional. By Remark \ref{rem:cocycle to derivation dictionary}, the image of $\overline{\Omega}_{\Mix}^B$ under the isomorphism \eqref{eq:mixing B to derivations A} is $\overline{d}_\lambda^A$, where $\overline{d}_\lambda^A$ is the class of $d_\lambda^A$ in $\HH^1(\W;A(\lambda))$. Therefore,
    $$\Der(\W,A(\lambda)) = \Inn(\W,A(\lambda)) \oplus \CC d_\lambda^A,$$
    as required.
    
    The computation of $\Der(\W,B(\lambda))$ follows similarly by applying Propositions \ref{prop:W_X mixing cocycles} and \ref{prop:W_A mixing cocycles}, and the isomorphism
    $$\HH^2_{\Mix}(\W_A(\lambda)) \cong \HH^1(\W;A(\lambda)') \cong \HH^1(\W;B(\lambda)),$$
    using the correspondences in Remark \ref{rem:cocycle to derivation dictionary}.
\end{proof}

We now proceed with the computation of $\End_\W(X(\lambda))$, which is the last step in the proof of Theorem \ref{thm:derivations}.

\begin{proposition}\label{prop:endomorphisms of X}
    Let $X \in \{A,B\}$ and $\lambda \in \PP^1$. Then
    $$\End_\W(X(\lambda)) = \CC d_{\Ab},$$
    where $d_{\Ab}$ is defined in Notation \ref{ntt:derivations}.
\end{proposition}
\begin{proof}
    The internally graded structures of $\W$ and $X(\lambda)$ imply that any $\W$-endomorphism of $X(\lambda)$ is graded of degree 0. Letting $\varphi \in \End_\W(X(\lambda))$, this means that
    $$\varphi(X_n) = \theta_n X_n$$
    for some $\theta_n \in \CC$. Replacing $\varphi$ with $\varphi - \theta_0 d_{\Ab}$, we may assume that $\theta_0 = 0$. To prove the result, it suffices to show that $\varphi = 0$. We now analyse the cases $X = A$ and $X = B$ separately.

    \case{1}{$X = A$.}

    The equality $\varphi(L_n \cdot A_m) = L_n \cdot \varphi(A_m)$ gives
    $$(n + m + n(n + 1) \lambda \delta_m^0)\theta_m = (n + m + n(n + 1) \lambda \delta_m^0)\theta_{n+m}$$
    for all $n,m \in \ZZ$. Therefore, $\theta_n = \theta_0 = 0$ for all $n \in \ZZ$. This completes the proof in the case $X = A$.

    \case{2}{$X = B$.}

    The equality $\varphi(L_n \cdot B_m) = L_n \cdot \varphi(B_m)$ gives
    $$(m - n(n + 1)\lambda \delta_{n+m}^0)\theta_m = (m - n(n + 1)\lambda \delta_{n+m}^0)\theta_{n+m}$$
    for all $n,m \in \ZZ$. Therefore, $\theta_n = \theta_0 = 0$ for all $n \in \ZZ$, so we are done.
\end{proof}

We are now ready to prove Theorem \ref{thm:derivations}.

\begin{proof}[Proof of Theorem \ref{thm:derivations}]
    By Corollary \ref{cor:derivations of WX split}, we have
    $$\HH^1(\W_X(\lambda);\W_X(\lambda)) \cong \HH^1(\W;X(\lambda)) \oplus \End_\W(X(\lambda)).$$
    Proposition \ref{prop:derivations W to X} computes the cohomology space $\HH^1(\W;X(\lambda))$, and Proposition \ref{prop:endomorphisms of X} computes $\End_\W(X(\lambda))$.
\end{proof}

\section{Cohomology of \texorpdfstring{$\widetilde{\W}$}{W tilde}}\label{sec:W tilde}

We finish the paper by computing one-dimensional central extensions and derivations of $\widetilde{\W} \coloneqq \W \ltimes \widetilde{I}$. For ease of notation, we view $\widetilde{I}$ as the space of Laurent polynomials with no constant term. Furthermore, thanks to the isomorphism $\widetilde{I} \cong \widetilde{J}$, we will often switch between $\widetilde{I}$ and $\widetilde{J}$ depending on the situation.

We start with the computation of one-dimensional central extensions.

\begin{theorem}\label{thm:central extensions W tilde}
    We have $\dim(\HH^2(\widetilde{\W})) = 4$. Specifically,
    $$\HH^2(\widetilde{\W}) = \CC \overline{\Omega}_{\Vir} \oplus \CC \overline{\Omega}_1 \oplus \CC \overline{\Omega}_2 \oplus \CC \overline{\Omega}_{\Ab},$$
    where $\Omega_{\Vir}$ is the Gelfand--Fuchs cocycle in \eqref{eq:Virasoro cocycle}, and the other cocycles are defined as follows:
    $$\begin{gathered}
        \Omega_1(f\del,g) = \mathcal{B}(df',g) = \Res(g \ df'), \\
        \Omega_2(f\del,g) = \mathcal{B}(d(t^{-1}f),g) = \Res(g \ d(t^{-1}f)), \\
        \Omega_{\Ab}(f,g) = \ang{f,g},
    \end{gathered}$$
    where $\ang{-,-}$ and $\mathcal{B}$ are defined in \eqref{eq:self duality of I} and \eqref{eq:B form} respectively, and we have only specified the non-zero components of the cocycles.
\end{theorem}

To prove Theorem \ref{thm:central extensions W tilde}, we use the decomposition
$$\HH^2(\widetilde{\W}) = \HH^2(\W) \oplus \HH^1(\W;\widetilde{I}) \oplus \B_\W(\widetilde{I})$$
from Proposition \ref{prop: cohomology splitting}, recalling the self-duality $\widetilde{I}' \cong \widetilde{I}$.

As mentioned in Section \ref{sec:preliminaries}, the classes of $\delta_1$ and $\delta_2$ span $\HH^1(\W;\widetilde{I})$. The proof of this follows as an immediate consequence of Proposition \ref{prop:derivations W to X}.

\begin{corollary}\label{cor:derivations W to I tilde}
    We have $\Der(\W,\widetilde{I}) = \Inn(\W,\widetilde{I}) \oplus \CC \delta_1 \oplus \CC \delta_2$, where $\delta_1$ and $\delta_2$ are defined in Section \ref{sec:preliminaries}. Consequently,
    $$\dim(\HH^1(\W;\widetilde{I})) = 2.$$
\end{corollary}
\begin{proof}
    Recall the isomorphism
    $$\widetilde{I} \cong \widetilde{J} = \W \cdot A(0) = \spn\{A_n \mid n \in \ZZ \nonzero\}.$$
    Since $\widetilde{J} \subseteq A(0)$, we have
    $$\Der(\W,\widetilde{J}) \subseteq \Der(\W,A(0)).$$
    Therefore, $\Der(\W,\widetilde{J})$ consists of the elements of $\Der(\W,A(0))$ whose images are contained in $\widetilde{J}$. By Proposition \ref{prop:derivations W to X}, all elements of $\Der(\W,A(0))$ have this property, in other words,
    $$\Der(\W,A(0)) = \Der(\W,\widetilde{J}).$$
    Noting that $\del_0^A = \delta_1$ and $\delta_0^A = \delta_2$, it follows that $\Der(\W,\widetilde{J}) = \Inn(\W,\widetilde{J}) \oplus \CC \delta_1 \oplus \CC \delta_2$. It is easy to see that the sum is direct, because $\Inn(\W,\widetilde{J})$ does not contain any elements of degree zero.
\end{proof}

To complete the proof of Theorem \ref{thm:central extensions W tilde}, it remains to compute $\B_\W(\widetilde{I})$, which we do next.

\begin{lemma}\label{lem:invariant forms on I tilde}
    The unique (up to multiplication by a scalar) $\W$-invariant skew-symmetric form on $\widetilde{I}$ is $\ang{-,-}$, where $\ang{-,-}$ is defined in \eqref{eq:self duality of I}. Consequently,
    $$\dim(\B_\W(\widetilde{I})) = 1.$$
\end{lemma}
\begin{proof}
    Suppose $\Omega \in \B_\W(\widetilde{J}) \nonzero$. By Theorem \ref{thm:internal grading},
    $$\Omega(A_n,A_m) = \alpha(n)\delta_{n+m}^0$$
    for all $n,m \in \ZZ \nonzero$, where $\alpha \colon \ZZ \nonzero \to \CC$. By the $\W$-invariance of $\Omega$, we have
    $$\Omega(A_n, L_{-n - m} \cdot A_m) + \Omega(L_{-n - m} \cdot A_n,A_m) = 0$$
    for all $n,m \in \ZZ \nonzero$. Therefore,
    $$n \alpha(n) - m \alpha(m) = 0$$
    for all $n,m \in \ZZ \nonzero$. Setting $m = 1$, we get $\alpha(n) = \frac{\alpha(1)}{n}$ for all $n \in \ZZ \nonzero$.

    By assumption, $\Omega \neq 0$, so $\alpha(1) \neq 0$. Rescaling if necessary, we may assume that $\alpha(1) = 1$. Therefore,
    $$\Omega(A_n,A_m) = \frac{1}{n}\delta_{n + m}^0$$
    is the unique (up to multiplication by a scalar) $\W$-invariant skew-symmetric form on $\widetilde{J}$. Under the isomorphism $\widetilde{J} \cong \widetilde{I}$, the form $\Omega$ is precisely $\ang{-,-}$.
\end{proof}

The computation of $\HH^2(\widetilde{\W})$ follows easily.

\begin{proof}[Proof of Theorem \ref{thm:central extensions W tilde}]
    We use the decomposition
    $$\HH^2(\widetilde{\W}) = \HH^2(\W) \oplus \HH^1(\W;\widetilde{I}) \oplus \B_\W(\widetilde{I})$$
    from Proposition \ref{prop: cohomology splitting}, and the correspondences between $\HH^1(\W;\widetilde{I})$ and mixing cocycles, and between $\B_\W(\widetilde{I})$ and abelian cocycles.

    It is easy to see that $\delta_1$ and $\delta_2$ give rise to the cocycles $\Omega_1$ and $\Omega_2$, respectively. Therefore, $\overline{\Omega}_1$ and $\overline{\Omega}_2$ span the space of mixing cocycles of $\widetilde{\W}$, by Corollary \ref{cor:derivations W to I tilde}.

    Certainly $\Omega_{\Ab}$ is simply the form $\ang{-,-}$ trivially extended to the whole of $\widetilde{\W}$. Therefore, by Lemma \ref{lem:invariant forms on I tilde}, $\Omega_{\Ab}$ is the unique (up to scalar multiplication) abelian cocycle on $\widetilde{\W}$.
\end{proof}

Next, we compute Leibniz central extensions of $\widetilde{\W}$. As we did in Section \ref{sec:Leibniz}, we do this by computing $\Inv(\widetilde{\W})$. The computation of $\Inv(\widetilde{\W})$ follows easily from the computation of $\Inv(\W_X(\lambda))$ from Section \ref{sec:Leibniz}.

\begin{theorem}\label{thm:Leibniz W tilde}
    We have $\HL^2(\widetilde{\W}) \cong \HH^2(\widetilde{\W})$.
\end{theorem}
\begin{proof}
    By Proposition \ref{prop:Leibniz InvForm exact sequence}, it suffices to show that $\Inv(\widetilde{\W}) = 0$. Letting $\theta \in \Inv(\W)$, the proof of Lemma \ref{lem: components of invariants forms on W_X()} implies that
    \begin{align*}
        \theta(L_n, L_m) = 0 \quad &\text{for all } n,m \in \ZZ; \\
        \theta(L_n, A_m) = 0 \quad &\text{if } n \in \ZZ, m \in \ZZ \nonzero, \text{ and } n + m \neq 0; \\
        \theta(A_n, A_m) = 0 \quad &\text{for all } n,m \in \ZZ \nonzero,
    \end{align*}
    where we view $\widetilde{\W}$ as $\W \ltimes \widetilde{J} \subseteq \W_A(0)$. Therefore, it remains to show that $\theta(L_n, A_{-n}) = 0$ for all $n \in \ZZ \nonzero$. This is done in the proof of Proposition \ref{prop:InvForm}.
\end{proof}

We end with the computation of derivations of $\widetilde{\W}$.

\begin{theorem}\label{thm:derivations W tilde}
    We have $\Der(\widetilde{\W}) = \Inn(\widetilde{\W}) \oplus \CC \delta_1 \oplus \CC \delta_2 \oplus \CC d_{\Ab}$, where $\delta_1$ and $\delta_2$ are defined in Section \ref{sec:preliminaries}, and $d_{\Ab}$ is defined as:
    $$d_{\Ab}(f\del) = 0, \quad \restr{d_{\Ab}}{\widetilde{I}} = \id_{\widetilde{I}},$$
    for all $f \in \CC[t,t^{-1}]$. Consequently, $\dim(\HH^1(\widetilde{\W};\widetilde{\W})) = 3$.
\end{theorem}

To prove Theorem \ref{thm:derivations W tilde}, we would like to apply Proposition \ref{prop:derivation splitting}, which requires us to prove that $\Hom_\W(\widetilde{I},\W) = 0$.

\begin{lemma}\label{lem:no homs I tilde to W}
    We have $\Hom_\W(\widetilde{I},\W) = 0$.
\end{lemma}
\begin{proof}
    Let $\varphi \in \Hom_\W(\widetilde{J},\W) = 0$. The internally graded structures of $\widetilde{J}$ and $\W$ imply that there exist $\eta_n \in \CC$ such that
    $$\varphi(A_n) = \eta_n L_n$$
    for all $n \in \ZZ \nonzero$. Since $\varphi([L_n,A_m]) = [L_n,\varphi(A_m)]$, we get
    $$(n + m)\eta_{n + m} = (m - n)\eta_m$$
    for all $n,m \in \ZZ \nonzero$. Setting $n = -m$, we deduce that $\eta_m = 0$ for all $m \in \ZZ \nonzero$, so $\varphi = 0$.
\end{proof}

By Proposition \ref{prop:derivation splitting} and Lemma \ref{lem:no homs I tilde to W}, we know that $\HH^1(\widetilde{\W})$ decomposes into
$$\HH^1(\widetilde{\W}) \cong \HH^1(\W;\widetilde{I}) \oplus \End_\W(\widetilde{I}).$$
Since $\HH^1(\W;\widetilde{I})$ was already computed in Corollary \ref{cor:derivations W to I tilde}, it only remains to compute $\End_\W(\widetilde{I})$.

\begin{lemma}\label{lem:endomorphisms I tilde}
    We have $\End_\W(\widetilde{I}) = \CC \id_{\widetilde{I}}$.
\end{lemma}
\begin{proof}
    Let $\varphi \in \End_\W(\widetilde{J}) \nonzero$. By the internally graded structures of $\W$ and $\widetilde{J}$, there exist $\theta_n \in \CC$ such that
    $$\varphi(A_n) = \theta_n A_n$$
    for all $n \in \ZZ \nonzero$. Since $\varphi$ is a map of $\W$-modules, we have $\varphi(L_n \cdot A_m) = L_n \cdot \varphi(A_m)$, and thus
    $$(n + m)\theta_{n + m} = (n + m)\theta_m$$
    for all $n,m \in \ZZ \nonzero$. It follows that $\theta_n = \theta_1$ for all $n \in \ZZ \nonzero$. Since $\varphi \neq 0$ by assumption, we know that $\theta_1 \neq 0$. Rescaling if necessary, we may assume that $\theta_1 = 1$. Therefore, $\varphi(A_n) = A_n$ for all $n \in \ZZ \nonzero$, so $\varphi = \id_{\widetilde{J}}$.
\end{proof}

\begin{proof}[Proof of Theorem \ref{thm:derivations W tilde}]
    Follows immediately from the decomposition
    $$\HH^1(\widetilde{\W}) \cong \HH^1(\W;\widetilde{I}) \oplus \End_\W(\widetilde{I})$$
    combined with Corollary \ref{cor:derivations W to I tilde} and Lemma \ref{lem:endomorphisms I tilde}.
\end{proof}


\begin{thebibliography}{BFOGV24}

\bibitem[Aiz12]{Aizawa}
N.~Aizawa, \emph{{Some representations of planar {G}alilean conformal
  algebra}}, {7th Mathematical Physics Meeting}: {Summer School and Conference
  on Modern Mathematical Physics}, 2012,
  arXiv:\texttt{\href{https://arxiv.org/abs/1212.6288}{1212.6288}}.

\bibitem[Bav16]{BavulaWitt}
V. Bavula, \emph{The groups of automorphisms of the {W}itt {$W_n$} and
  {V}irasoro {L}ie algebras}, Czechoslovak Math. J. \textbf{66 (141)} (2016), 1129--1141.

\bibitem[BBB{\etalchar{+}}24]{BBBM}
N. Banerjee, A. Bhattacharjee, S. Biswas, A. Mitra, and
  D. Mukherjee, \emph{{W(0,b) algebra and the dual theory of 3D
  asymptotically flat higher spin gravity}}, Phys. Rev. D \textbf{109} (2024), Paper No. 066002.

\bibitem[BCG20]{BatlleCarlesCampelloGomis}
C. Batlle, V. Campello, and J. Gomis, \emph{A canonical
  realization of the {W}eyl {BMS} symmetry}, Phys. Lett. B \textbf{811} (2020),
  Paper No. 135920, 5 pp.

\bibitem[BFOGV24]{BatlleFigueroaGomisVishwa}
C. Batlle, J.~M. Figueroa-O'Farrill, J. Gomis, and G.~S.
  Vishwa, \emph{BMS-like algebras: canonical realisations and BRST quantisation},
  2024, arXiv:\texttt{\href{https://arxiv.org/abs/2411.14866}{2411.14866}}.

\bibitem[BG09]{BagchiGopakumar}
A. Bagchi and R. Gopakumar, \emph{Galilean conformal algebras and
  {A}d{S}/{CFT}}, J. High Energy Phys. (2009), Paper No. 037, 22 pp.

\bibitem[BH86]{BrownHenneaux}
J.~D. Brown and M. Henneaux, \emph{Central charges in the canonical
  realization of asymptotic symmetries: an example from three-dimensional
  gravity}, Comm. Math. Phys. \textbf{104} (1986), 207--226.

\bibitem[Buz24]{Buzaglo}
L. Buzaglo, \emph{Derivations, extensions, and rigidity of subalgebras of
  the {W}itt algebra}, J. Algebra \textbf{647} (2024), 230--276.

\bibitem[BvdBM62]{BondivdBurgMetzner}
H.~Bondi, M.~G.~J. van~der Burg, and A.~W.~K. Metzner, \emph{Gravitational
  waves in general relativity. {VII}. {W}aves from axi-symmetric isolated
  systems}, Proc. Roy. Soc. London Ser. A \textbf{269} (1962), 21--52.

\bibitem[FOV24]{FigueroaVishwa}
J. Figueroa-O'Farrill and G.~S. Vishwa, \emph{The {BRST} quantisation
  of chiral {BMS}-like field theories}, 2024,
  arXiv:\texttt{\href{https://arxiv.org/abs/2407.12778}{2407.12778}}.

\bibitem[FPSSJ19]{FarahmandSafariSheikhJabbari}
A.~Farahmand~Parsa, H.~R. Safari, and M.~M. Sheikh-Jabbari, \emph{{On Rigidity
  of 3d Asymptotic Symmetry Algebras}}, JHEP \textbf{03} (2019), 143.

\bibitem[Fuc86]{Fuchs}
D.~B. Fuchs, \emph{Cohomology of infinite-dimensional {L}ie algebras},
  Contemporary Soviet Mathematics, Consultants Bureau, New York, 1986,
  Translated from the Russian by A. B. Sosinski\u{\i}.

\bibitem[FW21]{FeldvossWagemann}
J. Feldvoss and F. Wagemann, \emph{On {L}eibniz cohomology}, J.
  Algebra \textbf{569} (2021), 276--317.

\bibitem[GJP11]{GaoJiangPei}
S. Gao, C. Jiang, and Y. Pei, \emph{Low-dimensional cohomology
  groups of the {L}ie algebras {$W(a,b)$}}, Comm. Algebra \textbf{39} (2011), 397--423.

\bibitem[GLP16]{GaoLiuPei}
S. Gao, D. Liu, and Y. Pei, \emph{Structure of the planar {G}alilean
  conformal algebra}, Rep. Math. Phys. \textbf{78} (2016), 107--122.

\bibitem[GPSJ{\etalchar{+}}20]{GPSJTZ}
D. Grumiller, A. P\'erez, M.~M. Sheikh-Jabbari, R. Troncoso, and
  C. Zwikel, \emph{{Spacetime structure near generic horizons and soft
  hair}}, Phys. Rev. Lett. \textbf{124} (2020), 041601.

\bibitem[HPL08]{HuPeiLiu}
N. Hu, Y. Pei, and D. Liu, \emph{A cohomological characterization of
  {L}eibniz central extensions of {L}ie algebras}, Proc. Amer. Math. Soc.
  \textbf{136} (2008), 437--447.

\bibitem[IK90]{IkedaKawamoto}
T. Ikeda and N. Kawamoto, \emph{On the derivations of generalized
  {W}itt algebras over a field of characteristic zero}, Hiroshima Math. J.
  \textbf{20} (1990), 47--55.

\bibitem[KS85]{KaplanskySantharoubane}
I. Kaplansky and L.~J. Santharoubane, \emph{Harish-{C}handra modules over
  the {V}irasoro algebra}, Infinite-dimensional groups with applications
  ({B}erkeley, {C}alif., 1984), Math. Sci. Res. Inst. Publ., vol.~4, Springer,
  New York, 1985, pp.~217--231.

\bibitem[Lod98]{LodayBook}
J.-L. Loday, \emph{Cyclic homology}, second ed., Grundlehren der
  mathematischen Wissenschaften [Fundamental Principles of Mathematical
  Sciences], vol. 301, Springer-Verlag, Berlin, 1998, Appendix E by Mar\'{\i}a
  O. Ronco, Chapter 13 by the author in collaboration with T. Pirashvili.

\bibitem[LP93]{LodayPirashvili}
J.-L. Loday and T. Pirashvili, \emph{Universal enveloping algebras
  of {L}eibniz algebras and (co)homology}, Math. Ann. \textbf{296} (1993), 139--158.

\bibitem[Mat92]{Mathieu}
O. Mathieu, \emph{Classification of {H}arish-{C}handra modules over the
  {V}irasoro {L}ie algebra}, Invent. Math. \textbf{107} (1992),
  225--234.

\bibitem[MP91]{MartinPiard}
C. Martin and A. Piard, \emph{Indecomposable modules over the
  {V}irasoro {L}ie algebra and a conjecture of {V}. {K}ac}, Comm. Math. Phys.
  \textbf{137} (1991), 109--132.

\bibitem[MP92]{MartinPiard2}
C. Martin and A. Piard, \emph{Classification of the indecomposable bounded admissible modules
  over the {V}irasoro {L}ie algebra with weightspaces of dimension not
  exceeding two}, Comm. Math. Phys. \textbf{150} (1992), 465--493.

\bibitem[OR96]{OvsienkoRoger}
V.~Yu. Ovsienko and C.~Roger, \emph{Extensions of the {V}irasoro group and the
  {V}irasoro algebra by means of modules of tensor densities on {$S^1$}},
  Funktsional. Anal. i Prilozhen. \textbf{30} (1996), 86--88.

\bibitem[Sch08]{Schottenloher}
M.~Schottenloher, \emph{A mathematical introduction to conformal field theory},
  second ed., Lecture Notes in Physics, vol. 759, Springer-Verlag, Berlin,
  2008.

\bibitem[Sch16]{Schlichenmaier:2012ika}
M. Schlichenmaier, \emph{{Krichever-Novikov type algebras. An
  Introduction}}, Proc. Symp. Pure Math. \textbf{92} (2016), 181--220.

\bibitem[S{\v S}17]{SierraSpenko}
S.~J. Sierra and \v{S}. {\v S}penko, \emph{Generalised {W}itt algebras
  and idealizers}, J. Algebra \textbf{483} (2017), 415--428.

\bibitem[Su94]{Su}
Y.~C. Su, \emph{Harish-{C}handra modules of the intermediate series over the
  high rank {V}irasoro algebras and high rank super-{V}irasoro algebras}, J.
  Math. Phys. \textbf{35} (1994), 2013--2023.

\bibitem[Wei94]{Weibel}
C.~A. Weibel, \emph{An introduction to homological algebra}, Cambridge
  Studies in Advanced Mathematics, vol.~38, Cambridge University Press,
  Cambridge, 1994.

\end{thebibliography}
\end{document}